\documentclass[10pt]{article}
\usepackage[margin=1in]{geometry}
\usepackage{xparse}

\usepackage{amsmath, amsthm, amssymb}
\usepackage{nccmath}
\usepackage{booktabs}
\usepackage{empheq}
\usepackage{mdframed}
\usepackage{pifont}
\usepackage{enumitem}
\usepackage{graphicx} 
\usepackage{wrapfig}
\usepackage{subfigure}
\usepackage{empheq}

\usepackage{thmtools}
\usepackage{thm-restate}

\usepackage{nicefrac}
\usepackage{booktabs}
\usepackage{multirow}
\usepackage{rotating}
\usepackage{bm}

\usepackage{tikz}
\usetikzlibrary{shapes}
\usetikzlibrary{decorations.markings}
\usetikzlibrary{plotmarks}
\usetikzlibrary{patterns}
\usepackage{ctable}

\usepackage{color, colortbl}
\definecolor{mydarkblue}{rgb}{0,0.08,0.45}
\usepackage[colorlinks=true,
    linkcolor=mydarkblue,
    citecolor=mydarkblue,
    filecolor=mydarkblue,
    urlcolor=mydarkblue,
    pdfview=FitH]{hyperref}
    
\usepackage{empheq}

\definecolor{myteal}{RGB}{27,158,119}
\definecolor{myorange}{RGB}{217,95,2}
\definecolor{myred}{RGB}{231,41,138}
\definecolor{mypurple}{RGB}{152,78,163}
\definecolor{myblue}{RGB}{55,126,184}
\definecolor{mygreen}{RGB}{0,100,0}

\newtheorem{definition}{Definition}[section]

\newtheorem{proposition}{Proposition}[section]
\newtheorem{lemma}{Lemma}[section]
\newtheorem{theorem}{Theorem}[section]
\newtheorem{remark}{Remark}[section]
\newtheorem{corollary}{Corollary}[section]

\newcommand*\mybluebox[1]{\colorbox{myblue}{\hspace{1em}#1\hspace{1em}}}

\definecolor{myblue}{HTML}{D2E4FC}
\definecolor{Gray}{gray}{0.92}

\newtheorem{assumption}{Assumption}

\graphicspath{{./figures/}}

\usepackage[sort]{natbib}

\def\xx{{\boldsymbol x}}

\def\YY{{\boldsymbol Y}}

\def\aa{{\boldsymbol a}}
\def\bb{{\boldsymbol b}}

\def\WW{{\boldsymbol W}}

\def\II{{\boldsymbol I}}
\def\yy{{\boldsymbol y}}
\def\vv{{\boldsymbol v}}
\def\uu{{\boldsymbol u}}
\def\ww{{\boldsymbol w}}
\def\zz{{\boldsymbol z}}

\def\BB{{\boldsymbol B}}
\def\AA{{\boldsymbol A}}
\def\CC{{\boldsymbol C}}
\def\FF{{\boldsymbol F}}
\def\MM{{\boldsymbol M}}
\def\DD{{\boldsymbol D}}
\def\PP{{\boldsymbol P}}

\def\HH{{\boldsymbol H}}

\def\SSigma{{\boldsymbol \Sigma}}

\def\eeta{{\boldsymbol \eta}}

\def\g{{g}}

\def\dif{\mathop{}\!\mathrm{d}}

\def\MP{\mu_{\mathrm{MP}}}
\def\RR{{\mathbb R}}
\def\EE{{\mathbb E}\,}

\def\defas{\stackrel{\text{def}}{=}}

\DeclareMathOperator*{\argmax}{{arg\,max}}
\DeclareMathOperator*{\argmin}{{arg\,min}}

\DeclareMathOperator*{\ospan}{span}
\def\concentration{\! \! \underset{d \rightarrow \infty}{\mathcal{E}} \!  [\|\nabla f(\xx_{k})\|^2]\,}

\DeclareDocumentCommand{\Prto} {o} {
  \IfNoValueTF {#1}
  {\overset{\Pr}{\longrightarrow}}
  { \xrightarrow[ #1 \to \infty]{\Pr }}
}

\DeclareDocumentCommand{\Asto} {o} {
  \IfNoValueTF {#1}
  {\overset{\text{\rm a.s.}}{\longrightarrow}}
  { \xrightarrow[ #1 \to \infty]{\text{\rm a.s.} }}
}

\makeatletter
\newcommand{\vast}{\bBigg@{4}}
\newcommand{\Vast}{\bBigg@{5}}
\makeatother

\begin{document}

\title{Halting Time is Predictable for Large Models:\\
A Universality Property and Average-case Analysis
}
\date{}
\author{Courtney Paquette\thanks{Google Research, Brain Team} \footnotemark[2] \and Bart van Merri\"enboer\footnotemark[1] \and Elliot Paquette\thanks{Department of Mathematics and Statistics, McGill University, Montreal, QC, Canada, H3A 0B9; CP is a CIFAR AI chair; \url{https://cypaquette.github.io/}. Research by EP was supported by a Discovery Grant from the
Natural Science and Engineering Research Council (NSERC) of Canada.; \url{https://elliotpaquette.github.io/}. } \and Fabian Pedregosa\footnotemark[1]
}

\maketitle
\begin{abstract}
Average-case analysis computes the complexity of an algorithm averaged over all possible inputs.
Compared to worst-case analysis, it is more representative of the typical behavior of an algorithm, but remains largely unexplored in optimization. One difficulty is that the analysis can depend on the probability distribution of the inputs to the model. However, we show that this is not the case for a class of large-scale problems trained with first-order methods including random least squares and one-hidden layer neural networks with random weights. 
In fact, the halting time exhibits a \textit{universality property}: it is independent of the probability distribution. With this barrier for average-case analysis removed, we provide the first explicit average-case convergence rates showing a tighter complexity not captured by traditional worst-case analysis.
Finally, numerical simulations suggest this universality property holds for a more general class of algorithms and problems.
\end{abstract}



\section{Introduction}
\begin{wrapfigure}[18]{r}{0.50\textwidth}
    \centering
    \vspace{-2cm}
     \includegraphics[width = \linewidth]{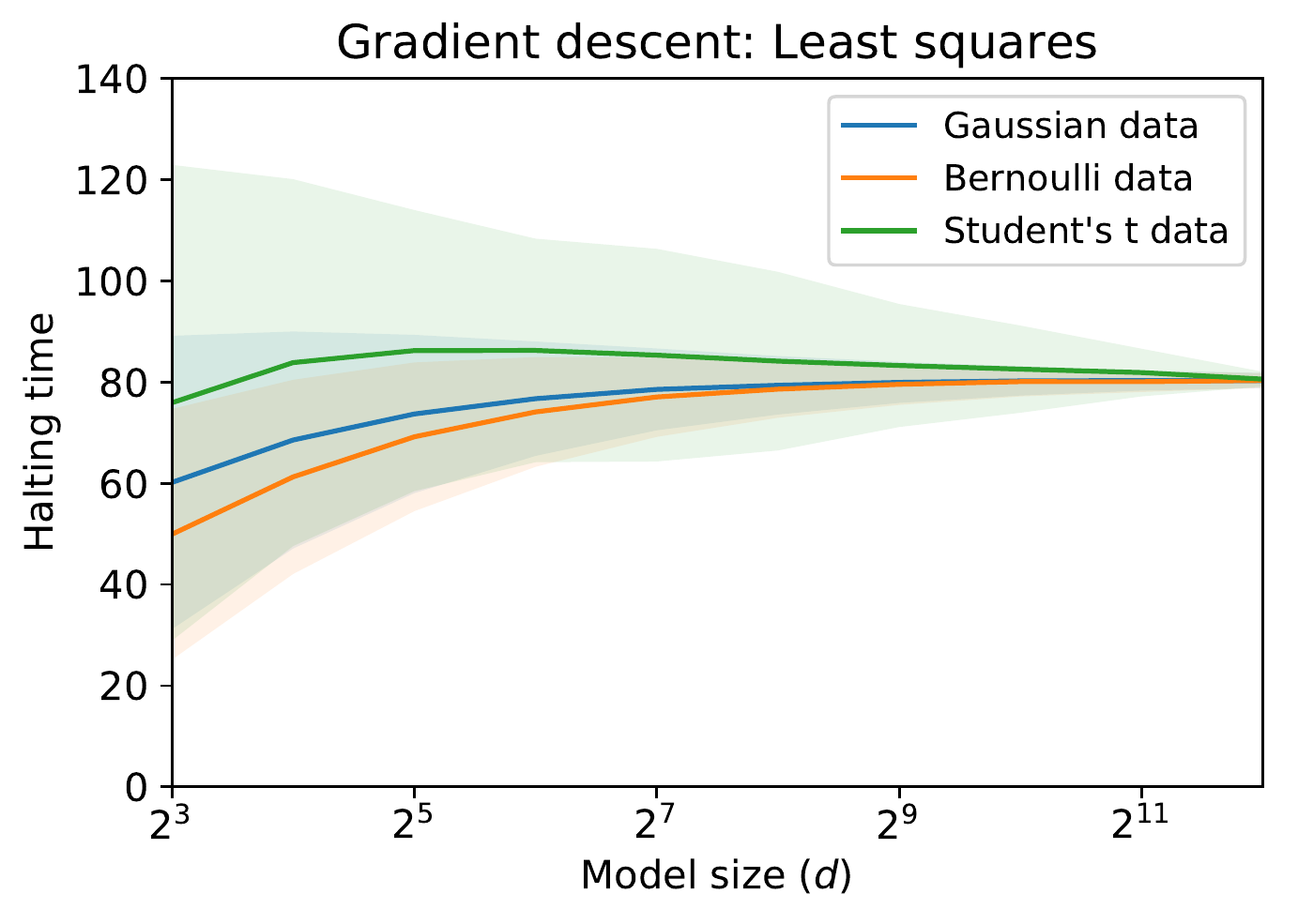} \vspace{-0.75cm}
     \caption{As the model grows ($x$-axis), the standard deviation (shaded region) in the halting time of gradient descent on random least squares vanishes and the halting time becomes {\bfseries predictable}. Note also a \textbf{universality} phenomenon, that is, the halting time limit is the same for problems generated from different distributions. (See Sec.~\ref{sec:numerical_simulations} for a description of simulations.)}
    \label{fig:gd-ls}
    \end{wrapfigure}
Traditional worst-case analysis of optimization algorithms provides complexity bounds for any input, no matter how unlikely \citep{nemirovski1995information, nesterov2004introductory}. It gives convergence guarantees, but the bounds are not always representative of the typical runtime of an algorithm.
In contrast, average-case analysis gives sharper runtime estimates when some or all of its inputs are random. This is often paired with  concentration bounds that quantify the spread of those estimates.
In this way, it is more representative of the typical behavior.

Yet, average-case analysis is rarely used in optimization because the complexity of algorithms is assumed to depend on the specific probability distribution of the inputs. Surprisingly, simulations reveal this is not the case for large-scale problems (see Figure \ref{fig:gd-ls}).

We show that almost all instances of high-dimensional data are indistinguishable to first-order algorithms. Particularly, the \emph{halting time}, i.e.\ the number of iterations to reach a given accuracy, for any first-order method converges to a deterministic value which is independent of the input distribution (see Figure~\ref{fig:gd-ls}). Since the halting time is deterministic, the empirical complexity coincides almost surely with the average-case rates.

\renewcommand{\arraystretch}{2}
\ctable[notespar, 
    caption = {\textbf{Comparison of convergence guarantees for non-strongly convex objectives} in terms of asymptotic behavior of $\|\nabla f(\xx_k)\|^2$ as problem size and iteration are large in the isotropic features model. The average-case guarantees are strictly faster than the traditional worst-case and adversarial rates. Furthermore, the traditional worst-case complexity bounds depend on the distance to the optimum which under our constant signal-to-noise model, grows as the problem size, or dimension, $d$, increases. The `without noise' setting refers to the case when the targets $\bb$ equal $\AA \widetilde{\xx}$ with $\widetilde{\xx}$ the signal and the `noisy' setting when the targets $\bb$ follow a generative model but are corrupted by noise, that is, $\bb = \AA \widetilde{\xx} + \eeta$, where $\eeta$ is a noise vector. The rates are stated in terms of an absolute constant $C$, the amount of signal $R$ and noise $\widetilde{R}$, the ratio of number of features to samples $d/n \to r \in (0,\infty)$, and the maximum $\lambda^+$ and minimum $\lambda^-$ eigenvalues. Denote $\|J_1^2(x)\|_{\infty}$ the maximum value of the squared Bessel function of the first kind $(J_1(x))$ over $[0, \infty)$. See Section~\ref{sec: average_case} and \ref{sec: avg_derivations} for derivations and definitions of terms such as non-strongly convex.},  
    captionskip=2ex,
    label={tab:comparison_worst_avg_cvx},
    pos =ht!
    ]{clll}{\tnote[1]{In the noisy setting, we lower bounded $\|\xx_0-\xx^*\|^2$ by $d$ (see Lemma~\ref{lem:growth_dist_optimum}) to the worst-case complexity bound provided in \citet[Section 4.1.3]{taylor2017smooth}. } \tnote[2]{\cite{nesterov2004introductory,Beck2009Fast}}
    \tnote[3]{\cite{nesterov2012how}} \tnote[4]{Adversarial model maximizes the norm of the gradient subject to a fixed condition number (see Section~\ref{sec: average_case}).} \tnote[5]{When noise is added, the convergence rates are dominated by the term with $\widetilde{R}$ in \eqref{eq: something_1_main}.}}{
\toprule
\textbf{Method} &  & \begin{minipage}{0.3\textwidth} \begin{center} \textbf{Non-strongly cvx\\
w/o noise} \end{center} \end{minipage} & \begin{minipage}{0.3\textwidth} \begin{center} \textbf{Non-strongly cvx\\ w/ noise}\tmark[5] \end{center} \end{minipage}\\
\midrule
\multirow{3}{*}{\begin{minipage}{0.1\textwidth} \begin{center} Gradient descent (GD) \end{center} \end{minipage}} &  Worst\tmark[1]  & $\textcolor{teal}{\mfrac{1}{(k+1)^2}} \cdot R^2 (\lambda^+)^2$ & $\textcolor{teal}{\mfrac{\textcolor{purple}{d}}{(k+1)^2}} \cdot  \widetilde{R}^2 (\lambda^+)^2 C$ \\
\cmidrule(r){2-4}
& Adversarial\tmark[4] & $\textcolor{teal}{\mfrac{1}{(k+1)^2}} \cdot \mfrac{R^2 (\lambda^+)^2}{e^{2}} $ & 
$\textcolor{teal}{\mfrac{1}{k}} \cdot \mfrac{\widetilde{R}^2 \lambda^+}{2}$\\
\cmidrule(r){2-4}
         & Average  & $\textcolor{teal}{\mfrac{1}{k^{5/2}}} \cdot \mfrac{R^2 (\lambda^+)^2 \Gamma(5/2)}{2^{3/2} \pi}$ & 
     $\textcolor{teal}{\mfrac{1}{k^{3/2}}} \cdot \mfrac{\widetilde{R}^2 \lambda^+ \Gamma (3/2 )}{2^{1/2} \pi}$\\
        \midrule
  \multirow{3}{*}{\begin{minipage}{0.16\textwidth} \begin{center} Nesterov\\ accelerated method \tmark[2] \end{center} \end{minipage}}      
  & Worst \tmark[3] & $\textcolor{teal}{\mfrac{1}{k(k+2)^2}} \cdot 8 R^2 (\lambda^+)^2$ & $\textcolor{teal}{\mfrac{\textcolor{purple}{d}}{k(k+2)^2}} \cdot 8\widetilde{R}^2 (\lambda^+)^2 C$ \\ 
  \cmidrule(r){2-4}
  & Adversarial & $\textcolor{teal}{\mfrac{1}{k^{7/2}}} \cdot \mfrac{8e^{-1/2}}{\sqrt{2} \pi} R^2 (\lambda^+)^2$ & 
  $\textcolor{teal}{\mfrac{1}{k^{2}}} \cdot \|J_1^2(x)\|_{\infty} \widetilde{R}^2 \lambda^+ $ \\
  \cmidrule(r){2-4}
          & Average & $\textcolor{teal}{\mfrac{1}{k^4}} \cdot \mfrac{8 R^2 (\lambda^+)^2}{\pi^2}$ & 
    $\textcolor{teal}{\mfrac{\log(k)}{k^3}} \cdot \mfrac{4\widetilde{R}^2 \lambda^+}{\pi^2 }$\\ 
\bottomrule}

\renewcommand{\arraystretch}{2.5}
    \ctable[notespar, caption = {\textbf{Comparison of convergence guarantees for strongly convex objectives} in terms of asymptotic behavior of $\|\nabla f(\xx_k)\|^2$ as problem size and iteration are large in the isotropic features model. Average-case matches the worst-case asymptotic guarantees  multiplied by an additional \textit{polynomial correction term} (\textcolor{teal}{green}). This polynomial term has little effect on the complexity compared to the linear rate. However as the matrix $\HH = \tfrac{1}{n} \AA^T \AA$ becomes ill-conditioned $(r\to1)$, the polynomial correction starts to dominate the average-case complexity. Indeed this shows that the support of the spectrum does not fully determine the rate. \textit{Many} eigenvalues contribute meaningfully to the average-rate.
    See Section~\ref{sec: avg_derivations} for derivations and Table~\ref{tab:comparison_worst_avg_cvx} for definition of terms in the rates.},
    label= {tab:comparison_worst_avg_str_cvx},
    captionskip=2ex,
    pos = ht!
    ]{cll}{\tnote[1]{\citet[Section 4.1.3]{taylor2017smooth}}}{
\toprule
         \textbf{Method} & &  \textbf{Strongly cvx w/ noise}   \\
  \midrule
   \multirow{2}{*}{\begin{minipage}{0.15\textwidth} \begin{center} Gradient descent (GD) \end{center} \end{minipage}} & Worst\tmark[1] &  $\textcolor{purple}{\big (1- \frac{\lambda^-}{\lambda^+} \big )^{2k}} (\lambda^+)^2  $\\
   \cmidrule(r){2-3}
   & Average & $\textcolor{purple}{\big (1- \frac{\lambda^-}{\lambda^+} \big )^{2k}} \textcolor{teal}{\mfrac{1}{k^{3/2}}} \big [ R^2 \lambda^- + \widetilde{R}^2 r \big ] \cdot C  $\\
         \midrule
         \multirow{2}{*}{\begin{minipage}{0.15\textwidth} \begin{center} Polyak\\ \citep{Polyak1962Some} \end{center} \end{minipage}}
          & Worst & $ \textcolor{purple}{ \big ( 1 - \frac{2 \sqrt{\lambda^-}}{\sqrt{\lambda^+} + \sqrt{\lambda^-}}\big )^{2k}} \cdot C$ \\ 
          \cmidrule(r){2-3}
          & Average & $ \textcolor{purple}{ \big ( 1 - \frac{2 \sqrt{\lambda^-}}{\sqrt{\lambda^+} + \sqrt{\lambda^-}}\big )^{2k}} \big [\frac{(\lambda^+-\lambda^-)}{2}R^2 +  \widetilde{R}^2r \big ] \cdot C$ \\ 
          \midrule
          \multirow{2}{*}{\begin{minipage}{0.2\textwidth} \begin{center} Nesterov accelerated method\\ \citep{nesterov2004introductory} \end{center} \end{minipage}} & Worst & $\textcolor{purple}{ \big ( 1 - \frac{2 \sqrt{\lambda^-}}{\sqrt{\lambda^+} + \sqrt{\lambda^-}}\big )^{k} \big (1- \frac{\lambda^-}{\lambda^+} \big )^k} \cdot C$ \\ 
          \cmidrule(r){2-3}
          & Average & $ \textcolor{purple}{ \big ( 1 - \frac{2 \sqrt{\lambda^-}}{\sqrt{\lambda^+} + \sqrt{\lambda^-}}\big )^{k} \big (1- \frac{\lambda^-}{\lambda^+} \big )^k} \big [\textcolor{teal}{\mfrac{1}{k^{1/2}}} \cdot R^2 \lambda^- + \textcolor{teal}{\mfrac{1}{k^{1/2}}} \cdot \widetilde{R}^2 r \big ] \cdot C$ \\
         \bottomrule}
\paragraph{Notation.} We write vectors in lowercase boldface ($\xx$) and matrices in uppercase boldface ($\HH$). The norm $\|\xx\|_2^2 = \xx^T \xx$ gives the usual Euclidean $2$-norm and $\|\HH\|_{\text{op}} = \text{maximum singular value of $\HH$}$ is the usual operator-2 norm. Given a matrix $\HH \in \RR^{d \times d}$, the largest eigenvalue of $\HH$ is $\lambda_{\HH}^+$ and its smallest eigenvalue is $\lambda_{\HH}^-$. A sequence of random variables $\{y_d\}_{d =0}^\infty$ converges in probability to $y$, indicated by $y_d \Prto[d] y$, if for any $\varepsilon > 0$, $\displaystyle \lim_{d \to \infty} \Pr(|y_d-y| > \varepsilon) = 0$. In other words, the probability that $y_d$ is far from $y$ goes to $0$ as $d$ increases. Probability measures are denoted by $\mu$ and their densities by $\dif\mu$. We say a sequence of random measures $\mu_d$ converges to $\mu$ weakly in probability if for any bounded continuous function $f$, we have $\int f \dif\mu_d \to \int f \dif\mu$ in probability. 
 
All stochastic quantities defined hereafter live on a probability space denoted by $(\Pr, \Omega, \mathcal{F})$ with probability measure $\Pr$ and the $\sigma$-algebra $\mathcal{F}$ containing subsets of $\Omega$. A random variable (vector) is a measurable map from $\Omega$ to $\RR$ $(\RR^d)$ respectively. Let $X : (\Omega, \mathcal{F}) \mapsto (\RR, \mathcal{B})$ be a random variable mapping into the Borel $\sigma$-algebra $\mathcal{B}$ and the set $B \in \mathcal{B}$. We use the standard shorthand for the event $\{X \in B\} = \{\omega : X(\omega) \in B\}$.\\

\subsection{Main results}
In this paper, we analyze the halting time and develop the first explicit average-case analysis for first-order methods on quadratic objectives.
Quadratic objective functions are rich enough to reproduce the dynamics that arise in more complex models, yet simple enough to be understood in closed form. 
Quadratic models are receiving renewed interest in the machine learning community as recent advances have shown that over-parameterized models, including neural networks, have training dynamics similar to those of quadratic problems~\citep{jacot2018neural, novak2018bayesian, arora2019exact, chizat2019lazy}.

The precise form of the quadratic problem we consider is
\begin{equation} \label{eq:LS_main}
\vspace{0.5em}\argmin_{\xx \in \RR^d} \Big \{ f(\xx) \defas \frac{1}{2n} \|\AA \xx-\bb\|^2 \Big \}, \quad \text{with } \bb \defas \AA \widetilde{\xx} + \eeta\,,
\end{equation}
where $\AA \in \RR^{n \times d}$ is the data matrix, $\widetilde{\xx} \in \RR^d$ is the signal vector \footnote{The signal $\widetilde{\xx}$ is not the same as the vector for which the iterates of the algorithm are converging to as $k \to \infty$.}, and $\eeta \in \RR^n$ is a source of noise. All of these inputs will possibly be random and the target $\bb = \AA \widetilde{\xx} + \eeta$ is produced by a generative model corrupted by noise. We refer to the noiseless (without noise) setting when $\bb = \AA \widetilde{\xx}$ and the noisy setting as $\bb = \AA \widetilde{\xx} + \eeta$. 

We work in the following setting:  Both the number of features $(d)$ and data dimension $(n)$ grow to infinity while $d/n$ tends to a fixed $r \in (0,\infty)$. 
We use $\widetilde{R}^2=\frac{1}{d}\mathbb{E}\left[\|\eeta\|^2\right]$ to denote the magnitude of the noise. For intuition, we implicitly define $R^2 \approx \frac{1}{d}\|\bb\|^2 - \widetilde{R}^2$ to measure the strength of the signal{\footnote{The definition of $\widetilde{R}^2$ in Assumption~\ref{assumption: Vector} does not imply that $R^2 \approx \frac{1}{d}\|\bb\|^2 - \widetilde{R}^2$. However the precise definition of $\widetilde{R}$ and this intuitive one yield similar magnitudes and both are generated from similar quantities. }}; we make the definition of $\widetilde{R}^2$ precise in Assumption \ref{assumption: Vector} of Section~\ref{sec: problem_setting}, one of two assumptions fundamental to this work. Throughout, the signal-to-noise ratio in $[0,\infty]$ is held constant as the problem size grows.
Moreover, we assume that the data matrix $\AA$ is \emph{independent} of both the signal, $\widetilde{\xx}$, and noise $\eeta.$  Note this, together with the generative model, allows for some amount of dependence between $\AA$ and the target $\bb.$  We will also assume that $\HH \defas \frac{1}{n} \AA^T \AA$ has a well-defined \textit{limiting spectral density}, denoted by $\dif\mu$, as $n,d \to \infty$ (see Assumption \ref{assumption: spectral_density} of Section~\ref{sec: problem_setting}).

Our first contribution is a framework to analyze the average-case complexity of gradient-based methods in the described setting.  Our framework highlights how the algorithm, signal and noise levels interact with each other to produce different average-case convergence guarantees. The culmination of this framework is the average-case convergence rates for first-order methods (see Tables~\ref{tab:comparison_worst_avg_cvx} and \ref{tab:comparison_worst_avg_str_cvx}). 

\begin{wrapfigure}[15]{r}{0.47\textwidth}
 \centering
    \vspace{-0.55cm}
     \includegraphics[width = 0.85\linewidth]{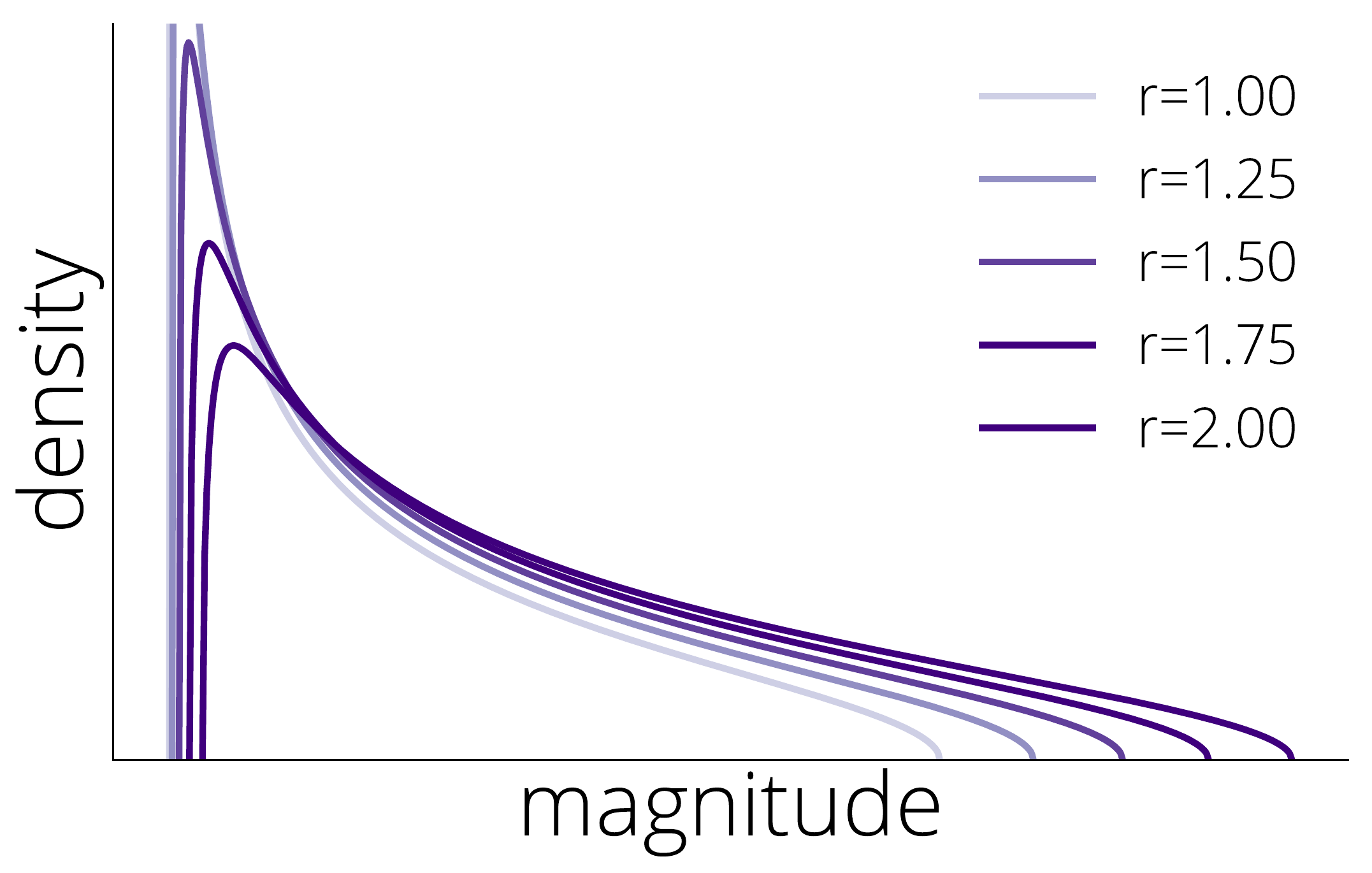}
      \vspace{-0.25cm}
     \caption{The spectrum of matrices $\tfrac{1}{n} \AA^T \AA$ under the isotropic features model converges as $n,d \to \infty$ to the \emph{Mar\v{c}enko-Pastur} distribution, shown here for different values of $r = d/n$.}
    \label{fig:MP}
\end{wrapfigure}

Our framework is broad enough to facilitate multiple perspectives on average-case analysis. Our motivating and central application is the \emph{fully-average-case}, in which we assume that all inputs are random.  The quintessential random data model is \emph{isotropic features}. This supposes the entries of $\AA$ are i.i.d.\ random variables with zero mean, equal variance, and bounded fourth moments, that is, $\EE[A_{ij}] = 0, \EE[A_{ij}^2] = \sigma^2, \EE[A_{ij}^4] < \infty$ for all $i, j$. In a celebrated theorem of \cite{marvcenko1967distribution}, the spectrum of $\HH = \frac{1}{n}\AA^T \AA$ converges to a compactly supported measure as the problem size grows \emph{without any further assumptions on the distribution of the entries of $\AA$}.  This limiting spectral distribution is known as the Mar\v{c}enko-Pastur law:
\begin{equation} \label{eq:MP}
\begin{gathered} \dif\MP(\lambda) \defas \delta_0(\lambda) \max\{1-\tfrac{1}{r}, 0\} + \frac{\sqrt{(\lambda-\lambda^-)(\lambda^+-\lambda)}}{2 \pi \lambda \sigma^2 r} 1_{[\lambda^-, \lambda^+]}\,,\\
\text{where} \qquad \lambda^- \defas \sigma^2(1 - \sqrt{r})^2 \quad \text{and} \quad \lambda^+ \defas \sigma^2(1+ \sqrt{r})^2\,.
\end{gathered} 
\end{equation}

However, our framework is built to be vastly more general.  To start, the framework covers a fully-average-case analysis with other data models, such as the one-hidden layer network with random weights and the correlated features model (see Section \ref{sec: data_generate}).  
More to the point, this framework also allows for a type of semi-average-case analysis, in which only $\bb$ is taken to be random.  When we do this and then choose $\AA$ in such a way as to maximize the halting time, we call this the \emph{adversarial average-case}.  See Section \ref{sec: average_case} for further details and motivations. 
 
We now discuss the contents of this framework in detail,
which is to say we survey how Assumptions \ref{assumption: Vector} and \ref{assumption: spectral_density} combine to show the halting time is concentrated and deterministic.
The first step is to express the conditional expectation of the gradient at the $k$-th iterate as a sum of expected traces of polynomials in the matrix $\HH = \frac{\AA^T \AA}{n}$ (c.f.\ Proposition~\ref{proposition:conditional}):
\begin{equation} \label{eq:conditional_main_result}
    \EE[\|\nabla f(\xx_k)\|^2 \, | \, \HH] = \tfrac{R^2}{d} \text{tr} \big ( \HH^2 P_k^2(\HH) \big )  + \tfrac{\widetilde{R}^2}{n} \text{tr} \big ( \HH P_k^2(\HH) \big ).
    \end{equation}
The polynomial $P_k$, known as the \textit{residual polynomial}, is a $k$-th degree polynomial associated with the gradient-based algorithm. This tool of associating each algorithm with polynomials is a classic technique in numerical iterative methods for solving linear systems \citep{Flanders1950Numerical,golub1961chebyshev,fischer1996polynomial,rutishauser1959refined}. Such polynomials are used to prove convergence of some of the most celebrated algorithms like the conjugate gradient method \citep{Hestenes&Stiefel:1952}. Explicit expressions of the residual polynomials for Nesterov's accelerated methods \citep{nesterov2004introductory, Beck2009Fast}, both convex and strongly convex, as well as, gradient descent and Polyak's momemtum (a.k.a Heavy-ball) \citep{Polyak1962Some} are derived in Section~\ref{sec: poly}. These polynomials may be of independent interest. 

The result in \eqref{eq:conditional_main_result} gives an \textit{exact expression} for the expected gradient depending only on traces of powers of $\HH$, which in turn can be expressed in terms of its \emph{eigenvalues}.  Our second main assumption (Assumption \ref{assumption: spectral_density}) then ensures that these traces converge to integrals against the spectral density $\dif\mu.$ In summary, the squared gradient norm concentrates to a deterministic quantity, \footnote{
In many situations this deterministic quantity
$\concentration$
is in fact the limiting expectation of the squared-norm of the gradient.  However, under the assumptions that we are using, this does not immediately follow.  It is however always the limit of the median of the squared-norm of the gradient.}\footnote{Technically, there is no need to assume the measure $\mu$ has a density -- the theorem holds just as well for any limiting spectral measure $\mu$.  In fact, a version of this theorem can be formulated at finite $n$ just as well, thus dispensing entirely with Assumption \ref{assumption: spectral_density} -- c.f.\ Proposition \ref{proposition:conditional}.} $\concentration$:
\begin{theorem}[Concentration of the gradient] \label{thm: concentration_main}
Under Assumptions~\ref{assumption: Vector} and~\ref{assumption: spectral_density} the norm of the gradient concentrates around a deterministic value:
\begin{equation} \label{eq: something_1_main} \vspace{0.25cm}
\hspace{-0.28cm}  \|\nabla f(\xx_k)\|^2 \Prto[d]  \textcolor{teal}{\overbrace{R^2}^{\text{signal}}}     \int  { \underbrace{\lambda^2 P_k^2(\lambda)}_{\text{algorithm}}} \textcolor{mypurple}{\overbrace{\dif\mu}^{\text{model}} }   +   \textcolor{purple}{\overbrace{ \widetilde{R}^2} ^{\text{noise}} }  r   \int  { \underbrace{\lambda P_k^2(\lambda)}_{\text{algorithm}}}  \textcolor{mypurple}{\overbrace{ \dif\mu}^{\text{model}} }  \defas \concentration.
\end{equation}
\end{theorem}
\noindent 
Notably, the deterministic value for which the gradient concentrates around depends only on $\HH$ through its eigenvalues.

The concentration of the norm of the gradient above yields a candidate for the limiting value of the halting time, or the first time the gradient $\|\nabla f(\xx_k)\|^2$ falls below some predefined $\varepsilon$. We define this candidate for the halting time $\tau_{\varepsilon}$ from $ \concentration$ and we denote the halting time $T_{\varepsilon}$, by 
\begin{align}
 \tau_{\varepsilon} \defas \inf \, \{ k > 0  : \concentration \le \varepsilon\} \quad \text{and} \quad T_{\varepsilon} \defas \inf \, \{ k > 0  :  \|\nabla f(\xx_k)\|^2 \le \varepsilon\}\,.
\end{align}
We note that the deterministic value $\tau_{\varepsilon}$ is, by definition, the average complexity of the first-order algorithm. This leads to our second main result that states the almost sure convergence of the halting time to a constant value.

\begin{theorem}[Halting time universality] \label{thm: Halting_time_main} Fix an $\varepsilon > 0$ and suppose $\concentration \neq \varepsilon$ for all $k$. Under Assumptions~\ref{assumption: Vector} and \ref{assumption: spectral_density},
\begin{empheq}[box=\mybluebox]{equation}
\vphantom{\sum_i^n}\lim_{d \to \infty} \Pr(T_{\varepsilon} = \tau_{\varepsilon} ) = 1\,.
\end{empheq}
\end{theorem}
A result of this form previously appeared in \cite{deift2019conjugate} for the conjugate gradient method. 

\subsubsection{Extension beyond least squares, ridge regression} \label{sec:ridge_regression_main}
One extension of Theorems~\ref{thm: concentration_main} and \ref{thm: Halting_time_main} to other objective functions is the ridge regression problem or $\ell_2$-regularization, that is, we consider a problem of the form
\begin{equation} \label{eq:ridge_regression_main}
    \argmin_{\xx \in \mathbb{R}^d} \left \{ f(\xx) \defas \frac{1}{2n} \|\AA \xx - \bb\|^2 + \frac{\gamma}{2} \|\xx\|^2 \right \}, \quad \text{with $\bb \defas \AA \widetilde{\xx} + \eeta$\,.}
\end{equation}
As discussed above, we assume that $\AA \in \mathbb{R}^{n \times d}$ is (possibly random) data matrix, $\widetilde{\xx} \in \mathbb{R}^d$ is an unobserved signal vector, and $\eeta \in \mathbb{R}^n$ is a noise vector. We make the same assumptions on the limiting spectral measure of $\HH = \frac{1}{n} \AA^T \AA$, the ratio of features to samples, that is, $d/n$ tends to some fixed $r \in (0,\infty)$ as $d \to \infty$, and the magnitude of the noise $\widetilde{R}^2 = \tfrac{1}{d} \mathbb{E}[\|\eeta\|^2]$. In addition to the independence assumption between the data matrix $\AA$ and the signal $\widetilde{\xx}$ and $\xx_0$, we add that the signal and the initialization are also independent of each other with magnitudes $\mathbb{E}[\|\xx_0\|^2] = \dot{R}^2$ and $\mathbb{E}[\|\widetilde{\xx}\|^2] = \widehat{R}^2$ (see Assumption~\ref{assumption:ridge_vector} for precise statement). The constant $\gamma > 0$ is the ridge regression parameter.

The addition of the $\ell_2$-regularizer to the least squares problem alters the Hessian of the least squares by adding a multiple of the identity. Therefore the matrix $\MM \defas \HH + \gamma \II$ and its eigenvalues play the role of $\HH$ and its eigenvalue in Theorem~\ref{thm: concentration_main}. The result is the following theorem.

\begin{theorem}[Concentration of the gradient for ridge regression] \label{thm: concentration_main_main_ridge}
Under Assumptions~\ref{assumption:ridge_vector} and~\ref{assumption: spectral_density} the norm of the gradient concentrates around a deterministic value:
\begin{equation} \begin{aligned} \label{eq: something_1_main_ridge_main} \vspace{0.25cm}
\hspace{-0.28cm}  \|\nabla f(\xx_k)\|^2 \Prto[d] &\textcolor{teal}{\overbrace{\dot{R}^2}^{\text{initial.}}} \! \!\!  \int  { \underbrace{(\lambda + \gamma)^2 P_k^2(\lambda + \gamma; \lambda^{\pm})}_{\text{algorithm}}} \textcolor{mypurple}{\overbrace{\dif\mu}^{\text{model}} } + \textcolor{teal}{\overbrace{\widehat{R}^2}^{\text{signal}}} \! \!\!   \int  { \underbrace{\lambda^2 P_k^2(\lambda + \gamma; \lambda^{\pm})}_{\text{algorithm}}} \textcolor{mypurple}{\overbrace{\dif\mu}^{\text{model}} }  \\
& \quad \quad +   \textcolor{purple}{\overbrace{ \widetilde{R}^2} ^{\text{noise}} }  r   \int  { \underbrace{\lambda P_k^2(\lambda + \gamma; \lambda^{\pm})}_{\text{algorithm}}}  \textcolor{mypurple}{\overbrace{ \dif\mu}^{\text{model}} }. \end{aligned}
\end{equation}
\end{theorem}
Here $\dif \mu$ is the limiting spectral density of $\HH$. The limiting gradient \eqref{eq: something_1_main_ridge_main} decomposes into three terms which highlight the effects of initialization, signal, and noise. This is unlike the two terms in \eqref{eq: something_1_main} which illustrate the noise and signal/initialization effects. The extra $\dot{R}^2$ term in \eqref{eq: something_1_main_ridge_main} only adds to the magnitude of the gradient due to the independence between the signal and initialization. We also note that the matrix $\MM$ always has eigenvalues bounded away from $0$ even in the limit as $d \to \infty$. As such, we expect linear convergence. By defining the right-hand side of \eqref{eq: something_1_main_ridge_main} to be $\concentration$, it follows that Theorem~\ref{thm: Halting_time_main} holds under Assumption~\ref{assumption:ridge_vector} in replace of Assumption~\ref{assumption: Vector}.  For additional discussion see Section~\ref{sec:ridge_regression}.

\subsection{Comparison between average and worst-case} \label{sec: average_case}
The average-case analysis we develop in this paper is effective in the large problem size limit, whereas worst-case analysis is performed for a fixed matrix size.  This implies that there are potentially \emph{dimension-dependent} quantities which must be addressed when making a comparison.

For example, all the first-order methods considered here converge linearly for the finite-dimensional least squares problem: the rate is determined by the gap between the smallest nonzero eigenvalue of the matrix $\HH$ and $0$.  However this could very well be meaningless in the context of a high-dimensional problem, as this gap becomes vanishingly small as the problem size grows.  

In the context of the isotropic features model, when the ratio of features to samples $r$ is $1,$ this is precisely what occurs: the smallest eigenvalues tend to $0$ as the matrix size grows.  In contrast, when $r$ is bounded away from $1$, the least squares problem in \eqref{eq:LS_main} has a \emph{dimension-independent} lower bound on the Hessian which holds with overwhelming probability, (c.f.\ Figure~\ref{fig:MP}).  However, for the comparison we do here, there is another dimension-dependent quantity which will have a greater impact on the worst-case bounds. 

Before continuing, we remark on some terminology we will use throughout the paper. While for any realization of the least squares problem the Hessian $\HH$ is almost surely positive definite, as problem size grows, the matrix $\HH$ can become ill-conditioned, that is, the smallest eigenvalues tend to $0$ as $n \to \infty$ when $r = 1$. Consequently, the computational complexity of first-order algorithms as $n \to \infty$ exhibit rates similar to non-strongly convex problems. On the other hand, when $r$ is bounded away from $1$, the gap between the smallest nonzero eigenvalue of $\HH$ and 0 results in first order methods having complexity rates similar to strongly convex problems. We use this terminology, \textit{non-strongly convex} and \textit{strongly convex}, in Tables~\ref{tab:comparison_worst_avg_cvx} and \ref{tab:comparison_worst_avg_str_cvx} to distinguish the different convergence behaviors when $r = 1$ and $r \neq 1$ resp. and for worst-case complexity comparisons. 

\paragraph{Worst-case rates and the distance to optimality.} 
Typical worst-case upper bounds for first-order algorithms depend on the distance to optimality, ${\|\xx_0-\xx^{\star}\|^2}$. For example, let us consider gradient descent (GD). Tight worst-case bounds for GD in the strongly convex and convex setting \citep{taylor2017smooth}, respectively, are
\begin{gather*}
  \|\nabla f(\xx_k)\|^2 \le (\lambda_{\HH}^+)^2 \|\xx_0-\xx^{\star}\|^2 \left ( 1- \tfrac{\lambda_{\HH}^-}{\lambda_{\HH}^+} \right )^{2k} \defas \mathrm{UB}_{\text{sc}}(\|\nabla f(\xx_k)\|^2\\
    \text{and} \quad \|\nabla f(\xx_k)\|^2 \le \frac{(\lambda^+_{\HH})^2 \|\xx_0-\xx^{\star}\|^2}{(k+1)^2} \defas \mathrm{UB}_{\text{cvx}}(\|\nabla f(\xx_k)\|^2),
\end{gather*}
where $\xx^{\star}$ is the solution to \eqref{eq:LS_main} found by the algorithm, i.e, the iterates of the algorithm converge $\xx_k \to \xx^{\star}$.

To formulate a comparison between the fully-average-case rates, where $\AA$ follows isotropic features, and the worst-case rates, we must make an estimate of this distance to optimality ${\|\xx_0-\xx^{\star}\|^2}$. In the noiseless setting $(\widetilde{R} = 0)$, the expectation of $\|\xx_0-\xx^{\star}\|^2$ is a constant multiple of $R^2$. In particular, it is independent of the dimension. Similarly when we have dimension-independent-strong-convexity, $(r \neq 1),$ even with noisy targets $\bb$ $(\widetilde{R} > 0)$, the distance to the optimum is well-behaved and $\mathbb{E}[\|\xx_0-\xx^{\star}\|^2]$ is a constant involving $\widetilde{R}^2$ and $R^2$. Hence, a direct comparison between worst and average-case is relatively simple. 

For the ill-conditioned case when $r=1$, the situation is more complicated with noisy targets.  To maintain a fixed and finite signal-to-noise ratio, the distance to optimality will behave like ${\|\xx^{\star} - \xx_0 \|^2} \approx d \widetilde{R}^2$; that is, it is dimension-dependent.\footnote{Precisely, we show that $\tfrac{d \widetilde{R}^2}{\|\xx^{\star}-\xx_0\|^2}$ is tight (see Section~\ref{sec: avg_derivations}, Lemma~\ref{lem:growth_dist_optimum}).} So the worst-case rates have a dimension-dependent constant whereas the average-case rates are dimension-independent. This dimension-dependent term can be see in the last column of Table~\ref{tab:comparison_worst_avg_cvx}.  Conversely, if one desires to make $\EE[\|\xx_0-\xx^\star\|^2]$ constant across dimensions using a generative model with noise, one is forced to scale $\eeta$ to go to zero as $d \to \infty$, thus reducing the full generative model to the noiseless regime.

\paragraph{Adversarial model.}
As mentioned above, the comparison with existing worst-case bounds is problematic due to dimension-dependent factors. To overcome this, we consider the following \textit{adversarial model}. First, we assume a noisy generative model for $\bb$ (Assumption~\ref{assumption: Vector} holds). Next, our adversary chooses the matrix $\AA$ without knowledge of $\bb$ to \textit{maximize the norm of the gradient} subject to the constraint that the convex hull of the eigenvalues of $\HH = \tfrac{1}{n}\AA^T \AA$ equals $[\lambda^{-},\lambda^+]$.  For comparison to the average-case analysis with isotropic features, we would choose $\lambda^{\pm}$ to be the endpoints of the Mar\v{c}enko-Pastur law. In light of Theorem~\ref{thm: concentration_main}, the adversarial model seeks to solve the constrained optimization problem
\begin{equation} \begin{aligned} \label{eq: adversary_worst_case_main}
    \lim_{d \to \infty} \max_{\HH} \, \mathbb{E} \big [ \|\nabla f(\xx_k)\|^2 \big ] 
    &= 
    \max_{ \lambda \in [\lambda^-, \lambda^+] } 
    \bigl\{ R^2 \lambda^2 P_k^2(\lambda) + \widetilde{R}^2 r \lambda P_k^2(\lambda)\bigr\}.
    \end{aligned}
\end{equation} 
We call this expression the \textit{adversarial average-case guarantee}.

The main distinction between worst-case and adversarial average-case is that traditional worst-case maximizes the gradient over \textit{all} inputs -- both targets $\bb$ and data matrix $\AA$.  This leads to dimension--dependent complexity, as there are usually exceptional target vectors that are heavily dependent on the data matrix $\AA$ (such as those built from extremal singular vectors of $\AA$) and cause the algorithm to perform exceptionally slowly.

In contrast, the adversarial average-case keeps the randomness of the target $\bb$ while maximizing over the data matrix $\AA$.  This is a more meaningful worst-case comparison: for example, in the setting of linear regression, the response and measurements of the independent variables are typically generated through different means and have different and independent sources of noise (see for example \cite[Example 10.1]{walpole1989probability}.  Hence the independence of the noise intervenes to limit how truly bad the data matrix $\AA$ can be. Furthermore, the complexity of the adversarial average-case is dimension-independent.  Table~\ref{tab:comparison_worst_avg_cvx} shows these adversarial complexities for non-strongly convex objectives \eqref{eq:LS_main}. Similar results can also be derived for strongly convex objectives but are omitted for brevity. 

\begin{figure}
 \centering
     \includegraphics[width = 0.48\linewidth]{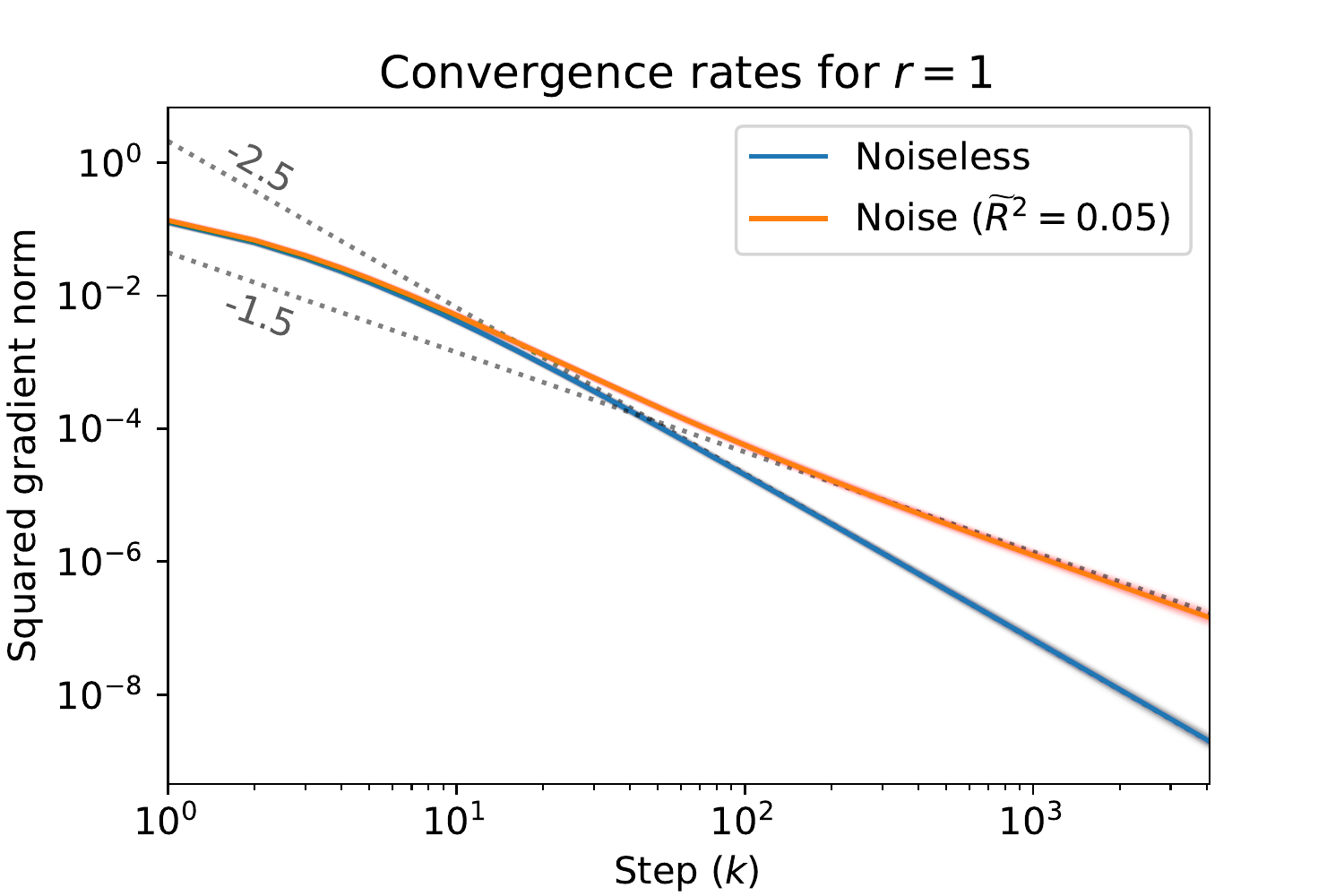}
    \includegraphics[width = 0.48\linewidth]{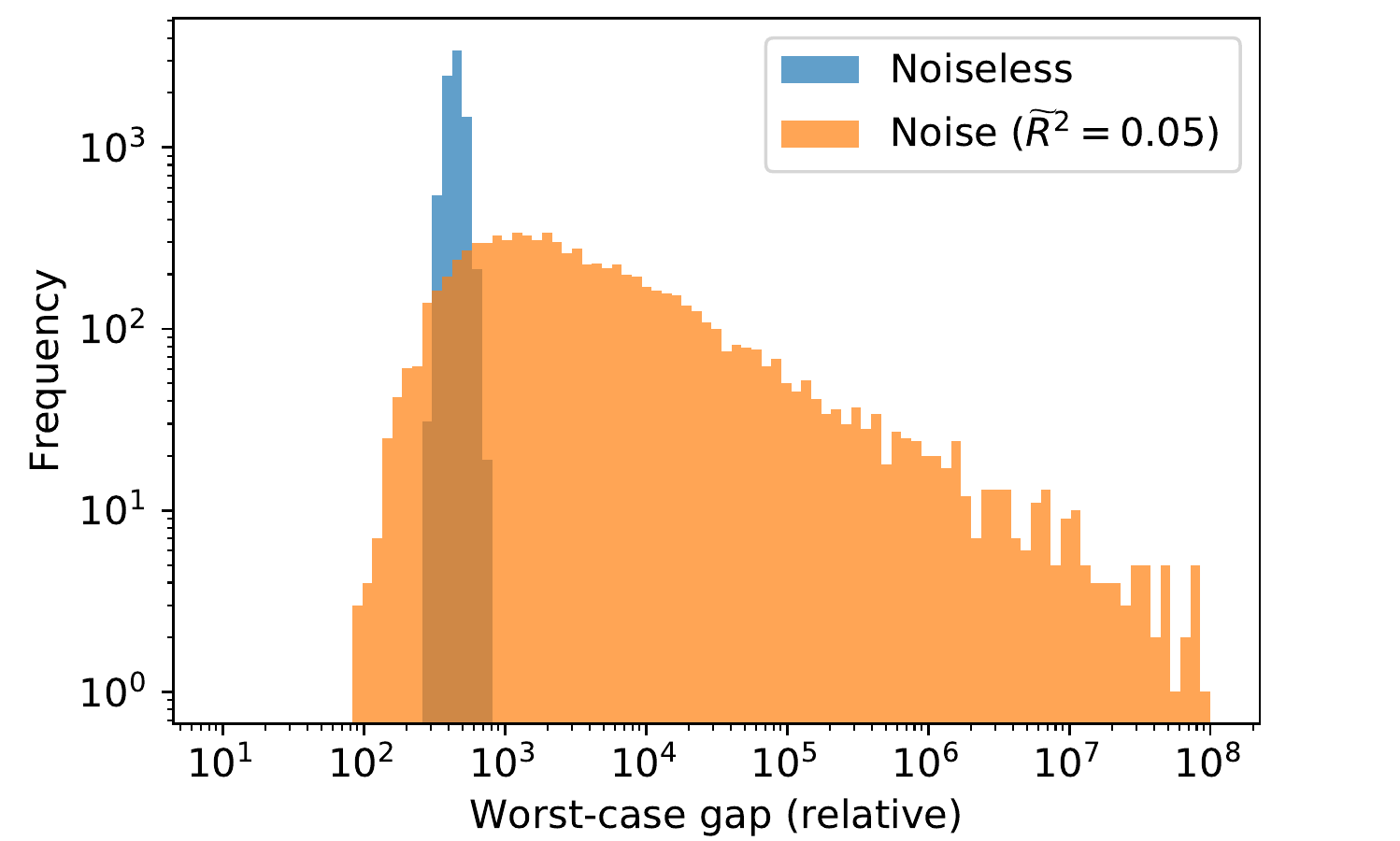} 
\caption{{\bfseries Average-case vs worst-case in least squares} with isotropic features ($r =1, d = 4096$). {\bfseries Left}: 8000 runs of GD, standard deviation (shaded region, undetectable), and theoretical rates (dashed lines). Empirical runs precisely match the theoretical average-case rates.
{\bfseries Right}: Ratio of the upper bound in worst-case to empirical gradient after $k=4096$ iterations,
$\mathrm{UB}_{\text{cvx}}(\|\nabla f(\xx_k)\|^2) / \|\nabla f(\xx_k)\|^2$. From the concentration of gradients (left), this implies that \emph{the norms of the gradient for worst-cases are always larger than average.} The distribution of worst-case gradient rates with noise has large variance contrasting the little variance in (left) and makes the expected worst-case unpredictable in contrast with the noiseless. 
 \label{fig:avg_rates} 
} 
\end{figure}

\paragraph{Comparison with adversarial and worst-case complexities.}
By construction, the average-case convergence rates are at least as good as the worst-case and adversarial guarantees. The average-case complexity in the convex, noiseless setting ($r =1, \widetilde{R}=0$) for Nesterov's accelerated method (convex) \citep{nesterov2004introductory, Beck2009Fast} and gradient descent (GD) are an order of magnitude faster in $k$ than the worst case rates (see Table~\ref{tab:comparison_worst_avg_cvx}, first column). It may appear at first glance (Table~\ref{tab:comparison_worst_avg_cvx}) that there is a discrepancy between the average-case and exact worst-case rate when $r =1$ and noisy setting ($\eeta \neq \bm{0}$). As noted in the previous section, the worst-case rates have dimension-dependent constants.  Provided the dimension is bigger than the iteration counter ($d \ge k^{1/2}$ for GD and $d \ge \log(k)$ for Nesterov), the average complexity indeed yields a faster rate of convergence. Average-case is always strictly better than adversarial rates (see Table~\ref{tab:comparison_worst_avg_cvx}). This improvement in the average rate indeed highlights that \textit{the support of the spectrum does not fully determine the rate.} Many eigenvalues contribute meaningfully to the average rate. Hence, our results are not and cannot be purely explained by the support of the spectrum. 

The average-case complexity in the strongly convex case matches the worst-case guarantees multiplied by an additional \textit{polynomial correction term} (\textcolor{teal}{green} in Table~\ref{tab:comparison_worst_avg_str_cvx}). This polynomial term has little effect on the complexity compared to the linear rate. However as the matrix $\HH$ becomes ill-conditioned $(r\to1)$, the polynomial correction starts to dominate the average-case complexity. The sublinear rates in Table~\ref{tab:comparison_worst_avg_cvx} show this effect and it accounts for the improved average-case rates. 

Our average-case rates accurately predict the empirical convergence observed in simulations, in contrast to the worst-case rates (see Figure~\ref{fig:avg_rates}). Although our rates only hold on average, surprisingly, even a single instance of GD exactly matches the theoretical predictions. Moreover, the noisy non-strongly convex worst-case is highly unpredictable due to the instability in $\xx^{\star}$ across runs. As such, the worst-case analysis is not representative of typical behavior (see Figure~\ref{fig:avg_rates}). 

These theoretical results are supported by simulations and empirically extended to other models, such as logistic regression, as well as other algorithms, such as stochastic gradient descent (SGD) (see Section~\ref{sec:numerical_simulations}). This suggests that this universality property holds for a wider class of problems. 

\paragraph{Related work.} The average-case analysis has a long history
in computer science and numerical analysis. Often it is used to
justify the superior performance of algorithms as compared with their
worst-case bounds such as Quicksort (sorting)
\citep{Hoare1962Quicksort} and the simplex method in linear
programming, see for example \citep{Spielman2004Smooth, smale1983on,
  borgwardt1986probabilistic, todd1991probabilistic} and references
therein. Despite this rich history, it is challenging to transfer
these ideas into continuous optimization due to the ill-defined notion
of a typical continuous optimization problem. Recently
\citet{pedregosa2020average, lacotte2020optimal} derived a framework
for average-case analysis of gradient-based methods and developed
optimal algorithms with respect to the average-case. The class of
problems they consider is a special case of \eqref{eq:LS_main} with
vanishing noise. We use a similar framework -- extending the results to all first-order methods and noisy quadratics while also providing concentration and explicit average-case convergence guarantees.

A natural criticism of a simple average-case analysis is that the complexity is data model dependent and thus it only has predictive power for a small subset of real world phenomena. Because of this, it becomes important to show that any modeling choices made in defining the data ensemble have limited effect. \citet{paquette2020universality} showed that the halting time for conjugate gradient becomes deterministic as the dimension grows and it exhibits a universality property, that is, for a class of sample covariance matrices, the halting times are identical (see also \citet{deift2019conjugate}).
It is conjectured that this property holds in greater generality -- for more distributions and more algorithms~\citep{deift2014universality,deift2018universality}). In \citet{Sagun2017Universal}, empirical evidence confirms this for neural networks and spin glass models. Our paper is in the same spirit as these-- definitively answering the question that all first-order methods share this universality property for the halting time on quadratic problems. 

This work is inspired by research in numerical linear algebra that uses random matrix theory to quantify the ``probability of difficulty" and ``typical behavior" of numerical algorithms \citep{demmel1988probability}. For many numerical linear algebra algorithms, one can place a random matrix as an input and analyze the algorithm's performance. It is used to help explain the success of algorithms and heuristics that could not be well understand through traditional worst-case analysis. Numerical algorithms such as the QR \citep{pfrang2014how}, Gaussian elimination \citep{sankar2006smoothed,trefethen1990average}, and other matrix factorization algorithms, for example, symmetric triadiagonalization and bidiagonalization \citep{edelman2005random} have had their performances analyzed under random matrix inputs (typically Gaussian matrices). In \cite{deift2019universality}, an empirical study extended these results beyond Gaussian matrices and showed that the halting time for a many numerical algorithms were independent of the random input matrix for a large class of matrix ensembles. This universality result eventually was proven for the conjugate gradient method \citep{paquette2020universality,deift2019conjugate}.

An alternative approach to explaining successes of numerical algorithms, introduced in \citep{Spielman2004Smooth}, is smoothed analysis. Smoothed analysis is a hybrid of worst-case and average-case analysis. Here one randomly perturbs the worst-case input and computes the maximum expected value of a measure for the performance of an algorithm. It has been used, for example, to successful analyze linear programming \citep{Spielman2004Smooth}, semi-definite programs \citep{bhojanapalli2018smoothed}, and conjugate gradient \citep{menon2016smoothed}. In this work, we instead focus on the random matrix approach to analyze first-order methods on optimization problems. 

Our work draws heavily upon classical polynomial based iterative methods. Originally designed for the Chebyshev iterative method \citep{Flanders1950Numerical,golub1961chebyshev}, the polynomial approach for analyzing algorithms was instrumental in proving worst-case complexity for the celebrated conjugate gradient method \citep{Hestenes&Stiefel:1952}. For us, the polynomial approach gives an explicit equation relating the eigenvalues of the data matrix to the iterates which, in turn, allows the application of random matrix theory.\\

The remainder of the article is structured as follows: in Section~\ref{sec: problem_setting} we introduce the full mathematical model under investigation including some examples of data models. Section~\ref{sec: poly} discusses the relationship between polynomials and optimization. Our main results are then described and proven in Section~\ref{sec: halting_time}. Section~\ref{sec: avg_derivations} details the computations involved in the average-case analysis, the proofs of which are deferred to the appendix. The article concludes on showing some numerical simulations in Section~\ref{sec:numerical_simulations}.

\section{Problem setting} \label{sec: problem_setting} In this paper, we develop an average-case analysis for first-order methods on quadratic problems of the form
\begin{equation} \label{eq:LS}
\vspace{0.5em}\argmin_{\xx \in \RR^d} \Big \{ f(\xx) \defas \frac{1}{2n} \|\AA \xx-\bb\|^2 \Big \}, \quad \text{with } \bb \defas \AA \widetilde{\xx} + \eeta\,,
\end{equation}
where $\AA \in \RR^{n \times d}$ is a (possibly random) matrix (discussed in the next subsection), $\widetilde{\xx} \in \RR^d$ is an unobserved signal vector, and $\eeta \in \RR^n$ is a noise vector.

\subsection{Data matrix, noise, signal, and initialization assumptions}\label{sec: assumptions}
Throughout the paper we make the following assumptions. 
\begin{assumption}[Initialization, signal, and noise] \label{assumption: Vector} The initial vector $\xx_0 \in \RR^d$, the signal $\widetilde{\xx} \in \RR^d$, and noise vector $\eeta \in \RR^n$ are independent of $\AA$ and satisfy the following conditions:
\begin{enumerate}[leftmargin=*]
    \item The entries of $\xx_0-\widetilde{\xx}$ are i.i.d. random variables and there exist constants $C, R > 0$ such that for $i = 1, \ldots, d$
    \begin{equation} \label{eq:R} \begin{gathered} \EE[\xx_0-\widetilde{\xx}] = \bm{0}, \quad   \EE[\|\xx_0-\widetilde{\xx}\|^2] = R^2,
    \quad \text{and} \quad  \EE[(\widetilde{\xx}-\xx_0)_{i}^4] \le \tfrac{1}{d^2} C.
    \end{gathered}
    \end{equation}
    \item The entries of the noise vector $\eeta$ are i.i.d. random variables satisfying the following for $i = 1, \ldots, n$ and for some constants $\widetilde{C}, \widetilde{R} > 0$  
    \begin{equation}
    \EE[\eeta] = \bm{0}, \quad \EE[\eta_i^2] = \widetilde{R}^2, \quad \text{and} \quad \EE[\eta_i^4] \le \widetilde{C}.
    \end{equation}
\end{enumerate}
\end{assumption}

Assumption~\ref{assumption: Vector} encompasses the setting where the signal is \textit{random} and the algorithm is initialized at $\xx_0 = \bm{0}$. But it is more general. Starting farther from the signal requires more iterations to converge. Hence, intuitively, \eqref{eq:R} restricts the distance of the algorithm's initialization to the signal so that it remains constant across problem sizes. The unbiased initialization about the signal, put another way, says the initialization is rotationally-invariantly distributed about the signal $\widetilde{x}$ (see Figure~\ref{fig: Assumption_1}).
 
Assumption~\ref{assumption: Vector} arises as a result of preserving a constant signal-to-noise ratio in the generative model. 
Such generative models with this scaling have been used in numerous works \citep{mei2019generalization,hastie2019surprises}. 

\begin{wrapfigure}[16]{r}{0.47\textwidth}
\vspace{-0.5cm}
    \centering \begin{tikzpicture}[scale = 0.72]
\filldraw[pattern color = darkgray, pattern= north west lines, draw = gray, dashed] (0,0) circle [x radius=1cm, y radius=5mm, rotate=30];
\draw (0,0) circle [radius=2.5];
  \node[mark size=2pt] at (0,0) {\pgfuseplotmark{*}};
  \node at (0.25,0) {$\widetilde{\xx}$};
   
      \node[mark size=2pt,color=red] at (-1.5,2) {\pgfuseplotmark{*}};
    \node[red] at (-1.9, 2.05) {$\xx_0$};
    \node[mark size=2pt,color=red] at (-0.21,0.5) {\pgfuseplotmark{*}};
    \node[color=red] at (-0.1,0.75) {$\xx_k$};
     \draw[red] (-1.5,2)--(1.25,1.25)--(-1,1)-- (-0.25,0.5);

  \node[mark size=3pt,color=blue] at (-2.45,-0.5) {\pgfuseplotmark{triangle*}};
      \node[color=blue] at (-2.8,-0.5) {$\xx_0$};
            \draw[blue] (-2.45,-0.5)--(0.75,-1.75)--(-0.3,-1)--(0.25,-0.5);
         \node[mark size=3pt,color=blue] at (0.25,-0.5) {\pgfuseplotmark{triangle*}};
      \node[color=blue] at (-0.1,-0.5) {$\xx_k$};
      
      \node[mark size=2.75pt,color=orange] at (2.29,-1) {\pgfuseplotmark{square*}};
      \node[color=orange] at (2.75,-1) {$\xx_0$};
      \draw[orange] (2.29, -1)--(1,-1.25)--(1.25,-0.5)--(0.9,0.4);
          \node[mark size=2.75pt,color=orange] at (0.9,0.4) {\pgfuseplotmark{square*}};
      \node[color=orange] at (1.25,0.4) {$\xx_k$};
    \end{tikzpicture}
        \caption{The pictured $\xx_0$ are equiprobable. Each colored line is a different run of GD with random matrix $\AA$ and the shaded gray area is the set where $\|\nabla f(\xx)\|^2 < \varepsilon$. Intuitively, our result says all runs of GD starting from a random $\xx_0$ take the same number of iterations to reach the shaded area. } \label{fig: Assumption_1}
\end{wrapfigure}
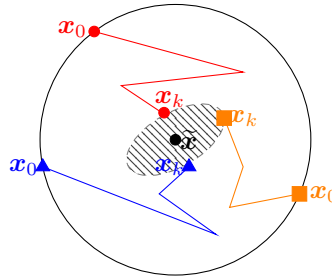

\paragraph{Tools from random matrix theory.} Random matrix theory studies properties of matrices $\HH$ (most notably, statistics of matrix eigenvalues) whose entries $H_{ij}$ are random variables. We refer the reader to \citep{bai2010spectral, tao2012topics} for a more thorough introduction. Many important statistics of random matrix theory can be expressed as functionals on the eigenvalues of a matrix $\HH$ (\textit{e.g.},  determinants and traces). Let $\lambda_1, \ldots, \lambda_d$ be the eigenvalues of $\HH$ and define the \textit{empirical spectral measure} (ESM), $\mu_{\HH}$, as
\begin{equation}
    \mu_{\HH}(\lambda) \defas \frac{1}{d} \sum_{i=1}^d \delta_{\lambda_i},
\end{equation}
where $\delta_{\lambda_i}$ is a Dirac delta function, \textit{i.e.}, a function equal to $0$ except at $\lambda_i$ and whose integral over the entire real line is equal to one. The empirical spectral measure puts a uniform weight on each of the eigenvalues of $\HH$. When $\HH$ is random, this becomes a random measure. A main interest in random matrix theory is to characterize the behavior of the empirical spectral measure as the dimension of the matrix tends to infinity.

Because the ESM is a well-studied object for many random matrix ensembles, we state the following assumption on the ESM for the data matrix, $\AA$. In Section~\ref{sec: data_generate}, we review practical scenarios in which this is verified.

\begin{assumption}[Data matrix] \label{assumption: spectral_density}
Let $\AA$ be a (possibly random) $n \times d$ matrix such that the number of features, $d$, tends to infinity proportionally to the size of the data set, $n$, so that $\tfrac{d}{n} \to r \in (0, \infty)$. Let $\HH \defas \tfrac{1}{n} \AA^T \AA$ with eigenvalues $\lambda_1 \leq \ldots \leq \lambda_d$ and let $\delta_{\lambda_i}$ denote the Dirac delta with mass at $\lambda_i$. We make the following assumptions on the eigenvalue distribution of this matrix:
\begin{enumerate}[leftmargin=*]
    \item The ESM converges weakly in probability to a deterministic measure $\mu$ with compact support,
    \begin{equation} \label{eq:ESM_convergence}
    \mu_{\HH} = \mfrac{1}{d}\sum_{i=1}^d \delta_{\lambda_i} \to \mu \quad \text{weakly in probability\,.}
    \end{equation}
    \item The largest eigenvalue of $\HH$ converges in probability to the largest element in the support of $\mu$. In particular, if $\lambda^+$ denotes the top edge of the support of $\mu$ then
    \begin{equation} \label{eq:max_eigenvalue} \lambda_{\HH}^+ \Prto[d] \lambda^+. \,\end{equation}
     \item (Required provided the algorithm uses the smallest eigenvalue) The smallest eigenvalue of $\HH$ converges in probability to the smallest, non-zero element in the support of $\mu$. In particular, if $\lambda^-$ denotes the bottom edge of the support of $\mu$ then
    \begin{equation} \label{eq:min_eigenvalue} \lambda_{\HH}^- \Prto[d] \lambda^-. \,\end{equation}
\end{enumerate}
\end{assumption}

\subsection{Examples of data distributions.} \label{sec: data_generate}

In this section we review three examples of data-generating distributions that verify Assumption~\ref{assumption: spectral_density}: a model with isotropic features, a correlated features model, and a one-hidden layer neural network with random weights. Numerous works studying the spectrum of the Hessian on neural networks have found that this spectrum shares many characteristics with the limiting spectral distributions discussed below including compact support, a concentration of eigenvalues near $0$, and a stable top eigenvalue \citep{dauphin2014identifying, papyan2018the, sagun2016eigenvalues, behrooz2019investigation}. In fact, the work of \citet{martin2018implicit} directly compares the Hessians of deep neural networks at various stages in training with the Mar\v{c}enko-Pastur density, that is, the limiting spectral density for the isotropic features model. 

\paragraph{Isotropic features.} 
We will now elaborate on the well developed theory surrounding the isotropic features model (see \eqref{eq:MP} and the text just above it).
In particular, parts 2 and 3 of Assumption~\ref{assumption: spectral_density} on the convergence of the largest and smallest eigenvalues is well known:
\begin{lemma}[Isotropic features]({\rm \textbf{\citet[Theorem 5.8]{bai2010spectral}}}) \label{lem:bai_Spectral} 
Suppose the matrix ${\AA \in \RR^{n \times d}}$ is generated using the isotropic features model. 
The largest and smallest eigenvalue of $\HH$, $\lambda_{\HH}^+$ and $\lambda_{\HH}^-$, resp., converge in probability to $\lambda^+$ and $\lambda^-$ resp. where $\lambda^+ = \sigma^2 (1+ \sqrt{r})^2$ is the top edge of the support of the Mar\v{c}enko-Pastur measure and $\lambda^- = \sigma^2(1+\sqrt{r})^2$ is the bottom edge of the support of the Mar\v{c}enko-Pastur measure. 
\end{lemma}


In addition, the isotropic features model is sufficiently random that it is possible to weaken Assumption \ref{assumption: Vector} and still derive for it the conclusion of Theorem \ref{thm: concentration_main}.
In particular, we may let $\bb_n$ be defined as 
\begin{equation}\label{eq:weakb}
\bb_n = 
\mfrac{1}{\sqrt{n}} R \AA \boldsymbol{\omega}_{1,d}
+
\widetilde{R} \boldsymbol{\omega}_{2,n}
\end{equation}
for any deterministic sequences of vectors $\{\boldsymbol{\omega}_{1,d}\}$ and $\{\boldsymbol{\omega}_{2,n}\}$ from the $d$-dimensional and the $n$-dimensional spheres, respectively, multiplied by the signal strength $R$ and noise $\widetilde{R}$. 
Then, under the \emph{further} moment assumption on $\AA$ that for any $k \in \mathbb{N}$
\begin{equation}\label{eq:strongmoments}
\sup_{i,j} \biggl\{\mathbb{E} |A_{i,j}|^k \biggr\} < \infty,
\end{equation}
it is a consequence of \cite[Theorem 3.6,3.7]{KnowlesYin},
that
\begin{equation}
    \label{eq:isotropicE}
  \|\nabla f(\xx_k)\|^2\Prto[d]
  R^2 \int {\lambda^2 P_k^2(\lambda)}\dif\mu+\widetilde{R}^2 r \int {\lambda P_k^2(\lambda)}\dif\mu
  =
  \concentration.
\end{equation}
This implies that for the isotropic features model under the stronger assumption for the data matrix \eqref{eq:strongmoments}, but the weaker target assumption \eqref{eq:weakb}, we obtain the same complexity results presented in Tables \ref{tab:comparison_worst_avg_cvx} and \ref{tab:comparison_worst_avg_str_cvx}.  See also \cite[Corollary 5.12]{paquette2020universality} in which a central limit theorem for the gradient is derived under these same assumptions

\paragraph{Correlated features.} 
In this model, one takes a random matrix $\WW \in \mathbb{R}^{n \times d}$ generated from the isotropic features model and a symmetric positive definite correlation matrix $\SSigma_d \in \mathbb{R}^{d \times d}$.  One then defines the correlated features model by
\[
\AA \defas \WW \SSigma_d^{1/2}.
\]
This makes $\HH = \frac{1}{n}\AA^T \AA$ the normalized sample covariance matrix of $d$ samples of a $n$-dimensional random vector with covariance structure $\SSigma_d.$

Under the assumption that the empirical spectral measure of $\SSigma_d$ converges to a measure $\nu$ and that the norm of $\SSigma_d$ is uniformly bounded, it is consequence of \cite{Bai1999a,Bai1999b} (see also the discussions in \cite{bai2004CLT,KnowlesYin, HachemHardyNajim}) that Assumption \ref{assumption: spectral_density} holds.  Unlike in isotropic features, the limiting spectral measure is not known explicitly, but is instead only characterized (in general) through a fixed-point equation describing its Stieltjes transform.

\paragraph{One-hidden layer network with random weights.} In this model, the entries of $\AA$ are the result of a matrix multiplication composed with a (potentially non-linear) activation function $g \, : \, \mathbb{R} \mapsto \mathbb{R}$:
\begin{align}
    A_{ij} \defas g \big (\tfrac{[\WW \YY]_{ij}}{\sqrt{m}} \big ), \quad \text{where $\WW \in \RR^{n \times m}$, $\YY \in \RR^{m \times d}$ are random matrices\,.} 
\end{align}
The entries of $\WW$ and $\YY$ are i.i.d. with zero mean, isotropic variances
    $\EE[W_{ij}^2] = \sigma_w^2$ and $\EE[Y_{ij}^2] = \sigma_y^2$, and light tails, that is, there exists constants $\theta_w, \theta_y > 0$ and $\alpha > 0$ such that for any $t > 0$
\begin{equation} \label{eq: light_tail}
    \Pr(|W_{11}| > t) \le \exp(-\theta_w t^{\alpha}) \quad \text{and} \quad \Pr(|Y_{11}| > t) \le \exp(-\theta_y t^{\alpha})\,.
\end{equation}
Although stronger than bounded fourth moments, this assumption holds for any sub-Gaussian random variables (\textit{e.g.}, Gaussian, Bernoulli, etc). As in the previous case to study the large dimensional limit, we assume that the different dimensions grow at comparable rates given by $\frac{m}{n} \to r_1 \in (0, \infty)$ and $\frac{m}{d} \to r_2 \in (0, \infty)$.
This model encompasses two-layer neural networks with a squared loss, where the first layer has random weights and the second layer's weights are given by the regression coefficients $\xx$.
In this case, problem \eqref{eq:LS} becomes
\begin{equation} \label{eq: general_LS}
    \min_\xx \, \left\{ f(\xx) = \mfrac{1}{2n} \|\g \big ( \tfrac{1}{\sqrt{m}} \WW \YY \big )\xx - \bb\|^2_2  \right\}.
\end{equation} 

The model was introduced by \citep{Rahimi2008Random} as a randomized approach for scaling kernel methods to large datasets, and has seen a surge in interest in recent years as a way to study the generalization properties of neural networks
\citep{hastie2019surprises,mei2019generalization,pennington2017nonlinear,louart2018random,liao2018dynamics}. 

The most important difference between this model and the isotropic features is the existence of a potentially non-linear activation function $g$. We assume $g$ to be entire with a growth condition and have zero Gaussian-mean,
\begin{align} \label{eq: Gaussian_mean}
    \hspace{-3em}\text{(Gaussian mean)} \qquad \int \g(\sigma_w \sigma_y z) \tfrac{e^{-z^2/2}}{\sqrt{2 \pi} } \, \dif z = 0\,.
\end{align}
The additional growth condition on the function $g$ is precisely given as there exists positive constants $C_g, c_g, A_0 > 0$ such that for any $A \ge A_0$ and any $n \in \mathbb{N}$ 
\begin{align}
    \sup_{z \in [-A,A]} |g^{(n)}(z)| \le C_g A^{c_g n}\, .
\end{align}
Here $g^{(n)}$ is the $n$th derivative of $g$. This growth condition is verified for common activation functions such as the sigmoid $\g(z) = (1+ e^{-z})^{-1}$ and the softplus $\g(z) = \log(1+e^z)$, a smoothed approximation to the ReLU. The Gaussian mean assumption \eqref{eq: Gaussian_mean} can always be satisfied by incorporating a translation into the activation function. 

\cite{benigni2019eigenvalue} recently showed that the empirical spectral measure and largest eigenvalue of $\HH$ converge to a deterministic measure and largest element in the support, respectively. 
This implies that this model, like the isotropic features one, verifies Assumption~\ref{assumption: spectral_density}.
However, contrary to the isotropic features model, the limiting measure does not have an explicit expression, except for some specific instances of $g$ in which it is known to coincide with the Mar\v{c}enko-Pastur distribution.

\begin{lemma}[One-hidden layer network]({\rm \textbf{\citet[Theorems~2.2 and~5.1]{benigni2019eigenvalue}}}) \label{lem:rand_feat_measure} Suppose the matrix $\AA \in \RR^{n \times d}$ is generated using the random features model. Then there exists a deterministic compactly supported measure $\mu$ such that $\mu_{\HH} \underset{d \to \infty}{\longrightarrow} \mu$ weakly in probability. Moreover $\lambda_{\HH}^+ \Prto[d] \lambda^+$ where $\lambda^+$ is the top edge of the support of $\mu$.
\end{lemma}

\renewcommand{\arraystretch}{2.5}
\ctable[notespar,
caption = {{\bfseries Residual Polynomials.} Summary of the residual polynomials associated with the methods discussed in this paper. $T_k$ is the $k$-th Chebyshev polynomial of the first kind, $U_k$ is $k$-th Chebyshev polynomial of the second kind, and $L_k$ is the $k$-th Legendre polynomial. Derivations of these polynomials can be found in Appendix~\ref{apx: derivation_polynomial}. In light of Proposition~\ref{prop: polynomials_methods}, an explicit expression for the polynomial $P_k$ is enough to determine the polynomial $Q_k$. },
label = {table:polynomials},
captionskip=2ex,
pos =!t
]{l c l}{\tnote[1]{\citep{nesterov2004introductory,Beck2009Fast}} \tnote[2]{\citep{Polyak1962Some}} \tnote[3]{\citep{nesterov2004introductory}} }{
\toprule
\textbf{Methods} &  \textbf{Polynomial $P_k$} & \textbf{Parameters} \\
\midrule
Gradient Descent & $(1-\alpha \lambda)^k$ & $\alpha = 1 / \lambda^+_{\HH}$\\
\midrule
\begin{minipage}{0.18\textwidth} Nesterov (cvx) \tmark[1]
\end{minipage} & $\frac{2(1-\alpha \lambda)^{(k+1)/2}}{\alpha \lambda k} \big ( \sqrt{1-\alpha \lambda} L_k(\sqrt{1-\alpha \lambda}) - L_{k+1}(\sqrt{1-\alpha \lambda}) \big )$
& $\alpha = {1}/{\lambda_{\HH}^+}$\\
\midrule
\begin{minipage}{0.15\textwidth} Polyak \tmark[2]
\end{minipage} &$\beta^k \big [ \tfrac{ ( \sqrt{\lambda_{\HH}^+}-\sqrt{\lambda_{\HH}^-})^2}{\lambda_{\HH}^+ + \lambda_{\HH}^-} \cdot T_k(\sigma(\lambda)) + \tfrac{2 \sqrt{\lambda_{\HH}^- \lambda_{\HH}^+}}{\lambda_{\HH}^+ + \lambda_{\HH}^-} \cdot U_k(\sigma(\lambda)) \big ]$ & \begin{minipage}{0.19\textwidth} $\sigma(\lambda) = \frac{\lambda_{\HH}^+ + \lambda_{\HH}^- - 2\lambda}{\lambda_{\HH}^+ - \lambda_{\HH}^-}$ \\
$\beta = \tfrac{\sqrt{\lambda_{\HH}^+}-\sqrt{\lambda_{\HH}^-}}{\sqrt{\lambda_{\HH}^+} + \sqrt{\lambda_{\HH}^-}}$ \end{minipage}
\\
\midrule
\begin{minipage}{0.18\textwidth} Nesterov\\
(Strongly cvx) \tmark[3]
\end{minipage} & $\tfrac{2\beta (\beta x)^{k/2} }{1+\beta} T_k \left ( \tfrac{1+\beta}{2 \sqrt{\beta}}  \sqrt{x} \right ) + \left (1 - \frac{2\beta}{1+\beta} \right ) (\beta x)^{k/2} U_k \left (\tfrac{1+\beta}{2 \sqrt{\beta}} \sqrt{x} \right )$
& \begin{minipage}{0.18\textwidth} $x = 1-\alpha\lambda$,\\ $\alpha = {1}/{\lambda_{\HH}^+}$\\
$\beta = 
    \tfrac{\sqrt{\lambda_{\HH}^+} - \sqrt{\lambda_{\HH}^-}}{\sqrt{\lambda_{\HH}^+} + \sqrt{\lambda_{\HH}^-} }$
\end{minipage}\\
 \bottomrule
}

\section{From optimization to polynomials} \label{sec: poly}
In this section, we look at the classical connection between optimization algorithms, iterative methods, and polynomials \citep{Flanders1950Numerical,golub1961chebyshev,fischer1996polynomial,rutishauser1959refined}. While the idea of analyzing optimization algorithms from the perspective of polynomials is well-established, many modern algorithms, such as the celebrated Nesterov accelerated gradient \citep{nesterov2004introductory}, use alternative approaches to prove convergence.

This connection between polynomials and optimization methods will be crucial to proving the average-case guarantees in Tables~\ref{tab:comparison_worst_avg_cvx} and \ref{tab:comparison_worst_avg_str_cvx}.
To exploit this connection, we will construct the residual polynomials associated with the considered methods and we prove novel facts which may be of independent interest. For example, the polynomials associated with Nesterov's method provides an alternative explanation to the ODE in \citep{su2016differential}.

Throughout the paper, we consider only gradient-based methods, algorithms which can be written as a linear combination of the previous gradients and the initial iterate. 

\begin{definition}[Gradient-based method] \rm{An optimization algorithm is called a \textit{gradient-based method} if each update of the algorithm can be written as a linear combination of the previous iterate and previous gradients. In other words, if every update is of the form
\begin{equation}\label{eq:gradient_based}
  \xx_{k+1} = \xx_0 + \sum_{i=0}^{k} c_{k i} \nabla f(\xx_k)~,
  \end{equation}
  for some scalar values $c_{k i}$ that can potentially depend continuously on $\lambda^+_{\HH}$ and $\lambda^-_{\HH}$. 
  }
\end{definition}
Examples of gradient-based methods include momentum methods \citep{Polyak1962Some}, accelerated methods \citep{nesterov2004introductory,Beck2009Fast}, and gradient descent. Now given any gradient-based method, we can associate to the method \textit{residual polynomials} $P_k$ and \textit{iteration polynomials} $Q_k$ which are polynomials of degree $k$, precisely as followed.

\begin{proposition}[Polynomials and gradient-based methods] \label{prop: polynomials_methods} Consider a gradient-based method with coefficients $c_{ki}$ that depend continuously on $\lambda^-_{\HH}$ and $\lambda^+_{\HH}$. 
Define the sequence of polynomials $\{P_k, Q_k\}_{k = 0}^\infty$ recursively by
\begin{equation} \begin{gathered} \label{eq:recursive_noise_poly}
    P_0(\HH; \lambda_{\HH}^{\pm}) = \II \quad \text{and} \quad P_k(\HH; \lambda_{\HH}^{\pm}) = \II - \HH Q_{k}(\HH; \lambda^{\pm}_{\HH})\\
      Q_0(\HH; \lambda_{\HH}^{\pm}) = \bm{0} \quad \text{and} \quad Q_k(\HH; \lambda_{\HH}^{\pm}) = \sum_{i=0}^{k-1} c_{k-1,i} \big [ \HH Q_i(\HH; \lambda_{\HH}^{\pm}) - \II \big ]\,.
\end{gathered} \end{equation}
These polynomials $P_k$ and $Q_k$ are referred to as the \emph{residual} and \emph{iteration} polynomials respectively.
We express the difference between the iterate at step $k$ and $\widetilde{\xx}$ in terms of these polynomials:
\begin{equation} \label{eq:recursive_noise_poly_1}
\xx_k - \widetilde{\xx} = P_k(\HH; \lambda_{\HH}^{\pm}) (\xx_0-\widetilde{\xx}) + Q_k(\HH; \lambda_{\HH}^{\pm}) \cdot  \frac{\AA^T \eeta}{n}\,.
\end{equation}
\end{proposition}

\begin{proof}
We will prove the result by induction.
For $k=0$, the claimed result holds trivially. We assume it holds up to iteration $k$ and we will prove it holds for $k+1$. To show this, we will use the following equivalent form of the gradient $\nabla f(\xx) = \HH (\xx - \widetilde{\xx}) - \tfrac{\AA^T \eeta}{n}$, which follows from the definition of $\bb$. Using this and the definition of gradient-based method, we have:
\begin{align*}
    \xx_{k+1} &- \widetilde{\xx} = \xx_0 - \widetilde{\xx} + \sum_{i=0}^{k} c_{ki} \nabla f(\xx_i) = \xx_0 - \widetilde{\xx} + \sum_{i=0}^k c_{ki} \big [\HH (\xx_i-\widetilde{\xx}) - \tfrac{\AA^T \eeta}{n} \big ]\\
    &= \xx_0-\widetilde{\xx} + \sum_{i=0}^k c_{ki}  \big [\HH \big ( \big ( \II - \HH Q_i(\HH; \lambda_{\HH}^{\pm}) \big ) (\xx_0-\widetilde{\xx}) + Q_i(\HH; \lambda_{\HH}^{\pm}) \tfrac{\AA^T \eeta}{n} \big )  - \tfrac{\AA^T \eeta}{n} \big ]\\
    &= \xx_0-\widetilde{\xx}+ \HH \sum_{i=0}^k c_{ki}  ( \II - \HH Q_i(\HH; \lambda_{\HH}^{\pm})) (\xx_0-\widetilde{\xx}) + \sum_{i=0}^k c_{ki} \big ( \HH Q_i(\HH; \lambda_{\HH}^{\pm}) - \II \big ) \tfrac{\AA^T \eeta}{n} \\
    &= \underbrace{\Big [\II - \HH Q_{k+1}(\HH; \lambda_{\HH}^{\pm}) \Big ]}_{=P_{k+1}(\HH, \lambda_{\HH}^{\pm})} (\xx_0-\widetilde{\xx}) + Q_{k+1}(\HH; \lambda_{\HH}^{\pm})  \tfrac{\AA^T \eeta}{n}\,,
\end{align*}
where in the second identity we have used the induction hypothesis and in the last one the recursive definition of $Q_{k+1}$.
\end{proof}

\subsection{Examples of residual polynomials.}\label{sec:Ex_polynomials}
Motivated by the identity linking the error and the residual polynomial in Proposition~\ref{prop: polynomials_methods}, 
 we derive the residual polynomials for some well-known optimization methods. Some of these residual polynomials are known but some, like Nesterov's accelerated methods, appear to be new. 

\paragraph{Gradient descent.} Due to the simplicity in the recurrence of the iterates for gradient descent, its residual polynomials $P_k$ and $Q_k$ are explicit. Take for example the typical step size $\alpha = \tfrac{1}{\lambda_{\HH}^+}$. Then iterates on \eqref{eq:LS} follow the recursion
\begin{equation}
    \xx_k - \widetilde{\xx} = \xx_{k-1} - \widetilde{\xx} - \alpha \nabla f(\xx_{k-1}) = \big ( \II - \alpha \HH \big )(\xx_{k-1}-\widetilde{\xx}) + \alpha \tfrac{\AA^T \eeta}{n}\,. 
\end{equation}
Applying Proposition~\ref{prop: polynomials_methods} to this recurrence, we obtain the following polynomials:
\begin{equation} \begin{gathered}
    P_k(\lambda; \alpha^{-1}) = (1-\alpha \lambda)^k, \quad
    Q_k(\lambda; \alpha^{-1}) = \alpha \sum_{i=0}^{k-1} (1-\alpha \lambda)^i \quad \text{with $Q_0(\lambda) = 0$}. 
    \end{gathered}
\end{equation}

\paragraph{Nesterov's accelerated method.} Nesterov's accelerated method \citep{nesterov2004introductory} and its variant FISTA \citep{Beck2009Fast} generate iterates on \eqref{eq:LS} satisfying the recurrence 
\begin{equation} \begin{gathered}
\xx_{k+1}-\widetilde{\xx} = (1 + \beta_{k-1}) (I- \alpha \HH) (\xx_k-\widetilde{\xx}) - \beta_{k-1} (I - \alpha \HH)(\xx_{k-1}-\widetilde{\xx}) + \alpha \cdot \tfrac{\AA^T \eeta}{n},\\
\text{where} \quad \alpha = \frac{1}{\lambda_{\HH}^+} \quad \text{and} \quad \beta_k = \begin{cases}
    \tfrac{\sqrt{\lambda_{\HH}^+} - \sqrt{\lambda_{\HH}^-}}{\sqrt{\lambda_{\HH}^+} + \sqrt{\lambda_{\HH}^-} }, &\text{if $\lambda_{\HH}^- \neq 0$}\\
    \frac{k}{k+3}, & \text{if $\lambda_{\HH}^- = 0$}\,,
    \end{cases}
\end{gathered}
\end{equation}
with initial vector $\xx_0 \in \RR^d$ and $\xx_1 = \xx_0-\alpha \nabla f(\xx_0)$. Unrolling the recurrence, we can obtain an explicit formula for the corresponding polynomials 
\begin{equation} \begin{gathered} \label{eq:Nesterov_polynomial_main}
    P_{k+1}(\lambda; \lambda_{\HH}^{\pm}) = (1+\beta_{k-1}) (1-\alpha \lambda) P_k(\lambda; \lambda_{\HH}^{\pm}) - \beta_{k-1}(1-\alpha \lambda) P_{k-1}(\lambda; \lambda_{\HH}^{\pm})\\
    \text{with} \quad P_0(\lambda; \lambda_{\HH}^{\pm}) = 1 \quad \text{and} \quad P_1(\lambda; \lambda_{\HH}^{\pm}) = 1-\alpha \lambda\\
    Q_{k+1}(\lambda; \lambda_{\HH}^{\pm}) = (1+\beta_{k-1})(1-\alpha \lambda) Q_k(\lambda; \lambda_{\HH}^{\pm}) - \beta_{k-1} (1 - \alpha \lambda) Q_{k-1}(\lambda; \lambda_{\HH}^{\pm}) + \alpha \\
    \text{with} \quad Q_0(\lambda; \lambda_{\HH}^{\pm}) = 0 \quad \text{and} \quad Q_1(\lambda; \lambda_{\HH}^{\pm}) = \alpha\,.
    \end{gathered}
    \end{equation}
We derive the polynomials $P_k$ explicitly in Appendix \ref{apx: Nesterov_accelerated_method}. When $\lambda_{\HH}^- > 0$ (strongly convex), the polynomial $P_k$ is given by 
\begin{gather}  P_k(\lambda; \lambda^{\pm}_{\HH}) = \tfrac{2\beta}{1+\beta} (\beta (1-\alpha \lambda))^{k/2} T_k \left ( \tfrac{1+\beta}{2 \sqrt{\beta}}  \sqrt{1-\alpha \lambda} \right ) + \left (1 - \tfrac{2\beta}{1+\beta} \right ) (\beta (1-\alpha \lambda))^{k/2} U_k \left (\tfrac{1+\beta}{2 \sqrt{\beta}} \sqrt{1-\alpha \lambda} \right ), \nonumber \\ \text{where $\beta = 
    \tfrac{\sqrt{\lambda_{\HH}^+} - \sqrt{\lambda_{\HH}^-}}{\sqrt{\lambda_{\HH}^+} + \sqrt{\lambda_{\HH}^-} }$, \quad and \quad $\alpha = \frac{1}{\lambda_{\HH}^+}$}\,,  \label{eq:Nesterov_poly_sc}
    \end{gather}
where $T_k$ and $U_k$ the Chebyshev polynomials of the 1st and 2nd-kind respectively. When the smallest eigenvalue of $\HH$ is equal to $0$ (non-strongly convex setting) the polynomial $P_k$ is given by 

\begin{equation} \label{eq: Nesterov_Legendre}
    P_k(\lambda; \lambda_{\HH}^{\pm}) = \frac{2(1-\alpha\lambda)^{(k+1)/2}}{k \alpha \lambda} \left ( \sqrt{1-\alpha \lambda} \, L_k(\sqrt{1-\alpha \lambda}) - L_{k+1}(\sqrt{1-\alpha \lambda}) \right )\,,
\end{equation}
where $L_k$ are the Legendre polynomials.

\begin{wrapfigure}[16]{r}{0.45\textwidth}
\vspace{-0.5cm}
 \centering
     \includegraphics[scale = 0.2]{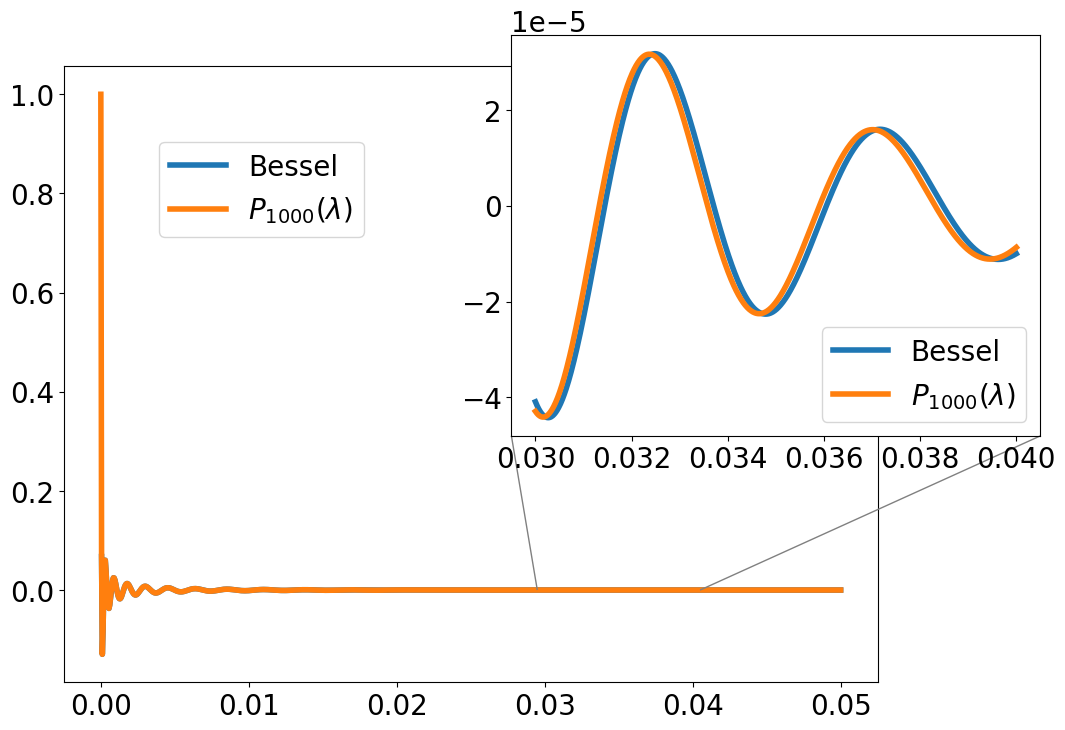}
     \caption{\textbf{Bessel approx. of Nesterov's (convex) poly.} For small $\lambda$, the Bessel approx. (blue) in \eqref{eq:Bessel_asymptotic_main} and Nesterov's (convex) poly. (orange) are indistinguishable. Only when $\lambda$ is far from zero that one sees any, albeit minor, differences.}
    \label{fig:Bessel}
\end{wrapfigure}
Working directly with the polynomial in \eqref{eq: Nesterov_Legendre} will prove difficult. As such, we derive an asymptotic expression for this polynomial. Nesterov's polynomial satisfies in a sufficiently strong sense
\begin{equation} \label{eq:Bessel_asymptotic_main}
    P_k(\lambda; \lambda_{\HH}^{\pm}) \sim \frac{2J_1(k\sqrt{\alpha \lambda})}{ k\sqrt{\alpha \lambda}} e^{-\alpha \lambda k / 2}
\end{equation}
where $J_1$ is the Bessel function of the first kind. A derivation of this can be found in Appendix~\ref{apx: Nesterov_accelerated_cvx}. Let $f(t,z) \defas P_{tn}(z n^{-2}; \lambda^{\pm}_{\HH})$.  Then the recurrence in \eqref{eq:Nesterov_polynomial_main} becomes a discrete approximation to the initial value problem
\[ \partial_{tt} f + \frac{3}{t} \partial_t f + z f = 0, \, \, f(t,0) = 1 \, \, \text{and} \, \, \partial f_t(t,0) = 0,\]
which bears a strong resemblance to the differential equation model for Nesterov's accelerated method in \citep{su2016differential}.
The solution to this initial value problem is $\frac{2 J_1(k \sqrt{\alpha \lambda})}{k \sqrt{\alpha \lambda}}$. Our result in \eqref{eq:Bessel_asymptotic_main}, not derived using this differential equation, yields an even tighter result for Nesterov's accelerated method by including the exponential. 


\paragraph{Polyak momentum algorithm.} We aim to derive the residual polynomials for the Polyak momentum algorithm (a.k.a Heavy-ball method) \citep{Polyak1962Some}. The Polyak momentum algorithm takes as arguments the largest and smallest eigenvalues of $\HH$ and iterates as follows
\begin{equation} \begin{gathered}
\xx_{k+1}-\widetilde{\xx} = \xx_k-\widetilde{\xx} + m (\xx_{k-1}-\widetilde{\xx}-(\xx_{k}-\widetilde{\xx})) + \alpha \nabla f(\xx_{k}),\\
\xx_0 \in \RR^d, \quad \xx_1-\widetilde{\xx} = \xx_0-\widetilde{\xx}-\tfrac{2}{\lambda_{\HH}^+ + \lambda_{\HH}^-} \nabla f(\xx_0)\\
\text{where $m = - \left ( \tfrac{\sqrt{\lambda_{\HH}^+} - \sqrt{\lambda_{\HH}^-}}{\sqrt{\lambda_{\HH}^+} + \sqrt{\lambda_{\HH}^-}} \right )^2$ and $\alpha = -\mfrac{4}{(\sqrt{\lambda_{\HH}^-}+\sqrt{\lambda_{\HH}^+})^2}$}. 
\end{gathered} \end{equation}
Using these initial conditions, the residual polynomials for Polyak momentum satisfy
\begin{equation} \begin{gathered} P_{k+1}(\lambda; \lambda^{\pm}_{\HH}) = (1-m + \alpha\lambda) P_k(\lambda; \lambda^{\pm}_{\HH}) + m P_{k-1}(\lambda; \lambda_{\HH}^{\pm}),\\
\text{with} \qquad P_0(\lambda; \lambda_{\HH}^{\pm}) = 1, \qquad P_1(\lambda; \lambda^{\pm}_{\HH}) = 1 - \tfrac{2}{\lambda_{\HH}^+ + \lambda_{\HH}^-} \lambda\\
\text{and} \qquad Q_{k+1}(\lambda; \lambda_{\HH}^{\pm}) = (1-m + \alpha\lambda) Q_k(\lambda; \lambda_{\HH}^{\pm}) + m Q_{k-1}(\lambda; \lambda_{\HH}^{\pm}) - \alpha,\\ 
\text{with} \qquad Q_0(\lambda; \lambda_{\HH}^{\pm}) = 0, \qquad Q_1(\lambda; \lambda^{\pm}_{\HH}) = \tfrac{2}{\lambda_{\HH}^+ + \lambda_{\HH}^-}.
\end{gathered}
\end{equation}
By recognizing this three-term recurrence as Chebyshev polynomials, we can derive an explicit representation for $P_k$ namely
\begin{equation} \begin{gathered}
    P_k(\lambda; \lambda_{\HH}^{\pm}) = \left ( \tfrac{\sqrt{\lambda_{\HH}^+}-\sqrt{\lambda_{\HH}^-}}{\sqrt{\lambda_{\HH}^+} + \sqrt{\lambda_{\HH}^-}} \right )^k \big [ \tfrac{ ( \sqrt{\lambda_{\HH}^+}-\sqrt{\lambda_{\HH}^-})^2}{\lambda_{\HH}^+ + \lambda_{\HH}^-} \cdot T_k(\sigma(\lambda)) + \tfrac{2 \sqrt{\lambda_{\HH}^- \lambda_{\HH}^+}}{\lambda_{\HH}^+ + \lambda_{\HH}^-} \cdot U_k(\sigma(\lambda)) \big ] \\
    \text{where $T_k(x)$ and $U_k(x)$ are the Chebyshev polynomials of the 1st and 2nd-kind respectively}\\
    \text{and \quad $\sigma(\lambda) = \tfrac{\lambda_{\HH}^+ + \lambda_{\HH}^- -2 \lambda}{\lambda_{\HH}^+ - \lambda_{\HH}^-}$.}
    \end{gathered}
\end{equation}


\begin{figure*}[t!]
\begin{center}
    \includegraphics[scale = 0.2]{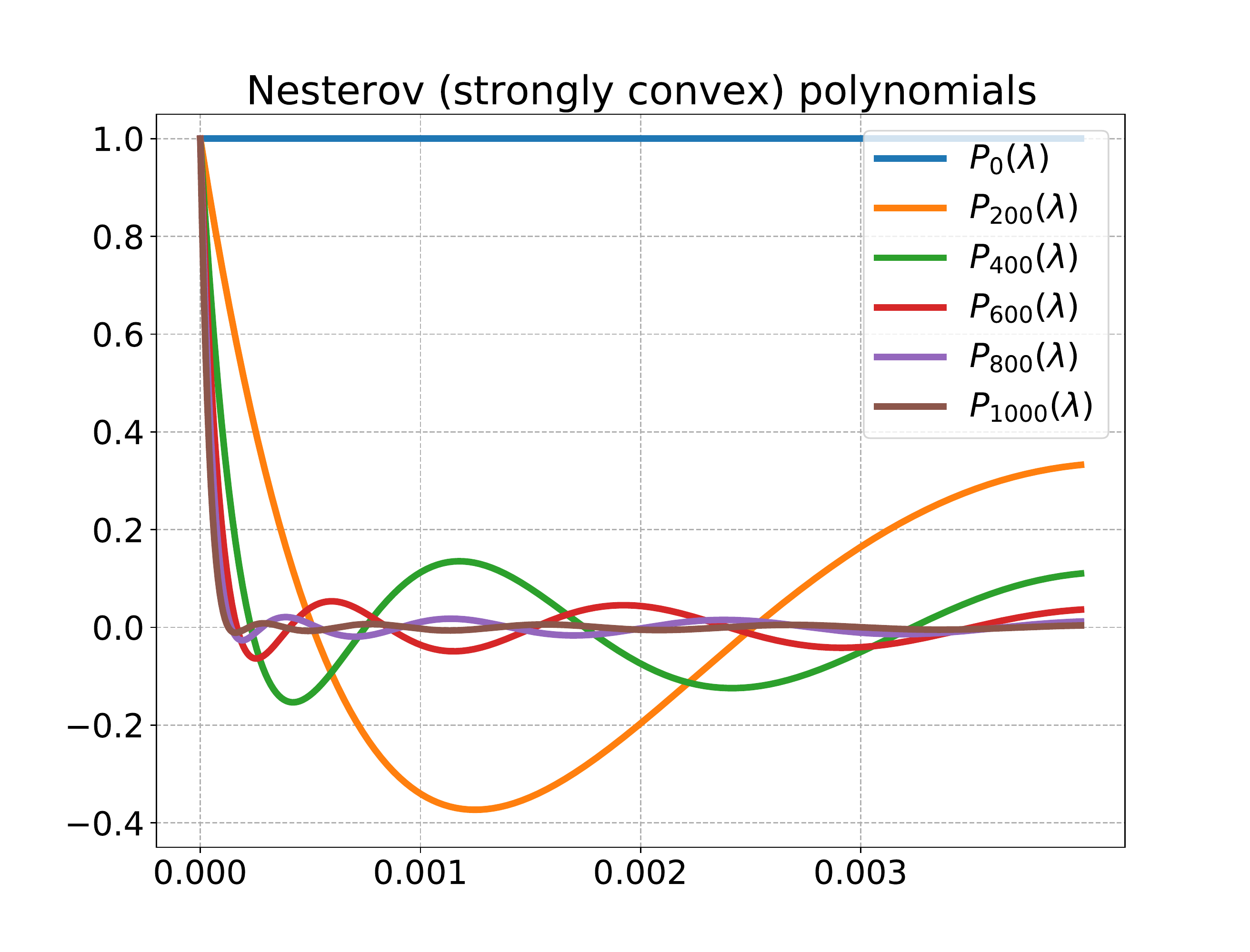}
    \hspace{-0.5cm}\includegraphics[scale = 0.2]{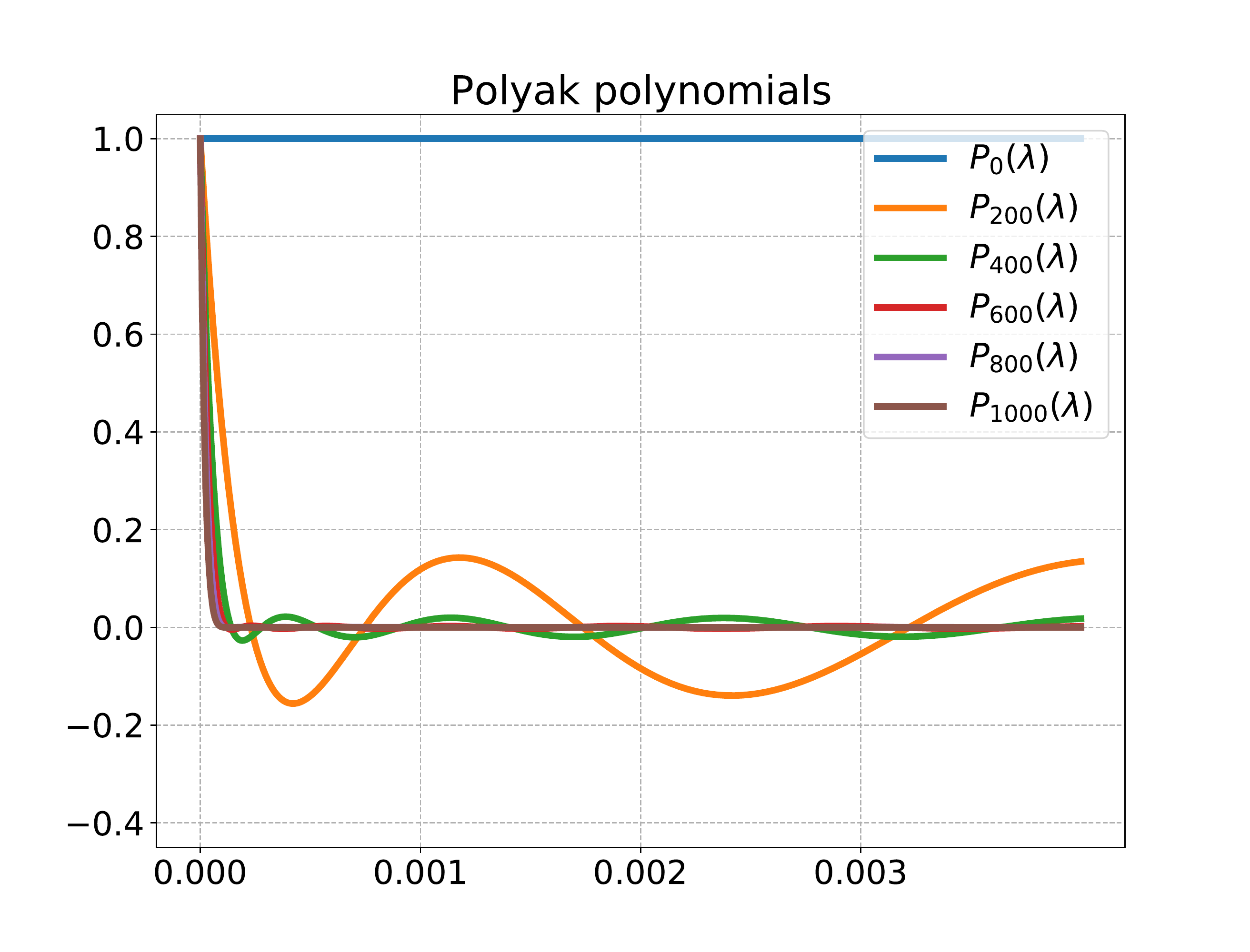}
       \hspace{-0.5cm}\includegraphics[scale = 0.2]{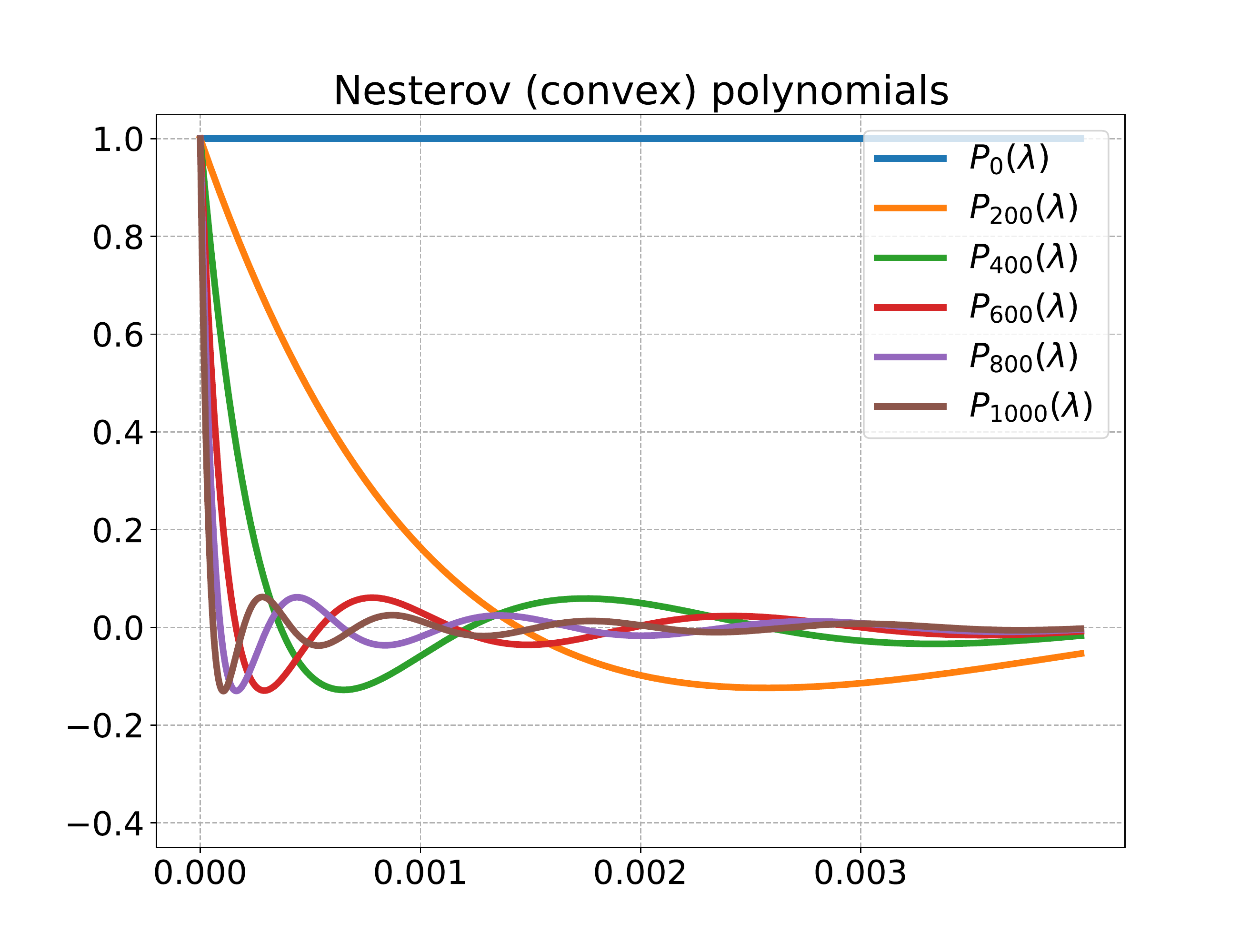}
\end{center}
\caption{{\bfseries Residual polynomials.} The oscillations in the polynomials for Nesterov's accelerated method (convex) are pronounced near zero compared with the other methods. In fact, both Nesterov (strongly convex) and Polyak momentum polynomials decay quite rapidly to the zero polynomial. To see these oscillations in the figures one needs to have a badly conditioned matrix (condition number 40,000). The slower decay to zero in the residual polynomials for Nesterov (strongly convex) as compared with Polyak's momentum suggest a worst rate of convergence. 
} \label{fig:polynomials}
\end{figure*}

\subsection{Properties of residual polynomials} 
In the following sections, it will be convenient to know some general properties of residual polynomials. Particularly, the polynomials, $\lambda^2 P_k^2(\lambda; \lambda^{\pm})$ and $\lambda P_k^2(\lambda; \lambda^{\pm})$ are uniformly bounded in $k$ and that these polynomials goes to zero on some fixed support $[\lambda^-, \lambda^+]$. The importance of these facts are twofold. First, these polynomials appear in the formula for the expected gradient, Theorem~\ref{thm: concentration_main}. Second, we use the boundedness and convergence properties in the proof of halting time universality, Theorem~\ref{thm: Halting_time_main}. If one \textit{a priori} knows an explicit expression for these polynomials, then these properties are easily deduced. However when such an expression does not exist, we still can conclude these properties hold provided that the algorithm is \textit{convergent}. 

\begin{definition}[Convergent algorithms] \rm{We say a gradient-based method is \textit{(strongly) convergent} if for every matrix $\AA$ such that  $(\AA^T \AA \succ 0)$ $\AA^T\AA \succeq 0$ and any vectors $\bb$ and $\xx_0$, we have that the sequence of iterates generated by the algorithm starting at $\xx_0$ satisfies $\|\nabla f(\xx_k)\|^2 \to 0$ as $k \to \infty$ and there exists constants $C, \widetilde{C}$ depending on $\lambda_{\HH}^+$ and $\lambda_{\HH}^-$ such that
\begin{equation} \begin{gathered} \label{eq: boundedness_grad} \|\nabla f(\xx_k)\|^2 \le C
\big ( f(\xx_0)-f(\xx^{\star}) + \|\xx_0-\xx^{\star}\|^2 \big )\\
f(\xx_k)-f(\xx^{\star}) \le \widetilde{C} (f(\xx_0)-f(\xx^{\star}) + \|\xx_0-\xx^{\star}\|^2)
\end{gathered}
\end{equation}
where $\xx^{\star}$ in the optimum of \eqref{eq:LS}. 
}
\end{definition}

\begin{remark}[Minimal norm solutions] \label{rmk: minimal_norm} For any gradient-based method, the iterates generated by the algorithm on the least squares problem \eqref{eq:LS} satisfy $\xx_k \in \xx_0 + \ospan\{ \nabla f(\xx_0), \hdots, \nabla f(\xx_{k-1})\} \subseteq \xx_0 + \text{\rm Null}(\AA)^{\perp}$. If the algorithm converges to some $\xx^{\star}$ and we have that $\AA^T \AA \xx^{\star} = \AA^T \bb$, then the solution $\xx^{\star}$ is independent of the algorithm in the following sense
\begin{equation}
    \xx^{\star} = \argmin_{\{ \xx \, : \, \AA^T \AA \xx = \AA^T \bb \}} \|\xx_0-\xx\|^2_2. 
\end{equation}
In particular when $\xx_0 \in \text{\rm Null}(\AA)^{\perp}$, the optimum $\xx^{\star}$ is the minimal norm solution. See \textit{e.g.}, \cite{gunasekar2018characterizing, wilson2017marginal} and references therein.
\end{remark}

\begin{remark} All the algorithms discussed in Section~\ref{sec:Ex_polynomials} are convergent. 
\end{remark}

The following lemma shows that convergent algorithms have residual polynomials which go to $0$ as $k \to \infty$ on compact subsets of the positive real line. In essence if optimality measures go to zero, then so must the residual polynomial. 

\begin{lemma}[Convergent algorithms $\Rightarrow$ Residual polynomials $\to 0$] \label{lem: convergent_algorithm}  Suppose the algorithm $\mathcal{A}$ is a (strongly) convergent gradient-based method. Fix positive constants $0 \le \lambda^- < \lambda^+$ for a convergent algorithm and constants $0 < \lambda^- < \lambda^+$ if one has a strongly convergent algorithm. The residual polynomial, $P_k$, for the algorithm $\mathcal{A}$ satisfies
    \[ \lim_{k \to \infty} \lambda^2 P_k^2(\lambda; \lambda^{\pm}) = 0 \quad \text{and} \quad \lim_{k \to \infty} \lambda P_k^2(\lambda; \lambda^{\pm}) = 0 \quad \text{for all $\lambda \in [\lambda_-, \lambda_+]$}. \]
\end{lemma}

\begin{proof} 
Suppose we consider the noiseless setting where $\eeta = (0,0,0)^T$ in the generative model so that $\AA \widetilde{\xx} = \bb$. Fix a constant $\lambda \in [\lambda^-, \lambda^+]$ and define the following matrix and vectors
\begin{equation} \label{eq:matrix_AA} \AA = \vast [ \begin{matrix} \sqrt{3 
\lambda^+} & 0 & 0
\vspace{-0.5cm}\\
0 & \sqrt{3 \lambda} & 0 \vspace{-0.5cm}\\
0 & 0 & \sqrt{3\lambda^-} 
\end{matrix} \vast ], \qquad \xx_0-\widetilde{\xx} = ( 0, 1, 0 )^T, \quad \text{and} \quad \eeta = ( 0, 0, 0 )^T.
\end{equation}
A simple computation shows that $\HH = \tfrac{1}{3}\AA^T\AA = \text{diag}(\lambda^+, \lambda, \lambda^-)$. Because the method is (strongly) convergent, the algorithm converges for these choices of $\HH$ and $\xx_0-\widetilde{\xx}$. Moreover, we know that $\nabla f(\xx_k) = \HH (\xx_k-\widetilde{\xx})$ and by Proposition~\ref{prop: polynomials_methods}, the vector $\xx_k-\widetilde{\xx} = P_k(\HH; \lambda^{\pm}) (\xx_0-\widetilde{\xx})$. Therefore we have that 
\begin{equation} \label{eq:convergent_stuff_1} \lim_{k \to \infty} \lambda^2 P_k^2(\lambda; \lambda^{\pm}) = \lim_{k \to\infty} (\xx_0-\widetilde{\xx})^T \HH^2 P_k^2(\HH; \lambda^{\pm}) (\xx_0-\widetilde{\xx}) = \lim_{k \to \infty} \|\nabla f(\xx_k)\|^2 = 0.
\end{equation}

Similarly, we consider the same matrix $\AA$ as in \eqref{eq:matrix_AA} but instead a pure noise setting,
\begin{equation}   \xx_0-\widetilde{\xx} = ( 0, 0, 0)^T, \quad \text{and} \quad \eeta = ( 0, \sqrt{3}, 0)^T.
\end{equation}
As before, the matrix $\HH = \text{diag}(\lambda^+, \lambda, \lambda^-)$. By Proposition~\ref{prop: polynomials_methods}, the iterates $\xx_k-\widetilde{\xx} = Q_k(\HH; \lambda^{\pm}) \frac{\AA^T \eeta}{3}$ as $\xx_0-\widetilde{\xx} = \bm{0}$. With this, the gradient equals
\[ \nabla f(\xx_k) = \HH(\xx_k- \widetilde{\xx}) - \tfrac{\AA^T \eeta}{3} = \big [ \HH Q_k(\HH; \lambda^{\pm}) - \II \big ] \tfrac{\AA^T \eeta}{3} = - P_k(\HH; \lambda^{\pm}) \tfrac{\AA^T \eeta}{3}. \]
Here, again, we used Proposition~\ref{prop: polynomials_methods}. A (strongly) convergent method has the following 
\begin{equation} 
    \lim_{k \to \infty} \lambda P_k^2(\lambda; \lambda^{\pm}) = \lim_{k \to \infty} \tfrac{\eeta^T \AA}{3} P_k^2(\HH; \lambda^{\pm}) \tfrac{\AA^T \eeta}{3} = \lim_{k \to \infty} \|\nabla f(\xx_k)\|^2 = 0.
\end{equation}
This completes the result.  
\end{proof}

The following lemma shows that the residual polynomials are uniformly bounded over $k$ on any compact subset of the positive real line.


\begin{lemma}[Convergent algorithms $\Rightarrow$ boundedness of $P_k$] \label{lem: convergent_bounded} Suppose $\mathcal{A}$ is a (strongly) convergent algorithm with residual polynomial $P_k$. 
Under the assumptions of Lemma~\ref{lem: convergent_algorithm},
\[ \max_{k, \lambda \in [\lambda^-, \lambda^+]} \lambda^2 P_k^2(\lambda; \lambda^{\pm}) \le B \quad \text{and} \quad \max_{k, \lambda \in [\lambda^-, \lambda^+]} \lambda P_k^2(\lambda; \lambda^{\pm}) \le \widetilde{B},\]
for some constants $B, \widetilde{B} > 0$. 
\end{lemma}

\begin{proof} Suppose we consider the noiseless setting $\eeta = \bm{0}$ in the generative model \eqref{eq:LS} so that $\AA \widetilde{\xx} = \bb$. It then follows that $f(\xx^{\star}) = 0$ where $\xx^{\star}$ is the optimum. A simple computation using Proposition~\ref{prop: polynomials_methods} shows that for all $k \ge 0$
\begin{equation} \begin{aligned} \label{eq: stuff_10}
    f(\xx_k) - f(\xx^{\star}) &= \tfrac{1}{2n} \|\AA (\xx_k-\widetilde{\xx})\|^2 = \tfrac{1}{2} (\xx_k-\widetilde{\xx})^T \HH (\xx_k-\widetilde{\xx})\\
    &= \tfrac{1}{2}(\xx_0-\widetilde{\xx})^T \HH P_k^2(\HH; \lambda_{\HH}^{\pm}) (\xx_0-\widetilde{\xx}).
\end{aligned}
\end{equation}
Next consider the matrix $\AA$ and vectors $\xx_0-\widetilde{\xx}$ and $\eeta$ as in \eqref{eq:matrix_AA} with the initial iterate $\xx_0 = (0,0,0)^T$. We consider cases: suppose $\lambda^- = 0$. Fix a constant $\lambda \in [\lambda^-, \lambda^+]$. It follows from our choice of $\AA$, $\xx_0$, $\widetilde{\xx}$, and $\eeta$ that the vector $\AA \widetilde{\xx} = \bb = (0, -\sqrt{3 \lambda}, 0)^T$ and by \eqref{eq: stuff_10} that
\begin{equation} \label{eq: stuff_11} f(\xx_0)-f(\xx^{\star}) = \tfrac{1}{2} \lambda P_0^2(\lambda; \lambda^{\pm}). 
\end{equation}
The solution set $\{\xx \, : \, \AA^T\AA \xx = \AA^T\bb\} = \{(0,-1, a)^T : a \in \RR\}$ if $\lambda > 0$ and otherwise it equals $\{(0, a, b)^T : a,b \in \RR\}$ if $\lambda = 0$. From Remark~\ref{rmk: minimal_norm}, we have that $\displaystyle \xx^{\star} = \argmin_{\AA^T\AA \xx = \AA^T \bb} \|\xx-\xx_0\|^2$ and thus we deduce that 
\[ \xx^{\star} = \begin{cases}
(0,-1,0)^T, & \text{if $\lambda > 0$}\\
(0, 0, 0)^T & \text{if $\lambda = 0$}.
\end{cases}
\]
In both cases, we have that $\|\xx_0-\xx^{\star}\|^2 \le 1$. Therefore using the boundedness assumption \eqref{eq: boundedness_grad} and the expression for the gradient in \eqref{eq:convergent_stuff_1}, we have that 
\begin{align*} \sup_{k, \, \lambda \in [\lambda^-, \lambda^+] } \! \! \! \! \! \! \lambda^2 P_k^2(\lambda; \lambda^{\pm}) &= \! \! \! \! \! \! \sup_{k, \, \lambda \in [\lambda^-, \lambda^+] } \! \! \! \! \! \! \|\nabla f(\xx_k)\|^2  \le \! \! \! \! \! \! \sup_{k, \, \lambda \in [\lambda^-, \lambda^+] } \! \! \! \! \! \! C ( f(\xx_0)-f(\xx^{\star}) + \|\xx_0-\xx_{\star}\|^2 )\\
&\le \! \! \! \! \! \! \sup_{k, \, \lambda \in [\lambda^-, \lambda^+] } \! \! \! \! \! \! C  \big ( \tfrac{1}{2} \lambda P_0^2(\lambda; \lambda^{\pm}) + 1 \big ) \le B.
\end{align*}
Here we used that the distance to the optimum $\|\xx_0-\xx^{\star}\|^2 \le 1$ and the polynomial in \eqref{eq: stuff_11} is bounded on a compact set. 

Now we suppose that $\lambda^- > 0$. As above, we use the same matrix $\AA$ and vectors $\xx_0-\widetilde{\xx}$ and $\eeta$ as in \eqref{eq:matrix_AA} and, in addition, we set $\xx_0 = (0,0,0)^T$. In this situation, the matrix $\AA$ is invertible and $\xx^{\star} = (0,-1,0)^T$. Hence both \eqref{eq: stuff_11} and $\|\xx_0-\xx^{\star}\|^2 \le 1$ holds. Using the boundedness assumption on function values \eqref{eq: boundedness_grad} and the expression for the function values in \eqref{eq: stuff_10}, we deduce
\begin{align*}
    \sup_{k, \, \lambda \in [\lambda^-, \lambda^+] } \! \! \! \! \! \! \lambda P_k^2(\lambda; \lambda^{\pm}) = \! \! \! \! \! \! \sup_{k, \, \lambda \in [\lambda^-, \lambda^+] } \! \! \! \! \! \! f(\xx_k)-f(\xx^{\star})
    &\le \! \! \! \! \! \! \sup_{k, \, \lambda \in [\lambda^-, \lambda^+] } \! \! \! \! \! \! \widetilde{C} ( f(\xx_0)-f(\xx^{\star}) + \|\xx_0-\xx^{\star}\|^2 )\\
&\le \! \! \! \! \! \! \sup_{k, \, \lambda \in [\lambda^-, \lambda^+] } \! \! \! \! \! \! \widetilde{C} \big ( \tfrac{1}{2} \lambda P_0^2(\lambda; \lambda^{\pm}) + 1 \big ) \le \widetilde{B}.
\end{align*}
The result immediately follows. 

\end{proof}

\section{Halting time is almost deterministic} \label{sec: halting_time}

In this section we develop a framework for the average-case analysis and state a main result of this paper: the concentration of the halting time.
We define the halting time $T_{\varepsilon}$ as the first iteration at which the gradient falls below some predefined $\varepsilon$:
\begin{equation} \label{eq:something_2} T_{\varepsilon} \defas \inf \, \{ k > 0 \, : \, \|\nabla f(\xx_k)\|^2 \le \varepsilon\}\,. 
\end{equation}

Our main result (Theorem~\ref{thm: Halting_time}) states that this halting time is predictable for almost all high-dimensional data, or more precisely,
\begin{equation}
\lim_{d \to \infty} \Pr(T_{\varepsilon} = \text{constant}) = 1\,.
\end{equation}
Furthermore, we provide an implicit expression for this constant, otherwise known as the average complexity, and in Tables~\ref{tab:comparison_worst_avg_cvx} and \ref{tab:comparison_worst_avg_str_cvx} an explicit expression under further assumptions. The rest of this section provides a proof of this result.

\subsection{First-order methods as polynomials} \label{apx: GD_poly}

\begin{proposition}[Residual polynomials and gradients] \label{prop:gradient_polynomial} Suppose the iterates $\{\xx_k\}_{k=0}^\infty$ are generated from a gradient based method. Let $\{P_k\}_{k=0}^\infty$ be a sequence of polynomials defined in \eqref{eq:recursive_noise_poly}. Then the following identity exists between the iterates and its residual polynomial,
\begin{equation} \begin{gathered} \label{eq:grad_optimality_cond_app}
    \| \nabla f(\xx_k) \|^2 = (\xx_0-\widetilde{\xx})^T \HH^2 P_k^2(\HH; \lambda_{\HH}^{\pm}) (\xx_0-\widetilde{\xx}) + \tfrac{\eeta^T \AA}{n} P_k^2(\HH; \lambda_{\HH}^{\pm}) \tfrac{\AA^T \eeta}{n} \\
     -2(\xx_0-\widetilde{\xx})^T \HH P_k^2(\HH; \lambda_{\HH}^{\pm}) \tfrac{\AA^T \eeta}{n}. \nonumber
\end{gathered}
\end{equation}
\end{proposition}

\begin{proof} The gradient in \eqref{eq:LS} is given by the expression $\nabla f(\xx_k) = \HH (\xx_k-\widetilde{\xx}) - \frac{\AA^T \eeta}{n}$. The result follows immediately by plugging in \eqref{eq:recursive_noise_poly_1} into the formula for the gradient and using the relationship that $\HH^2 Q_k^2(\HH; \lambda_{\HH}^{\pm}) -2 \HH Q_k(\HH; \lambda_{\HH}^{\pm}) + \II = (\II - \HH Q_k(\HH; \lambda_{\HH}^{\pm}))^2 = P_k^2(\HH; \lambda_{\HH}^{\pm})$.
\end{proof}

This \textit{equality} for the squared norm of the gradient is crucial for deriving average-case rates. In contrast, worst-case analysis typically uses only \textit{bounds} on the norm. A difficulty with the polynomials $P_k$ and $Q_k$ is that their coefficients depend on the largest and smallest eigenvalue of $\HH$, and hence are random. We can remove this randomness thanks to Assumption~\ref{assumption: spectral_density}, replacing $\lambda_\HH^+$ and $\lambda_\HH^-$ with the top (bottom) edge of the support of $\mu$, denoted by $\lambda^+$ and $\lambda^-$, without loss of generality. 

\begin{proposition}[Remove randomness in coefficients of polynomial] \label{proposition: norm} Suppose Assumption~\ref{assumption: spectral_density} holds. Fix any $k$-degree polynomial $\widetilde{P}_k$ whose coefficients depend continuously on the largest and smallest eigenvalues of $\HH$. Then the following hold
\begin{equation} \| \widetilde{P}_k(\HH; \lambda_{\HH}^{\pm}) - \widetilde{P}_k(\HH; \lambda^{\pm})\|^2_{\text{\rm op}} \Prto[d] 0\,. \end{equation}
\end{proposition}

\begin{proof} Fix any $\varepsilon, \delta > 0$. Let $c_i(\cdot)$ where $i = 0, \hdots, k$ be the coefficients associated with the term of degree $i$ in $\widetilde{P}_k(\HH; \cdot)$. For each $i$, the continuity of $c_i(\cdot)$ implies there exists $\delta_{\varepsilon} > 0$ such that 
\begin{equation} \text{whenever} \quad \|(\lambda_{\HH}^+, \lambda_{\HH}^-)-(\lambda^+, \lambda^-)\| \le \delta_{\varepsilon} \quad \Rightarrow \quad |c_i(\lambda_{\HH}^{\pm})-c_i(\lambda^{\pm})| \le \frac{\varepsilon}{4 (4\lambda^+)^i}\,. \end{equation}
For sufficiently large $d$, Assumption~\ref{assumption: spectral_density} implies $\Pr \big (|\lambda_{\HH}^+-\lambda^+| > \min\{\tfrac{\delta_{\varepsilon}}{2}, \lambda^+\} \big ) \le \tfrac{\delta}{2}$ and $\Pr \big (|\lambda_{\HH}^--\lambda^-| > \min\{ \tfrac{\delta_{\varepsilon}}{2}, \lambda^+\} \big ) \le \tfrac{\delta}{2}$.
With this, we  define the event 
$\mathcal{S} = \{ | \lambda_{\HH}^+ - \lambda^+| \le \min\{ \tfrac{\delta_{\varepsilon}}{2}, \lambda^+ \} \} \cap \{ | \lambda_{\HH}^- - \lambda^-| \le \min\{ \tfrac{\delta_{\varepsilon}}{2}, \lambda^+ \} \}.$ We have for all sufficiently large $d$
\begin{align}
    \Pr \big ( \|\widetilde{P}_k(\HH; \lambda_{\HH}^{\pm})-&\widetilde{P}_k(\HH; \lambda^{\pm}) \|_\text{op} > \varepsilon \big )
    = \Pr \big ( \mathcal{S} \cap  \big \{ \|\widetilde{P}_k(\HH; \lambda_{\HH}^{\pm})-\widetilde{P}_k(\HH; \lambda^{\pm}) \|_\text{op} > \varepsilon \big \} \big ) \nonumber \\
    &\qquad \qquad + \Pr \big ( \mathcal{S}^c \cap  \big \{ \|\widetilde{P}_k(\HH; \lambda_{\HH}^{\pm})-\widetilde{P}_k(\HH; \lambda^{\pm}) \|_\text{op} > \varepsilon \big \} \big ) \nonumber\\
    &\le \Pr \big ( \mathcal{S} \cap  \big \{ \|\widetilde{P}_k(\HH; \lambda_{\HH}^{\pm})-\widetilde{P}_k(\HH; \lambda^{\pm}) \|_\text{op} > \varepsilon \big \} \big ) + \delta. \label{eq:rand_feat_blah_1}
\end{align}
Here we used that $\Pr \big ( \mathcal{S}^c \cap  \big \{ \|\widetilde{P}_k(\HH; \lambda_{\HH}^{\pm})-\widetilde{P}_k(\HH; \lambda^{\pm}) \|_\text{op} > \varepsilon \big \} \big ) \le \Pr(\mathcal{S}^c) \le \delta$ for large $d$. Therefore, it suffices to consider the first term in \eqref{eq:rand_feat_blah_1} and show that it is $0$. By construction of the set $\mathcal{S}$, any element in $\mathcal{S}$ satisfies both $\| \HH \|_{\text{op}} \le 2 \lambda^+$ and $|c_i(\lambda_{\HH}^{\pm})- c_i(\lambda^{\pm})| \le \tfrac{\varepsilon}{4 (4\lambda^+)^i}$. Hence on the event $\mathcal{S}$, we have the following
\begin{align}
    \|\widetilde{P}_k(\HH, \lambda_{\HH}^{\pm} ) - \widetilde{P}_k(\HH; \lambda^{\pm})\|_{\text{op}} \le \sum_{i=0}^k | c_i(\lambda_{\HH}^{\pm})-c_i(\lambda^{\pm})| \|\HH\|_{\text{op}}^i \le  \sum_{i=0}^k \frac{ (2\lambda^+)^i  \varepsilon}{4 (4 \lambda^+)^i} \le \frac{\varepsilon}{2}\,.
\end{align}
From this, we deduce that $\Pr \big (\mathcal{S} \cap \{ \| \widetilde{P}_k(\HH; \lambda_{\HH}^{\pm})-\widetilde{P}_k(\HH; \lambda^{\pm})\|_{\text{op}} > \varepsilon \} \big ) = 0$ and the result immediately follows by \eqref{eq:rand_feat_blah_1}.
\end{proof}

The squared norm of the gradient in \eqref{eq:grad_optimality_cond_app} is a quadratic form. In Proposition~\ref{proposition: norm}, we removed the randomness in the coefficients of the polynomial and now we will relate this back to the squared norm of the gradient, and particularly, the quadratic form. The following lemmas state this precisely. 

\begin{lemma} \label{lemma: probability_lemma} Suppose the sequences of non-negative random variables $X_{d}, Y_{d} \ge 0$ satisfy $\mathbb{E}[X_{d}] \le \gamma < \infty$ and $Y_{d} \Prto[d] 0$. Then $X_{d} Y_{d} \Prto[d] 0$.
\end{lemma}

\begin{proof} Fix constants $\varepsilon, \delta > 0$ and suppose we set $\hat{\varepsilon} = \frac{\varepsilon \delta}{2\gamma}$ and $\hat{\delta} = \frac{\delta}{2}$. Because $Y_d$ converges in probability, we have $\Pr(Y_{d} > \hat{\varepsilon}) \le \hat{\delta}$ for sufficiently large $d$. Define the event $\mathcal{S} = \{Y_d \le \hat{\varepsilon} \}$ and decompose the space based on this set $\mathcal{S}$ so that for large $d$
\begin{align*}
    \Pr(X_d Y_d > \varepsilon) = \Pr(\mathcal{S} \cap \{X_d Y_d > \varepsilon\}) + \Pr(\mathcal{S}^c \cap \{X_d Y_d > \varepsilon \})
    \le \Pr(\mathcal{S} \cap \{X_d Y_d > \varepsilon\}) + \tfrac{\delta}{2}.
\end{align*}
Here we used that $\Pr(\mathcal{S}^c \cap \{X_d Y_d > \varepsilon\}) \le \Pr(\mathcal{S}^c)$. For the other term, a direct application of Markov's inequality yields 
\begin{align*}
    \Pr(\mathcal{S} \cap \{X_d Y_d > \varepsilon\}) \le \Pr(\mathcal{S} \cap \{\hat{\varepsilon} X_d > \varepsilon\}) \le  \tfrac{\hat{\varepsilon}}{\varepsilon} \cdot \mathbb{E}[X_d] \le \tfrac{\delta}{2}.
\end{align*}
The result immediately follows.
\end{proof}

\begin{lemma}[Remove randomness in coefficients of quadratic form]\label{proposition: remove_norm} Suppose Assumption~\ref{assumption: spectral_density} holds and let the vectors $\ww \in \mathbb{R}^d$ and $\vv \in \mathbb{R}^d$ be i.i.d. satisfying $\EE[\|\ww\|_2^2] = R^2$ and $\EE[\|\vv\|_2^2] = \widetilde{R}^2$ for some constants $R, \widetilde{R} > 0$. 
For any $k$ degree polynomial $\widetilde{P}_k$ whose coefficients depend continuously on $\lambda_{\HH}^+$ and $\lambda_{\HH}^-$, the quadratic form converges in probability
\begin{align*}
    \ww^T & \widetilde{P}_k \left (\HH; \lambda_{\HH}^{\pm} \right ) \vv - \ww^T \widetilde{P}_k \left (\HH; \lambda^{\pm} \right ) \vv \Prto[d] 0.
\end{align*}
\end{lemma}

\begin{proof} Using the Cauchy-Schwarz inequality, it suffices to show that for every $\varepsilon > 0$ we have
\begin{align*}
    \lim_{d \to \infty} \Pr \left (\|\ww\|_2 \cdot \|\vv\|_2 \cdot \big \| \widetilde{P}_k \left ( \HH; \lambda_{\HH}^{\pm} \right ) - \widetilde{P}_k \left ( \HH; \lambda^{\pm} \right )   \big \|_{\text{op}} > \varepsilon \right ) = 0\,. 
\end{align*} 
Define $X_{d} = \|\ww\|_2 \|\vv\|_2$ and $Y_{d} = \big \|\widetilde{P}_k \left ( \HH; \lambda_{\HH}^{\pm} \right ) - \widetilde{P}_k \left ( \HH; \lambda^{\pm} \right )  \big \|_{\text{op}}$. Proposition~\ref{proposition: norm} immediately gives that $Y_{d} \Prto[d] 0$. Next, Cauchy-Schwartz implies 
\[\EE[X_d] = \EE[ \|\ww\|_2 \|\vv\|_2] \le \EE[\|\ww\|_2^2]^{1/2} \EE[\|\vv\|_2^2]^{1/2} = R \widetilde{R}.\]
The result immediately follows after applying Lemma~\ref{lemma: probability_lemma}.
\end{proof}
From Lemma~\ref{proposition: remove_norm} and the expression for the squared norm of the gradient in \eqref{eq:grad_optimality_cond_app}, we can replace the maximum (minimum) eigenvalue $\lambda_{\HH}^+$ $(\lambda_{\HH}^-)$ in \eqref{eq:grad_optimality_cond_app} with the top (bottom) edge of the support of $\mu$, $\lambda^+$ ($\lambda^-$). This followed because the vectors $\xx_0-\widetilde{\xx}$ and $\tfrac{1}{n} \AA^T \eeta$ satisfy $\ww$ and $\vv$ in Lemma~\ref{proposition: remove_norm} and the terms surrounding these vectors in \eqref{eq:grad_optimality_cond_app} are polynomials in $\HH$. 

\subsection{Concentration of the gradient}

Having established the key equation linking the gradient to a polynomial in Proposition~\ref{prop:gradient_polynomial}, we now show that for almost any large model the magnitude of the gradient after $k$ iterations converges to a deterministic value which we denote by $\concentration$. We recall the statement of Theorem~\ref{thm: concentration_main}: 

\noindent \textbf{Theorem.} \rm{(Concentration of the gradient)} \textit{
Under Assumptions~\ref{assumption: Vector} and~\ref{assumption: spectral_density} the norm of the gradient concentrates around a deterministic value:
\begin{equation} \label{eq: something_1} \vspace{0.25cm}
\hspace{-0.28cm} \! \|\nabla f(\xx_k)\|^2 \! \! \Prto[d] \! \! \! \textcolor{teal}{\overbrace{R^2}^{\text{signal}}} \! \! \! \! \int \! { \underbrace{\lambda^2 P_k^2(\lambda; \lambda^{\pm})}_{\text{algorithm}}} \textcolor{mypurple}{\overbrace{\dif\mu}^{\text{model}} }   + \!  \textcolor{purple}{\overbrace{ \widetilde{R}^2} ^{\text{noise}} } \! r \! \! \int \! { \underbrace{\lambda P_k^2(\lambda; \lambda^{\pm})}_{\text{algorithm}}}  \textcolor{mypurple}{\overbrace{ \dif\mu}^{\text{model}} }  \defas \concentration. \! 
\end{equation}
}

Intuitively, the value of $\concentration$ is the expected gradient after first taking the model size to infinity. The above expression explicitly illustrates the effects of the model and the algorithm on the norm of the gradient: the \textcolor{teal}{signal ($R^2$)} and \textcolor{purple}{noise ($\widetilde{R}^2$)}, the {optimization algorithm} which enters into the formula through the polynomial $P_k$, and the \textcolor{mypurple}{model used to generate $\AA$} by means of the measure $\mu$. 

The main tool to prove Theorem~\ref{thm: concentration_main} is the moment method which requires computing explicit expressions for the moments of the norm of the gradient. We summarize this in the following proposition. 
To ease notation in the next few propositions, we define the following matrices and vectors 
\begin{equation} \begin{gathered} \label{eq:blah_10}
    \quad \uu \defas \xx_0-\widetilde{\xx}, \quad \BB \defas \HH^2 P_k^2(\HH; \lambda^{\pm}), \quad \CC \defas  P_k^2(\HH; \lambda^{\pm}),\\
    \text{and} \quad \DD \defas -2 \HH P_k^2(\HH; \lambda^{\pm})
    \end{gathered}
\end{equation}
and let $y_k$ be the quadratic form given by 
\begin{equation}\label{eq: norm_with_noise1}
    y_k \defas \uu^T \BB \uu + \tfrac{1}{n} \uu^T \DD \AA^T \eeta + \tfrac{1}{n^2} \eeta^T \AA \CC \AA^T \eeta. 
\end{equation}
Observe that the value $y_k$ is simply $\|\nabla f(\xx_k)\|^2$ in \eqref{eq:grad_optimality_cond_app} with $\lambda_{\HH}^{\pm}$ replaced with $\lambda^{\pm}$.
\begin{proposition} \label{proposition:conditional} Suppose the matrix $\AA$ and vectors $\xx_0, \widetilde{\xx},$ and $\eeta$ satisfy Assumptions~\ref{assumption: Vector} and \ref{assumption: spectral_density}. Let $P_k$ be the $k$-degree polynomial defined in \eqref{eq:recursive_noise_poly}. Using the notation in \eqref{eq:blah_10} and \eqref{eq: norm_with_noise1}, the following holds for any $\varepsilon > 0$
\begin{equation} \begin{aligned} \label{eq:conditional}
    \Pr \big ( | y_k - \big [ R^2 \text{ \rm tr} \big ( \tfrac{\BB}{d} \big ) &+ \tilde{R}^2 \text{ \rm tr} \big (\tfrac{\CC \HH}{n} \big )\big ] | > \varepsilon  \, \big | \, \HH \big ) \\
    &\le \tfrac{1}{\varepsilon^2} \left ( \tfrac{C-R^4}{d} \text{ \rm tr} \big ( \tfrac{\BB^2}{d} \big ) + \tfrac{\tilde{C}-\tilde{R}^4}{n} \text{ \rm tr} \big ( \tfrac{(\CC \HH)^2}{n} \big ) + \tfrac{ R^2 \tilde{R}^2}{n} \big [ \tfrac{\text{tr}( \DD^2 \HH)}{d} \big ]  \right ).
\end{aligned} \end{equation}
Without loss of generality, we assume that the constants $C$ and $\widetilde{C}$ are large enough such that $C > 3 R^4$ and $\widetilde{C} > 3 \widetilde{R}^4$.
\end{proposition}


\begin{proof} We can write any quadratic form as $\ww^T \FF \zz = \sum_{i,j} w_i z_j F_{ij}$. Expanding the quadratic forms, the following holds
\begin{align}
    \EE[y_k \, | \, \HH] = \EE[\uu^T\BB \uu \, &| \, \HH] + \tfrac{1}{n} \EE[ \uu^T \DD \AA^T \eeta \, | \, \HH] +  \tfrac{1}{n^2} \EE[ \eeta^T \AA \CC \AA^T \eeta \, | \, \HH]\\
     \text{(ind. of $\eeta$ and $\uu$, $\EE[\eeta] = \bm{0}$)} \quad &= \EE \big [ \sum_{i,j} u_i u_j B_{ij} \, | \, \HH \big ] + \tfrac{1}{n^2} \EE \big [ \sum_{i,j}  \eta_i \eta_j (\AA \CC \AA^T)_{ij} \, | \, \HH \big ]   \\
     \text{ (isotropic prop. of $\eeta$ and $\uu$)} \quad &= R^2 \cdot \sum_i \tfrac{B_{ii}}{d} + \widetilde{R}^2 \cdot \sum_i \tfrac{(\AA \CC \AA^T)_{ii}}{n^2}\\
    &= R^2 \cdot  \tfrac{\text{tr}(\BB)}{d}  + \widetilde{R}^2 \cdot   \tfrac{\text{tr}(\CC \HH)}{n}. 
\end{align}
In the last equality, we used that $\text{tr}(\AA \CC \AA^T) = \text{tr}(\CC \AA^T \AA) = n \cdot \text{tr}(\CC \HH)$. 

To prove \eqref{eq:conditional}, we will use Chebyshev's inequality; hence we need to compute the $\text{Var} \big ( y_k  | \HH \big ) = \mathbb{E} \big [ y^2_k | \HH \big ] - \big (\mathbb{E} [y_k  | \HH] \big )^2$. First, a simple computation yields that
\begin{equation} \label{eq:variance_noise_11}
    \big ( \EE[ y_k | \HH ] \big )^2 = \underbrace{\big [ \tfrac{R^2 \text{tr}(\BB)}{d}  \big ]^2}_{(i)} + \underbrace{ \big [ \tfrac{\tilde{R}^2 \text{tr}(\CC \HH)}{n}  \big ]^2}_{(ii)} + \underbrace{2 \big [ \tfrac{R^2 \text{tr}(\BB)}{d}  \big ] \big [ \tfrac{\tilde{R}^2 \text{tr}(\CC \HH)}{n}  \big ]}_{(iii)}.
\end{equation}
Next, we compute $\EE[y^2_k | \HH]$. By expanding out the terms in \eqref{eq: norm_with_noise1}, we get the following 
\begin{equation} \begin{aligned} \label{eq:variance_noise_22}
    \EE[y^2_k | \HH] &= \underbrace{\EE [ (\uu^T \BB \uu)^2 | \HH ]}_{(a)} + \underbrace{ \EE \big [ \left ( \tfrac{\eeta^T \AA \CC \AA^T \eeta}{n^2} \right )^2 | \HH \big ] }_{(b)} + \underbrace{ \EE \big [ \tfrac{2 \uu^T \BB \uu \cdot \eeta^T \AA \CC \AA^T \eeta}{n^2} \,  | \HH \big ] }_{(c)} \\
    & \qquad + \underbrace{  \EE \big [ \left ( \tfrac{\uu^T \DD \AA^T \eeta}{n} \right )^2 | \HH \big ] }_{(d)} + \underbrace{ \EE \big [ 2 \left ( \uu^T \BB \uu + \tfrac{\eeta^T \AA \CC \AA^T \eeta}{n^2} \right ) \cdot \tfrac{\uu^T \DD \AA^T \eeta}{n} \, | \HH \big ]}_{(e)}. 
\end{aligned} \end{equation}
To compute the variance of $y_k$, we take \eqref{eq:variance_noise_22} and subtract \eqref{eq:variance_noise_11}. Since this is quite a long expression, we will match up terms and compute these terms individually. First consider the terms (a) and (i) in equations~\eqref{eq:variance_noise_22} and \eqref{eq:variance_noise_11} respectively. By expanding out the square, we get
\begin{align*}
    \text{Var}(\uu^T \BB \uu | \HH) = \EE \big [ \left ( \uu^T \BB \uu  \right )^2 | \HH \big ] - \big [\tfrac{R^2 \text{tr}(\BB)}{d} \big ]^2=  \sum_{i,j,k,\ell} \EE[u_i u_j u_k u_{\ell}] B_{ij} B_{k \ell} - \big [\tfrac{R^2 \text{tr}(\BB)}{d} \big ]^2.
\end{align*}
We need each index to appear exactly twice in the above for its contribution to be non-negligible since $\EE[u_i^2] = \tfrac{R^2}{d}$ and $\EE[\uu] = \bm{0}$. There are four possible ways in which this can happen: $\{i = j = k = \ell\}$, $\{i = j, k = \ell, k \neq i\}$, $\{i = k, j = \ell, i \neq j\}$, or $\{i = \ell, j = k, i \neq j\}$. By the symmetry of the $\BB$ matrix, the last two cases are identical. Noting that $\EE[u_i^4] \le \tfrac{C}{d^2}$ and $\EE[u_i^2] = \tfrac{R^2}{d}$, we, consequently, get the following expression for the variance
\begin{equation} \begin{aligned} \label{eq:blah_23}
    \text{Var}(\uu^T &\BB \uu \, | \, \HH) = \sum_{i} \EE[u_i^4] \cdot B_{ii}^2 + \sum_{i \neq j} \EE[u_i^2] \cdot \EE[u_j^2] \cdot \left (B_{ii} B_{jj} + 2 B_{ij}^2 \right )- \tfrac{R^4}{d^2}  [\text{tr}(\BB)]^2\\
    &\le \frac{C-R^4}{d^2} \cdot \sum_i B_{ii}^2 + \frac{2R^4}{d^2} \sum_{i \neq j} B_{ij}^2 + \frac{R^4}{d^2} \big ( \sum_i B_{ii}^2 + \sum_{i \neq j} B_{ii} B_{jj} - [\text{tr}(\BB)]^2 \big )\\
    &= \frac{C-R^4}{d^2} \cdot \sum_{i} B_{ii}^2 + \frac{2R^4}{d^2} \sum_{i \neq j} B_{ij}^2 \\
    &\le \frac{C-R^4}{d^2} \cdot \big (\sum_{i} B_{ii}^2 + \sum_{i \neq j} B_{ij}^2 \big ) =  \frac{C-R^4}{d} \cdot  \left [ \frac{ \text{tr}(\BB^2)}{d} \right ].
\end{aligned} \end{equation}
In the second equality, we used that $\sum_i B_{ii}^2 + \sum_{i \neq j} B_{ii} B_{jj}  = [\text{tr}(\BB)]^2$ and in the second inequality we can without loss of generality choose $C$ so that  $C > 3R^4$. Finally, we used that $\sum_i B_{ii}^2 + \sum_{i \neq j} B_{ij}^2 = \text{tr}(\BB^2)$. 

Next, we consider the terms (b) and (ii) in equations~\eqref{eq:variance_noise_22} and \eqref{eq:variance_noise_11} respectively.
Similar to the previous case, by expanding out the square, we get the following
\begin{align*}
    \text{Var} \big (\tfrac{\eeta^T \AA \CC \AA^T \eeta}{n^2} \, | \, \HH\big ) = \EE \big [ \frac{1}{n^4} \sum_{i,j,k,\ell} \eta_i \eta_j \eta_k \eta_{\ell} (\AA \CC \AA^T)_{ij} (\AA \CC \AA^T)_{k \ell} \, | \, \HH \big ] - \big [\tfrac{\widetilde{R}^2 \text{tr}(\CC \HH)}{n} \big ]^2.
\end{align*} 
Because of independence, isotropic variance $\EE[\eta_i^2] =\widetilde{R}^2$, and mean $\EE[\eeta] = \bm{0}$, we need each index to appear exactly twice in the above expression in order for its contribution to be non-negligible. There are four possible ways in which this can happen: $\{i = j= k = \ell\}, \{i = j, k = \ell, k \neq i\}, \{ i = k, j = \ell, i \neq j \}$, or $\{i = \ell, j = k, i \neq j\}$. As before, we have the following expression for the variance
\begin{equation}
\begin{aligned} \label{eq:noisy_GD_blah1}
    \text{Var} &\big ( \tfrac{1}{n^2} \cdot \eeta^T \AA \CC \AA^T \eeta \, | \, \HH \big ) \le \tfrac{\widetilde{C}}{n^4}  \sum_i (\AA \CC \AA^T)_{ii}^2  + \tfrac{\widetilde{R}^4}{n^4} \sum_{i \neq j} (\AA \CC \AA^T)_{ii} (\AA \CC \AA^T)_{jj}\\
    & \qquad \qquad \qquad \qquad + \tfrac{2 \widetilde{R}^4}{n^4} \sum_{i \neq j} (\AA \CC \AA^T)_{ij}^2  - \tfrac{\widetilde{R}^4}{n^2} [ \text{tr}(\CC \HH) ]^2\\
    &=  \tfrac{\widetilde{C}-\widetilde{R}^4}{n^4} \sum_i (\AA \CC \AA^T)_{ii}^2  + \tfrac{\widetilde{R}^4}{n^4} \big [ \big ( \sum_i (\AA \CC \AA^T)_{ii}^2  + \sum_{i \neq j} (\AA \CC \AA^T)_{ii} (\AA \CC \AA^T)_{jj} \big ) \big ] \\
    & \qquad \qquad \quad + \tfrac{2 \tilde{R}^4}{n^4 }  \sum_{i \neq j} (\AA \CC \AA^T)_{ij}^2 - \tfrac{\tilde{R}^4}{n^2} [ \text{tr}(\CC \HH) ]^2\\
    &\le  \tfrac{\widetilde{C}-\widetilde{R}^4}{n^4} \big [ \sum_i (\AA \CC \AA^T)_{ii}^2 + \sum_{i \neq j} (\AA \CC \AA^T)^2_{ij} \big ] = \tfrac{\widetilde{C}-\widetilde{R}^4}{n} \cdot \big [ \tfrac{\text{tr}( (\CC \HH)^2 )}{n} \big ].
\end{aligned}
\end{equation}
Here we can without loss of generality choose $\widetilde{C}$ so that $\widetilde{C} > 3 \widetilde{R}^4$. Next, we compare (c) and (iii) in equation~\eqref{eq:variance_noise_22} and \eqref{eq:variance_noise_11}, respectively. We begin by expanding out (c)  in equation~\eqref{eq:variance_noise_22} which yields
\begin{align*}
    \EE \big [ \tfrac{2}{n^2} \cdot \uu^T \BB \uu \cdot \eeta^T \AA \CC \AA^T \eeta \, | \, \HH \big ]= \EE \big [ \tfrac{2}{n^2} \big ( \sum_{i,j} u_i B_{ij} u_j \big ) \big ( \sum_{k, \ell} \eta_k (\AA \CC \AA^T)_{k \ell} \eta_\ell  \big ) \, | \, \HH \big ].
\end{align*}
The only terms which contribute are when $i = j$ and $k = \ell$. Therefore, we deduce the following
\begin{equation} \begin{aligned} \label{eq:blah_20}
    \tfrac{2}{n^2} \cdot \EE \big [\uu^T \BB \uu \cdot \eeta^T \AA \CC \AA^T \eeta \, | \, & \HH \big ] - 2 \big [ \tfrac{R^2 \text{tr}(\BB)}{d} \big ] \cdot \big [ \tfrac{\widetilde{R}^2 \text{tr}(\CC \HH)}{n} \big ]\\
    &= \tfrac{2\widetilde{R}^2 R^2}{n^2 d} \sum_{i,j} B_{ii} (\AA \CC \AA^T)_{jj}  - 2 \big [ \tfrac{R^2 \text{tr}(\BB)}{d} \big ] \cdot \big [ \tfrac{\tilde{R}^2 \text{tr}(\CC \HH)}{n} \big ]\\
    &= \tfrac{2\widetilde{R}^2 R^2}{n d} \big [ \text{tr}(\BB) \text{tr}(\CC \HH) \big ] - 2 \big [ \tfrac{R^2 \text{tr}(\BB)}{d} \big ] \cdot \big [ \tfrac{\widetilde{R}^2 \text{tr}(\CC \HH)}{n} \big ] = 0.
\end{aligned} \end{equation}

We have now used up all the terms in \eqref{eq:variance_noise_11} so the remaining terms, (d) and (e), in \eqref{eq:variance_noise_22} we will show are themselves already going to $0$ as $d \to \infty$. Again expanding the term (d), we get
\begin{equation}
    \EE \big [ \tfrac{1}{n^2} \big ( \uu^T \DD \AA^T \eeta \big )^2 \, | \, \HH \big ] = \EE \big [ \tfrac{1}{n^2} \big ( \sum_{i,j} u_i (\DD \AA^T)_{ij} \eta_j \big )^2 \, | \, \HH \big ].
\end{equation}
By independence and isotropic variance of $\uu$ and $\eeta$, the only terms which remain after taking expectations are the ones with $u_i^2$ and $\eta_j^2$ terms. Therefore, we deduce \begin{equation}
 \begin{aligned} \label{eq:GD_noisy_blah_22}
     \tfrac{1}{n^2} \EE \big [ \big ( \uu^T \DD \AA^T \eeta \big )^2 \, | \, \HH \big ] &= \tfrac{1}{n^2} \sum_{i,j} \EE \big [u_i^2  \big ] \cdot \EE[ \eta_j^2 ] \cdot (\DD \AA^T)_{ij}^2 = \tfrac{ R^2 \widetilde{R}^2}{n^2 d} \sum_{i,j} (\DD \AA^T)_{ij}^2\\
     &= \tfrac{ R^2 \widetilde{R}^2}{n} \cdot \big [ \tfrac{\text{tr}( \DD^2 \HH)}{d} \big ].
 \end{aligned} 
 \end{equation}
 The only term which remains in \eqref{eq:variance_noise_22} is (e). Since $\EE[\eeta] = \bm{0}$, the term $\uu^T \BB \uu \cdot \uu^T\tfrac{\DD \AA^T }{n} \eeta$ contributes nothing to the expectation. Similarly since $\EE[\uu] = \bm{0}$, the term $\tfrac{1}{n^3} \cdot \eeta^T \AA \CC \AA^T \eeta \cdot \uu^T \CC \AA^T \eeta$ is also zero in expectation.

Putting all the quantities \eqref{eq:blah_23}, \eqref{eq:noisy_GD_blah1}, \eqref{eq:blah_20}, \eqref{eq:GD_noisy_blah_22} together with \eqref{eq:variance_noise_11} and \eqref{eq:variance_noise_22}, a straight forward application of Chebyshev's inequality yields the result.

\end{proof}
The only difference between $\|\nabla f(\xx_k)\|^2$ and $y_k$ is the coefficients of the polynomials in $\|\nabla f(\xx_k)\|^2$ continuously depend on $\lambda^{\pm}_{\HH}$ while the coefficients of $y_k$ depend on $\lambda^{\pm}$. The polynomials $P_k$ and $Q_k$ together with Assumptions~\ref{assumption: Vector} and~\ref{assumption: spectral_density} ensure that all the conditions of Lemma~\ref{proposition: remove_norm} hold by setting $\ww$ and $\vv$ to combinations of $\uu$ and $\tfrac{1}{n} \AA^T \eeta$ and the polynomials to $\BB$, $\CC$, and $\DD$. Therefore we have $| \|\nabla f(\xx_k)\|^2 - y_k | \Prto[d] 0$ so we can replace $y_k$ with $\|\nabla f(\xx_k)\|^2$. The proof of Proposition~\ref{proposition:conditional} shows, that conditioned on $\HH$, the $\text{Var}(\|\nabla f(\xx_k)\|^2 | \HH)$ is $\mathcal{O}( \tfrac{1}{d})$ and 
\begin{equation} \label{eq:something_3_1} \EE[ \|\nabla f(\xx_k)\|^2 | \HH] = R^2 \text{tr} \big (\tfrac{\BB}{d} \big ) + \tilde{R}^2 \text{tr} \big (\tfrac{\CC \HH}{n} \big ).
\end{equation}


Consequently, conditioned on $\HH$, the squared norm of the gradient is roughly \eqref{eq:something_3_1}. So in view of this, it suffices to understand the expected traces of polynomials in $\HH$. Random matrix theory studies convergence properties of the limiting distribution of high dimensional matrices, particularly the empirical spectral measure. An important tool derived from using Assumption~\ref{assumption: spectral_density} linking the ESM and the expected trace to the moments of the measure $\mu$ is given below.

\begin{proposition}[Convergence of ESM] \label{proposition: moments} Let $\widetilde{P}_k$ be any $k$-degree polynomial. Under Assumption~\ref{assumption: spectral_density}, the following is true
\begin{align*}
    \tfrac{1}{d}\text{\rm tr}\, \widetilde{P}_k(\HH; \lambda^{\pm})  = \int \widetilde{P}_k(\lambda; \lambda^{\pm}) \, \dif\mu_{\HH} \Prto[d] \int \widetilde{P}_k(\lambda; \lambda^{\pm}) \dif\mu\,.
\end{align*}
\end{proposition}

\begin{proof} 
For sufficiently large $d$, Assumption~\ref{assumption: spectral_density} says $\Pr(\lambda_{\HH}^+ > \lambda^+ + \hat{\varepsilon}) \le \frac{\delta}{2}$. Define the event $\mathcal{S} = \{ \lambda_{\HH}^+ \le \lambda^+ + \hat{\varepsilon} \}$. We construct a bounded, continuous function $h$ by
\[h(\lambda) = \begin{cases}
\widetilde{P}_k(0; \lambda^{\pm}), & \text{if $\lambda < 0$}\\
\widetilde{P}_k(\lambda; \lambda^{\pm}), & \text{if $0 \le \lambda \le \lambda^+ + \hat{\varepsilon}$}\\
\widetilde{P}_k(\lambda^+ + \hat{\varepsilon}; \lambda^{\pm}), & \text{otherwise}.
\end{cases}\]
Because the function $h$ is bounded and continuous, Assumption~\ref{assumption: spectral_density} guarantees that
\begin{equation} \label{eq:rand_feature_blah_4}
\Pr \big ( \big  |\int h(\lambda) \, \dif\mu_{\HH} - \int h(\lambda) \, \dif\mu \, \big | > \varepsilon \big ) \le \delta.
\end{equation}
Depending on whether $S$ has occurred, we have for all sufficiently large $d$
\begin{align}
    \Pr \big ( \big | \int \widetilde{P}_k(\lambda; \lambda^{\pm}) \, \dif\mu_{\HH} - \int &\widetilde{P}_k(\lambda; \lambda^{\pm}) \, \dif\mu  \big | > \varepsilon \big ) \nonumber\\
    &= \Pr \big ( \mathcal{S} \cap \{ \big | \int \widetilde{P}_k(\lambda; \lambda^{\pm}) \, \dif\mu_{\HH} - \int \widetilde{P}_k(\lambda; \lambda^{\pm}) \, | > \varepsilon \} \big ) \nonumber \\
    & \qquad + \Pr \big (\mathcal{S}^c \cap \{ \big | \int \widetilde{P}_k(\lambda; \lambda^{\pm}) \, \dif\mu_{\HH} - \int \widetilde{P}_k(\lambda; \lambda^{\pm}) \, \dif\mu \big | > \varepsilon  \} \big ) \nonumber \\
    &\le \Pr \big ( \mathcal{S} \cap \{ \big | \int \widetilde{P}_k(\lambda; \lambda^{\pm}) \, \dif\mu_{\HH} - \int \widetilde{P}_k(\lambda; \lambda^{\pm}) \, \dif\mu  \big | > \varepsilon \} \big ) + \tfrac{\delta}{2}. \label{eq: rand_feature_blah_3}
\end{align}
In the last line, the probability $\Pr \big (\mathcal{S}^c \cap \{ \big | \int \widetilde{P}_k(\lambda; \lambda^{\pm}) \, \dif\mu_{\HH} - \int \widetilde{P}_k(\lambda; \lambda^{\pm}) \, \dif\mu \big | > \varepsilon  \} \big ) \le \Pr(\mathcal{S}^c) \le \tfrac{\delta}{2}$ for large $d$. Hence, we consider only the first term in \eqref{eq: rand_feature_blah_3}. By construction, for any element in $\mathcal{S}$ it is clear that $h(\lambda) = \widetilde{P}_k(\lambda)$. For sufficiently large $d$, equation \eqref{eq:rand_feature_blah_4} yields
\[ \Pr \big ( \mathcal{S} \cap \{ \big | \int \widetilde{P}_k(\lambda; \lambda^{\pm}) \, \dif\mu_{\HH} - \int \widetilde{P}_k(\lambda; \lambda^{\pm}) \, \dif\mu  \big | > \varepsilon \} \big ) \le \Pr \big ( \big | \int h(\lambda) \, \dif\mu_{\HH} - \int h(\lambda) \, \dif\mu \big | > \varepsilon \big ) \le \frac{\delta}{2}.\]
The result follows after combining with \eqref{eq: rand_feature_blah_3}. 
\end{proof}
Now that we have described the main components of our argument, we present a preliminary concentration result for the gradient. 

\begin{proposition}\label{thm: probability_convergence} Suppose the vectors $\xx_0, \widetilde{\xx},$ and $\eeta$ and the matrix $\AA$ satisfy Assumptions~\ref{assumption: Vector} and \ref{assumption: spectral_density} resp. 
The following holds 
\begin{align} \label{eq:grad_convergence_prob}
\big | \|\nabla f(\xx_k)\|^2 - \big (  \underbrace{\textcolor{teal}{R^2} \textcolor{black}{\tfrac{1}{d} \text{\rm tr}(\HH^2 P_k^2(\HH; \lambda^{\pm}) )}}_{\text{signal}}  +  \underbrace{\textcolor{purple}{\widetilde{R}^2} \textcolor{black}{\tfrac{1}{n} \text{\rm tr}(\HH P_k^2(\HH ; \lambda^{\pm}))}}_{\text{noise}} \big ) \big | \Prto[d] 0.
\end{align}
\end{proposition}

\begin{proof}
Recall the definitions in \eqref{eq:blah_10} and \eqref{eq: norm_with_noise1} and equation \eqref{eq:grad_optimality_cond_app}. We note that the only difference between $\|\nabla f(\xx_k)\|^2$ and $y_k$ is that the coefficients of the polynomials in $\|\nabla f(\xx_k)\|^2$ continuously depend on $\lambda_{\HH}^{\pm}$ while the coefficients in $y_k$ depend on $\lambda^{\pm}$. The polynomials $P_k$ and $Q_k$ together with Assumptions~\ref{assumption: Vector} and~\ref{assumption: spectral_density} ensure that all the conditions of Lemma~\ref{proposition: remove_norm} hold by setting $\ww$ and $\vv$ to combinations of $\uu$ and $\tfrac{1}{n} \AA^T \eeta$ and the polynomials to $\BB$, $\CC$, and $\DD$. Therefore we have $| \|\nabla f(\xx_k)\|^2 - y_k | \Prto[d] 0$ so it suffices to prove \eqref{eq:grad_convergence_prob} with $\|\nabla f(\xx_k)\|^2$ replaced by $y_k$.

Fix constants $\varepsilon, \delta > 0$. Proposition~\ref{proposition: moments} guarantees convergence in probability of any expected trace to a constant which depends on the polynomial and the deterministic measure $\mu$. This together with the definitions of $\BB$, $\CC$, and $\DD$ yield for sufficiently large $d$ 
\begin{equation}\label{eq:bound_traces}
\begin{gathered} 
\Pr \big ( \big | \tfrac{\text{tr}(\BB^2)}{d} \big | > M_1 \defas \varepsilon + \int \lambda^4 P_k^4(\lambda; \lambda^{\pm}) \, \dif\mu \big ) \le \tfrac{\delta}{6},\\
\Pr \big ( \big | \tfrac{\text{tr}((\CC \HH)^2)}{n} \big | > M_2 \defas  \varepsilon + r \int \lambda^2 P_k^4(\lambda; \lambda^{\pm}) \, \dif\mu \big ) \le \tfrac{\delta}{6},\\
\text{and} \quad \Pr \big ( \big | \tfrac{\text{tr}(\DD^2 \HH)}{d} \big | >  M_3 \defas \varepsilon + 4\int \lambda^3 P_k^4(\lambda; \lambda^{\pm}) \, \dif\mu  \big ) \le \tfrac{\delta}{6}.
\end{gathered}
\end{equation}
 We define the set $\mathcal{S}$ for which the expected traces of the random matrices are bounded, namely, 
\[ \mathcal{S} = \big \{ \big | \tfrac{\text{tr}(\BB^2)}{d} \big | \le M_1 \} \cap  \big \{ \big | \tfrac{\text{tr}((\CC \HH)^2)}{n} \big | \le M_2 \} \cap \big \{ \big | \tfrac{\text{tr}(\DD^2 \HH)}{d} \big | \le  M_3 \big )  \big \}, \]
and we observe because of \eqref{eq:bound_traces} that the probability $\Pr(\mathcal{S}^c) \le \frac{\delta}{2}$. The total law of probability yields the following
\begin{align}
    \Pr \big ( \big | y_k - \big [R^2 \text{tr} \big (\tfrac{\BB}{d} \big ) &+ \widetilde{R}^2 \text{tr} \big (\tfrac{\CC \HH}{n} \big ) \big ] \big | > \varepsilon \big ) =  \Pr \big ( \mathcal{S} \cap \big \{ \big | y_k - \big [R^2 \text{tr} \big (\tfrac{\BB}{d} \big ) + \widetilde{R}^2 \text{tr} \big (\tfrac{\CC \HH}{n} \big ) \big ] \big | > \varepsilon \big \} \big ) \nonumber \\
    & \qquad \qquad + \Pr \big ( \mathcal{S}^c \cap \big \{ \big | y_k - \big [R^2 \text{tr} \big (\tfrac{\BB}{d} \big ) + \widetilde{R}^2 \text{tr} \big (\tfrac{\CC \HH}{n} \big ) \big ] \big | > \varepsilon \big \} \big ) \nonumber \\
    &\le \Pr \big ( \mathcal{S} \cap \big \{ \big | y_k - \big [R^2 \text{tr} \big (\tfrac{\BB}{d} \big ) + \widetilde{R}^2 \text{tr} \big (\tfrac{\CC \HH}{n} \big ) \big ] \big | > \varepsilon \big \} \big ) + \tfrac{\delta}{2}. \label{eq:blah_30}
\end{align}
Hence it suffices to bound the first term in \eqref{eq:blah_30}. The idea is to condition on the matrix $\HH$ and apply Proposition~\ref{proposition:conditional}. The law of total expectation yields
\begin{align}
    \Pr \big ( \mathcal{S} \cap \big \{ \big | y_k - \big [\text{tr} &\big (\tfrac{R^2 \BB}{d} \big ) + \text{tr} \big (\tfrac{\widetilde{R}^2  \CC \HH}{n} \big ) \big ] \big | > \varepsilon \big \} \big ) \nonumber \\
   \text{(conditioned on $\HH$)} \, \, \, &= \EE \big [ 1_{\mathcal{S}} \Pr \big ( \big | y_k - \big [ \text{tr} \big (\tfrac{R^2 \BB}{d} \big ) + \text{tr} \big (\tfrac{\widetilde{R}^2 \CC \HH}{n} \big ) \big ] \big | > \varepsilon | \HH \big ) \big ] \nonumber \\
    \text{(Proposition~\ref{proposition:conditional})} \, \, \, &\le \tfrac{1}{\varepsilon^2} \EE \big [ 1_{\mathcal{S}} \left ( \tfrac{C-R^4}{d} \text{ \rm tr} \big ( \tfrac{\BB^2}{d} \big ) + \tfrac{\widetilde{C}-\widetilde{R}^4}{n} \text{ \rm tr} \big ( \tfrac{(\CC \HH)^2}{n} \big ) + \tfrac{ R^2 \widetilde{R}^2}{n} \big [ \tfrac{\text{tr}( \DD^2 \HH)}{d} \big ]  \right ) \big ]. \label{eq:blah_31}
\end{align}
Here for the indicator of the event $\mathcal{S}$ we use the notation $1_{\mathcal{S}}(\omega)$ where the indicator is $1$ if $\omega \in \mathcal{S}$ and $0$ otherwise. By construction of the event $\mathcal{S}$, each of the expected traces in \eqref{eq:blah_31} are bounded and therefore, we deduce that 
\[\Pr \big ( \mathcal{S} \cap \big \{ \big | y_k - \big [\text{tr} \big (\tfrac{R^2 \BB}{d} \big ) + \text{tr} \big (\tfrac{\widetilde{R}^2  \CC \HH}{n} \big ) \big ] \big | > \varepsilon \big \} \big ) = \tfrac{1}{\varepsilon^2} \cdot \mathcal{O} \big ( \tfrac{1}{d} \big ). \]
By choosing $d$ sufficiently large, we can make the right hand side smaller than $\tfrac{\delta}{2}$. The result immediately follows from \eqref{eq:blah_30}.
\end{proof}
Proposition~\ref{thm: probability_convergence} reveals that for high-dimensional data the squared norm of the gradient $\|\nabla f(\xx_k)\|^2$ is a polynomial in the eigenvalues of the matrix $\HH$. Every eigenvalue, not just the largest or smallest, appears in this formula \eqref{eq:grad_convergence_prob}. This means that first-order methods indeed see all of the eigenvalues of the matrix $\HH$, not just the top or bottom one. However, the expected trace is still a random quantity due to its dependency on the random matrix. We remove this randomness and complete the proof of Theorem~\ref{thm: concentration_main} after noting that the moments of the empirical spectral measure converge in probability to a deterministic quantity, denoted by $\concentration$.

\begin{proof}[Proof of Theorem~\ref{thm: concentration_main}] 
Propositions~\ref{proposition: moments} and \ref{thm: probability_convergence} yield the result.  
\end{proof}

\subsection{Halting time converges to a constant} \label{apx: halting_time_deterministic}

The concentration of the norm of the gradient in \eqref{eq: something_1} gives a candidate for the limiting value of the halting time $T_{\varepsilon}$. More precisely, we define this candidate for the halting time $\tau_{\varepsilon}$ from $\concentration$ and we recall the halting time, $T_{\varepsilon}$, as 
\begin{align}
 \tau_{\varepsilon} \defas \inf \, \{ k > 0  : \concentration \le \varepsilon\} \quad \text{and} \quad T_{\varepsilon} \defas \inf \, \{ k > 0  :  \|\nabla f(\xx_k)\|^2 \le \varepsilon\}\,.
\end{align}
We note that the deterministic value $\tau_{\varepsilon}$ is, by definition, the average complexity of GD whereas $T_{\varepsilon}$ is a random variable depending on randomness from the data, noise, signal, and initialization. 
This leads to our main result that states the almost sure convergence of the halting time to a constant value. We begin by showing that $\tau_{\varepsilon}$ is well-defined.

\begin{lemma}[$\tau_{\varepsilon}$ is well-defined] \label{lem: tau_finite}Under the assumptions of Proposition~\ref{thm: probability_convergence},
the iterates of a convergent algorithm satisfy $\concentration \underset{k \to \infty}{\to} 0$.
\end{lemma}

\begin{proof} 
 Both $\lambda^2 P_k^2(\lambda; \lambda^{\pm}) \to 0$ and $\lambda P_k^2(\lambda; \lambda^{\pm}) \to 0$ and these polynomials are uniformly bounded in $k$ for each $\lambda \in [\lambda^-, \lambda^+]$ (see Lemma~\ref{lem: convergent_algorithm} and \ref{lem: convergent_bounded}). By dominated convergence theorem, the result follows.
\end{proof}

With our candidate for the limiting halting time $\tau_{\varepsilon}$ well-defined, we show that number of iterations until $\|\nabla f(\xx_k)\|^2 \le \varepsilon$ equals $\tau_{\varepsilon}$ for high-dimensional data. We state a more general result of Theorem~\ref{thm: Halting_time_main}.

\begin{theorem}[Halting time universality] \label{thm: Halting_time} 
Provided that $\concentration \neq \varepsilon$ for all $k$, the probability of reaching $\varepsilon$ in a pre-determined number of steps satisfies
\[\lim_{d \to \infty} \Pr(T_{\varepsilon} = \tau_{\varepsilon} ) = 1.\]
If the constant $\varepsilon = \concentration$ for some $k$, then the following holds
\[ \lim_{d \to \infty} \Pr(T_{\varepsilon} \in [ \tau_{\varepsilon}, \tau_{\varepsilon} + M_{\varepsilon}]) = 1, \quad
\text{where $M_{\varepsilon} \defas \inf \{ k-\tau_{\varepsilon} > 0 \, | \, \xi_k < \varepsilon \}$.}\]
\end{theorem}

\begin{proof}
To simplify notation, we define $\xi_k \defas \concentration$. First, we consider the case where $\varepsilon \neq \xi_k$ for all $k$. We are interested in bounding the following probabilities
\begin{equation} \Pr(T_{\varepsilon} \neq \tau_{\varepsilon}) = \Pr(T_{\varepsilon} < \tau_{\varepsilon} ) + \Pr(T_{\varepsilon} > \tau_{\varepsilon}). \label{eq: Halting_time_1} \end{equation}
We bound each of these probabilities independently; first consider $\Pr(T_{\varepsilon} < \tau_{\varepsilon})$ in \eqref{eq: Halting_time_1}. For $\tau_{\varepsilon} = 0$, we note that $\Pr(T_{\varepsilon} < \tau_{\varepsilon}) = 0$ since $T_{\varepsilon} \ge 0$. So we can assume that $\tau_{\varepsilon} > 0$. 
Since $T_{\varepsilon} \le \tau_{\varepsilon} -1$, we obtain 
\begin{equation} \label{eq: Halting_time_3} \Pr(T_{\varepsilon} < \tau_{\varepsilon}) = \Pr \Big ( \bigcup_{k=0}^{\tau_{\varepsilon}-1} \{T_{\varepsilon} = k\} \Big ) \le \sum_{k=0}^{\tau_{\varepsilon}-1} \Pr(T_{\varepsilon} = k) \le \sum_{k=0}^{\tau_{\varepsilon}-1} \Pr(\|\nabla f(\xx_k)\|^2 \le \varepsilon) .\end{equation}
 Now we bound the probabilities $\Pr(\|\nabla f(\xx_k)\|^2 \le \varepsilon)$. As $\tau_{\varepsilon}$ is the first time $\xi$ falls below $\varepsilon$, we conclude that $  \xi_{\tau_\varepsilon}< \varepsilon < \xi_{\tau_{\varepsilon}-1}, \xi_{\tau_{\varepsilon}-2}, \hdots, \xi_0$ where we used that $\varepsilon \neq 
\xi_k$ for any $k$. Next we define the constant $0 < \delta \defas \displaystyle \min_{ 0 \le k \le \tau_{\varepsilon}} \, \{ |\varepsilon-\xi_k|  \}$ and we observe that $\delta < |\varepsilon - \xi_k| = \xi_k- \varepsilon$ for all $k < \tau_{\varepsilon}$. Fix a constant $\hat{\varepsilon} > 0$ and index $k$. 
Theorem~\ref{thm: concentration_main} says that by making $d(k)$ sufficiently large
\begin{align*}  \Pr(\|\nabla f(\xx_k)\|^2 \leq \varepsilon) \le \Pr(\|\nabla f(\xx_k)\|^2 < \xi_{k} - \delta) 
\le \frac{\hat{\varepsilon}}{\tau_{\varepsilon}}.
\end{align*}
Here we used that 
$\tau_{\varepsilon}$ is finite for every $\varepsilon >0$ (Lemma~\ref{lem: tau_finite}). Set $D \defas{} \max\{d(0), d(1), d(2), \hdots, d(\tau_{k}-1)\}$. Then for all $d > D$, we have from \eqref{eq: Halting_time_3} the following
\[ \Pr(T_{\varepsilon} < \tau_{\varepsilon}) \le \sum_{k=0}^{\tau_{\varepsilon}-1} \Pr(\|\nabla f(\xx_k)\|^2 \le \varepsilon) \le \sum_{k=0}^{\tau_{\varepsilon}-1} \frac{\hat{\varepsilon}}{\tau_{\varepsilon}} = \hat{\varepsilon}. \]
Lastly, we bound $\Pr(T_{\varepsilon} > \tau_{\varepsilon})$. The idea is similar to the other direction. Let $\delta$ be defined as above. 
Therefore, again by Theorem~\ref{thm: concentration_main}, we conclude for sufficiently large $d$
\begin{align*}
    \Pr(T_{\varepsilon} > \tau_{\varepsilon}) &\le \Pr( \|\nabla f(\xx_{\tau_{\varepsilon}}) \|^2 > \varepsilon)\le \Pr(\|\nabla f(\xx_{\tau_{\varepsilon}}) \|^2 - \xi_{\tau_{\varepsilon}} > \delta) \to 0.
\end{align*}
Indeed, we used that $ \xi_{\tau_{\varepsilon}} < \varepsilon$ and $\delta < |\varepsilon-\xi_{\tau_{\varepsilon}}| = \varepsilon - \xi_{\tau_{\varepsilon}}$. This completes the proof when $\varepsilon \neq \xi_k$.

Next, we consider the second case where  $\xi_{k} = \varepsilon$. Note that $M_{\varepsilon} < \infty$ for all $\varepsilon$ because $\displaystyle \lim_{k \to \infty} \xi_k = 0$. In this setting, we are interested in bounding
\[ \Pr( T_{\varepsilon} \not \in [\tau_{\varepsilon}, \tau_{\varepsilon} + M_{\varepsilon}]) = \Pr(T_{\varepsilon} < \tau_{\varepsilon}) + \Pr(T_{\varepsilon} > \tau_{\varepsilon} + M_{\varepsilon}).\]
The arguments will be similar to the previous setting. Replacing the definition of $\delta$ above with $\displaystyle \delta \defas \min_{0 \le k \le \tau_{\varepsilon}-1} \{|\varepsilon - \xi_k| \}$ yields that $\delta > 0$ since $\varepsilon < \xi_{\tau_{\varepsilon}-1}, \xi_{\tau_{\varepsilon}-2}, \hdots, \xi_0$. With this choice of $\delta$, the previous argument holds and we deduce that $\Pr(T_{\varepsilon} < \tau_{\varepsilon}) \to 0$. Next we show that $\Pr(T_{\varepsilon} > \tau_{\varepsilon} + M_{\varepsilon})$. As before, we know that $\Pr(T_{\varepsilon} > \tau_{\varepsilon} + M_{\varepsilon}) \le \Pr( \|\nabla f(\xx_{\tau_{\varepsilon+M_{\varepsilon}}}) \|^2  > \varepsilon)$. By definition of $M_{\varepsilon}$, we have that $\varepsilon > \xi_{\tau_{\varepsilon}+M_{\varepsilon}}$. Now define $\delta \defas \varepsilon - \xi_{\tau_{\varepsilon} + M_\varepsilon} > 0$. The previous argument holds with this choice of $\delta$; therefore, one has that $\Pr(T_{\varepsilon} > \tau_{\varepsilon} + M_{\varepsilon}) \to 0$.
\end{proof}

For large models the number of iterations to reach a nearly optimal point equals its average complexity which loosely says $T_{\varepsilon} = \tau_{\varepsilon}$. The variability in the halting time goes to zero. Since the dependence in $\tau_{\varepsilon}$ on the distribution of the data is limited to only the first two moments, almost all instances of high-dimensional data have the same limit. In Tables~\ref{tab:comparison_worst_avg_cvx} and \ref{tab:comparison_worst_avg_str_cvx}, we compute the value of $\concentration$ for various models. 

\subsection{Extension beyond least squares, ridge regression} \label{sec:ridge_regression}
In this section, we extend the results from Theorems~\ref{thm: concentration_main} and \ref{thm: Halting_time_main} to the ridge regression problem. We leave the proofs for the reader as they follow similar techniques as the least squares problem \eqref{eq:LS_main}. We consider the ridge regression problem of the form
\begin{equation} \label{eq:ridge_regression}
    \argmin_{\xx \in \mathbb{R}^d} \left \{ f(\xx) \defas \frac{1}{2n} \|\AA \xx - \bb\|^2 + \frac{\gamma}{2} \|\xx\|^2 \right \}, \quad \text{with $\bb \defas \AA \widetilde{\xx} + \eeta$\,.}
\end{equation}
As in Section~\ref{sec: problem_setting}, we will assume that $\AA \in \mathbb{R}^{n \times d}$ is a (possibly random) matrix satisfying Assumption~\ref{assumption: spectral_density}, $\widetilde{\xx} \in \mathbb{R}^d$ is an unobserved signal vector, and $\eeta \in \mathbb{R}^n$ is a noise vector. The constant $\gamma > 0$ is the ridge regression parameter. Unlike the least squares problem, the gradient of $\eqref{eq:ridge_regression}$ does not decompose into a term involving $\xx_0-\widetilde{\xx}$ and $\eeta$. As such, we alter Assumption~\ref{assumption: Vector} placing an independence assumption between the initialization vector $\xx_0$ and the signal $\widetilde{\xx}$, that is, 

\begin{assumption}[Initialization, signal, and noise.]\label{assumption:ridge_vector} The initial vector $\xx_0 \in \mathbb{R}^d$, the signal $\widetilde{\xx} \in \mathbb{R}^d$, and noise vector $\eeta \in \mathbb{R}^n$ are independent of each other and independent of $\AA$. The vectors satisfy the following conditions:
\begin{enumerate}[leftmargin=*]
    \item The entries of $\xx_0$ and $\widetilde{\xx}$ are i.i.d. random variables and there exists constants $C, \dot{R}, \widehat{R} > 0$ such that for $i = 1, \hdots, d$
    \begin{equation} \begin{gathered} \label{eq:ridge_initial}
    \EE[\xx_0] = \EE[\widetilde{\xx}] = 0, \quad \EE[(x_0)_i^2] = \tfrac{1}{d} \dot{R}^2, \quad \EE[\widetilde{x}_i^2] = \tfrac{1}{d} \widehat{R}^2,\\
    \EE[ (x_0)^4_i ] \le \tfrac{1}{d^2} C, \quad \text{and} \quad \EE[\widetilde{x}_i^4] \le \tfrac{1}{d^2} C.
    \end{gathered} \end{equation}
    \item The entries of noise vector are i.i.d. random variables satisfying the following for $i = 1, \hdots, n$ and for some constants $\widetilde{C}, \widetilde{R} > 0$
    \begin{equation}
        \EE[\eeta] = 0, \quad  \EE[\eta_i^2] = \widetilde{R}^2, \quad and \quad \EE[\eta_i^4] \le \widetilde{C}. 
    \end{equation}
\end{enumerate}
\end{assumption}
The difference between Assumption~\ref{assumption: Vector} and Assumption~\ref{assumption:ridge_vector} is that \eqref{eq:ridge_initial} guarantees the initial vector $\xx_0$ and the signal $\widetilde{\xx}$ are independent. One relates $R^2$ to $\dot{R}^2$ and $\widehat{R}^2$ by $R^2 = \dot{R}^2 + \widehat{R}^2$. First, the gradient of \eqref{eq:ridge_regression} is  
\begin{equation} \label{eq:grad_ridge_regression}
\nabla f(\xx) = \HH (\xx - \widetilde{\xx}) - \tfrac{\AA^T \eeta}{n} + \gamma \xx = (\HH + \gamma \II) (\xx-\widetilde{\xx}) - \tfrac{\AA^T \eeta}{n} + \gamma \widetilde{\xx} = \MM (\xx-\widetilde{\xx}) - \tfrac{\AA^T \eeta}{n} + \gamma \widetilde{\xx}, 
\end{equation}
where the matrix $\MM \defas \HH + \gamma \II $ and $\II$ is the identity matrix. 
Under Assumption~\ref{assumption:ridge_vector} and \ref{assumption: spectral_density}, we derive a similar recurrence expression for the iterates of gradient-based algorithms as Proposition~\ref{prop: polynomials_methods}

\begin{proposition}[Prop.~\ref{prop: polynomials_methods} for ridge regression] Consider a gradient-based method with coefficients that depend continuously on $\lambda^-_{\MM}$ and $\lambda^+_{\MM}$. Define the sequence of polynomials $\{P_k, Q_k\}_{k = 0}^{\infty}$ recursively by
\begin{equation} \begin{gathered} \label{eq:recursive_noise_poly_ridge_regression}
    P_0(\MM; \lambda_{\MM}^{\pm}) = \II \quad \text{and} \quad P_k(\MM; \lambda_{\MM}^{\pm}) = \II - \MM Q_{k}(\MM; \lambda^{\pm}_{\MM})\\
      Q_0(\MM; \lambda_{\MM}^{\pm}) = \bm{0} \quad \text{and} \quad Q_k(\MM; \lambda_{\MM}^{\pm}) = \sum_{i=0}^{k-1} c_{k-1,i} \big [ \MM Q_i(\MM; \lambda_{\MM}^{\pm}) - \II \big ]\,.
\end{gathered} \end{equation}
These polynomials $P_k$ and $Q_k$ are referred to as the \emph{residual} and \emph{iteration} polynomials respectively.
We express the difference between the iterate at step $k$ and $\widetilde{\xx}$ in terms of these polynomials:
\begin{equation} \label{eq:recursive_noise_poly_1_ridge_regression}
\xx_k - \widetilde{\xx} = P_k(\MM; \lambda_{\MM}^{\pm}) (\xx_0-\widetilde{\xx}) + Q_k(\MM; \lambda_{\MM}^{\pm}) \cdot  \frac{\AA^T \eeta}{n} - \gamma Q_k(\MM; \lambda_{\MM}^{\pm}) \widetilde{\xx}\,.
\end{equation}
\end{proposition}
The proof of this proposition follows the same argument as in Proposition~\ref{prop: polynomials_methods}, replacing the gradient of the least squares problem \eqref{eq:LS_main} with the gradient for the ridge regression problem \eqref{eq:grad_ridge_regression}. The polynomials $P_k$ and $Q_k$ are exactly the same as in Proposition~\ref{prop: polynomials_methods} but applied to a different matrix $\MM$ instead of $\HH$ (see Section~\ref{sec:Ex_polynomials} for examples the polynomials $P_k$ and $Q_k$ for various first-order algorithms). Given the resemblance to the least squares problem, it follows that one can relate the residual polynomial to the squared norm of the gradient.

\begin{proposition}[Prop.~\ref{prop:gradient_polynomial} for ridge regression] \label{prop:gradient_polynomial_ridge_regression} Suppose the iterates $\{\xx_k\}_{k=0}^\infty$ are generated from a gradient based method. Let $\{P_k\}_{k=0}^\infty$ be a sequence of polynomials defined in \eqref{eq:recursive_noise_poly_ridge_regression}. Then the following identity exists between the iterates and its residual polynomial,
\begin{equation} \begin{gathered} \label{eq:grad_optimality_cond_app_ridge_regression}
    \| \nabla f(\xx_k) \|^2 = (\xx_0-\widetilde{\xx})^T \MM^2 P_k^2(\MM; \lambda_{\MM}^{\pm}) (\xx_0-\widetilde{\xx}) + \tfrac{\eeta^T \AA}{n} P_k^2(\MM; \lambda_{\MM}^{\pm}) \tfrac{\AA^T \eeta}{n} + \gamma^2 \widetilde{\xx}^T P_k^2(\MM; \lambda_{\MM}^{\pm}) \widetilde{\xx}\\
     -2(\xx_0-\widetilde{\xx})^T \MM P_k^2(\MM; \lambda_{\MM}^{\pm}) \tfrac{\AA^T \eeta}{n} + 2 \gamma (\xx_0-\widetilde{\xx})^T \MM P_k^2(\MM; \lambda_{\MM}^{\pm}) \widetilde{\xx} - 2 \gamma \widetilde{\xx}^T P_k^2(\MM; \lambda_{\MM}^{\pm}) \tfrac{\AA^T \eeta}{n}. \nonumber
\end{gathered}
\end{equation}

\end{proposition}
As in the least squares problem, one can replace the $\lambda_{\MM}^{\pm}$ in the polynomial with $\lambda^{\pm} + \gamma$ for $\MM = \HH + \gamma \II$. Under Assumptions~\ref{assumption: spectral_density} and \ref{assumption:ridge_vector}, using the same technique as in Propositions~\ref{proposition:conditional} for the least squares problem, we derive the following.

\begin{proposition}[Prop.~\ref{thm: probability_convergence_ridge} for ridge regression] \label{thm: probability_convergence_ridge} Suppose the vectors $\xx_0, \widetilde{\xx},$ and $\eeta$ and the matrix $\AA$ satisfy Assumptions~\ref{assumption:ridge_vector} and \ref{assumption: spectral_density} resp. 
The following holds 
\begin{align} \label{eq:grad_convergence_prob_ridge}
\big | \|\nabla f(\xx_k)\|^2 - \big (  \underbrace{\textcolor{teal}{\dot{R}^2} \textcolor{black}{\tfrac{1}{d} \text{\rm tr}(\MM^2 P_k^2(\MM; \lambda^{\pm}) )}}_{\text{initialization}} + \underbrace{\textcolor{teal}{\widehat{R}^2} \tfrac{1}{d} \text{\rm tr}(\HH^2 P_k^2(\MM; \lambda^{\pm})) }_{\text{signal}}  +  \underbrace{\textcolor{purple}{\widetilde{R}^2} \textcolor{black}{\tfrac{1}{n} \text{\rm tr}(\HH P_k^2(\MM ; \lambda^{\pm}))}}_{\text{noise}} \big ) \big | \Prto[d] 0.
\end{align}
\end{proposition}
We remark that we used the independence between $\xx_0$ and $\widetilde{\xx}$ to obtain \eqref{eq:grad_convergence_prob_ridge}. This independence leads to two terms in the gradient corresponding to the initialization and the signal. As the polynomials $\MM^2 P_k^2(\MM; \lambda^{\pm})$, $\HH^2 P_k^2(\MM; \lambda^{\pm})$, and $\HH P_k^2(\MM; \lambda^{\pm})$ are polynomials in $\HH$ (the identity $\II$ commutes with $\HH$), Proposition~\ref{proposition: moments} still holds. Therefore, the equivalent to Theorem~\ref{thm: concentration_main} for ridge regression follows (recall, Theorem~\ref{thm: concentration_main_main_ridge} in Section~\ref{sec:ridge_regression_main}).

\textbf{Theorem.} \rm{(Concentration of the gradient for ridge regression)}
\textit{Under Assumptions~\ref{assumption:ridge_vector} and~\ref{assumption: spectral_density} the norm of the gradient concentrates around a deterministic value:
\begin{equation} \begin{aligned} \label{eq: something_1_main_ridge} \vspace{0.25cm}
\hspace{-0.28cm}  \|\nabla f(\xx_k)\|^2 \Prto[d] &\textcolor{teal}{\overbrace{\dot{R}^2}^{\text{initial.}}} \! \!\!  \int  { \underbrace{(\lambda + \gamma)^2 P_k^2(\lambda + \gamma; \lambda^{\pm})}_{\text{algorithm}}} \textcolor{mypurple}{\overbrace{\dif\mu}^{\text{model}} } + \textcolor{teal}{\overbrace{\widehat{R}^2}^{\text{signal}}} \! \!\!   \int  { \underbrace{\lambda^2 P_k^2(\lambda + \gamma; \lambda^{\pm})}_{\text{algorithm}}} \textcolor{mypurple}{\overbrace{\dif\mu}^{\text{model}} }  \\
& \quad \quad +   \textcolor{purple}{\overbrace{ \widetilde{R}^2} ^{\text{noise}} }  r   \int  { \underbrace{\lambda P_k^2(\lambda + \gamma; \lambda^{\pm})}_{\text{algorithm}}}  \textcolor{mypurple}{\overbrace{ \dif\mu}^{\text{model}} }. \end{aligned}
\end{equation}}

The equivalent to Theorem~\ref{thm: Halting_time_main} immediately follows by replacing $\concentration$ with the right-hand side of \eqref{eq: something_1_main_ridge}.

\section{Derivation of the worst and average-case complexity} \label{sec: avg_derivations}
In this section, we derive an expression for the average-case complexity in the isotropic features model. Here the empirical spectral measure $\mu_{\HH}$ converges to the Mar\v{c}enko-Pastur measure $\MP$ \eqref{eq:MP}. 
The average-case complexity, $\tau_{\varepsilon}$, is controlled by the value of the expected gradient norm in \eqref{eq: something_1}. Hence to analyze the average-case rate, it suffices to derive an expression for this value, $\concentration$.

In light of \eqref{eq: something_1}, we must integrate the residual polynomials in Table~\ref{table:polynomials} against the Mar\v{c}enko-Pastur measure. By combining Theorem~\ref{thm: concentration_main} with the integrals derived in Appendix~\ref{apx:integral_computations}, we obtain the average-case complexities. Apart from Nesterov's accelerated method (convex), an \textit{exact} formula for the average-case rates are obtained. In the convex setting for Nesterov, the integral is difficult to directly compute so instead we use the asymptotic polynomial in \eqref{eq:Bessel_asymptotic_main}. Hence for Nesterov's accelerated method (convex), we only get an asymptotic average-case rate for sufficiently large $k$ (see Appendix~\ref{apx:integral_computations}).  Tables~\ref{tab:comparison_worst_avg_cvx} and ~\ref{tab:comparison_worst_avg_str_cvx} summarize the asymptotic rates where both iteration and problem size are large.

We now turn to the worst-case guarantees. We discuss below how to make the worst-case complexity comparable.

\subsection{Traditional worst-case complexity}

Recall, the prior discussion on the dimension-dependent constants in the typical worst-case complexity bounds. We now make this precise below. 


\paragraph{Worst-case complexity: strongly convex and noiseless non-strongly convex regimes.} Consider GD and note the other methods will have similar analysis. Recall, the standard analytical worst-case bound for the strongly convex regime and the exact worst-case bound for the non-strongly convex setting \citep{taylor2017smooth}, respectively,
\begin{equation*} \begin{gathered}
    \|\nabla f(\xx_k)\|^2 \le (\lambda_{\HH}^+)^2 \|\xx_0-\xx^{\star}\|^2 \left ( 1- \tfrac{\lambda_{\HH}^-}{\lambda_{\HH}^+} \right )^{2k} \quad \text{(strongly convex)}\\
    \text{and} \quad \|\nabla f(\xx_k)\|^2 \le \frac{(\lambda^+_{\HH})^2 \|\xx_0-\xx^{\star}\|^2}{(k+1)^2} \quad \text{(convex)}
\end{gathered}\end{equation*}
where $\xx^{\star}$ is the minimal norm solution of \eqref{eq:LS}. For sufficiently large $d$, the largest $\lambda_{\HH}^+$ and smallest eigenvalues $\lambda_{\HH}^-$ of $\HH$ converge in probability to $\sigma^2 (1+\sqrt{r})^2$ and $\sigma^2(1-\sqrt{r})^2$ respectively. These are the top and bottom edge of the Mar\v{c}enko-Pastur distribution. We also note in the noiseless setting $\|\xx_0-\xx^{\star}\|^2 = \|\xx_0-\widetilde{\xx}\|^2$. Hence by Assumption~\ref{assumption: Vector} and $\widetilde{R}^2 = 0$, on average, $\|\xx_0-\widetilde{\xx}\|^2 = R^2$. Moreover when the matrix $\HH$ is nonsingular, as in the strongly convex setting, the optimum $\|\xx^{\star}\|^2$ does not grow as dimension increases despite the noise. As a sequence of random variables in $d$, $\|\xx_0-\xx^{\star}\|^2$ is tight. From these observations we derive the worst-case complexities.


\paragraph{Worst-case complexity: noisy non-strongly convex regime.} While discussing the worst-case complexity in Section~\ref{sec: average_case}, we noted a discrepancy in the noisy, non-strongly convex regime between the average rate and the exact worst complexity. For instance, the exact worst complexity for gradient descent (GD) \citep{taylor2017smooth} is
\begin{equation} \label{eq:worst_case_complexity}
\|\nabla f(\xx_k)\|^2 \le \frac{(\lambda^+_{\HH})^2 \|\xx_0-\xx^{\star}\|^2}{(k+1)^2} \quad \text{where $\xx^{\star}$ is the optimum of \eqref{eq:LS}.}
\end{equation}
For sufficiently large $d$, the largest eigenvalue of $\lambda^+_{\HH}$ converges a.s. to $\lambda^+ = 4 \sigma^2$, the top edge of the support of $\MP$. Hence, to derive worst-case complexity bounds, it suffices to understand the behavior of the distance to the optimum. 

The vectors $\xx^\star$ and $\widetilde{\xx}$ are different when noise is added to the signal. For simplicity, we consider the setting where the matrix $\AA$ is invertible. Intuitively, the optimum $\xx^{\star} \approx \AA^{-1} \bb = \widetilde{\xx} + \AA^{-1} \eeta$ where $\widetilde{\xx}$ is the underlying random signal. Because the signal $\widetilde{\xx}$ is scaled, Assumption~\ref{assumption: Vector} says $\EE[ \| \xx_0-\widetilde{\xx}\|^2] = R^2$. Therefore, the distance to the optimum $\xx_0-\xx^{\star}$ is controlled by the noise which in turn is bounded by the reciprocal of the minimum eigenvalue of $\AA^T \AA$, namely \[\|\xx_0-\xx^\star\|^2 \approx \|\xx_0-\tilde{\xx}\|^2 + \|\AA^{-1} \eeta\|^2 \ge \frac{|\uu_{\min}^T \eeta|^2}{\lambda_{\min}(\AA^T \AA)}, \]
where $(\lambda_{\min}(\AA^T \AA), \uu_{\min})$ is an eigenvalue-eigenvector pair corresponding to the minimum eigenvalue of $\AA^T\AA$. Unfortunately, the smallest eigenvalue is not well-behaved. Particularly there does not exist any scaling so that expectation of $\lambda_{\min}(\AA^T\AA)^{-1}$ is finite and the distribution is heavy-tailed. Instead we show that this quantity $\frac{|\uu_{\min}^T \eeta|^2}{\lambda_{\min}(\AA^T \AA)}$ grows faster than $\widetilde{R}^2 d$. To do so, we appeal to a theorem in \citep{tao2010random}, that is, we assume that all moments of the entries of the matrix $\AA$ are bounded, namely,
\begin{equation} \label{eq: bounded_moment}
\max_{i,j} \mathbb{E}[|A_{ij}|^k] < \infty \quad \text{for all $k \le 10^4$.}
\end{equation}
This bounded moment assumption is a mild assumption on the entries. For instance it includes any sub-exponential random variables. It should be noted here that under the simple isotropic features model it is clear that $\|\xx_0-\xx^{\star}\|^2$ is \textit{dimension-dependent}, but the exact dependence is more complicated. Under this condition \eqref{eq: bounded_moment}, we can prove a bound, which gives the dependence on the problem size, for the growth rate of the distance to the optimum.

\begin{lemma}[Growth of $\|\xx_0-\xx^\star\|^2$] \label{lem:growth_dist_optimum} Suppose Assumptions~\ref{assumption: Vector} and \ref{assumption: spectral_density} hold such that the noise vector $\eeta \in \RR^d$ and the entries of the data matrix $\AA \in \mathbb{R}^{d \times d}$ satisfy bounded moments \eqref{eq: bounded_moment}. Let $\xx^\star$ be the minimal norm solution to \eqref{eq:LS}. For any $\delta > 0$ there exists a constant $M_{\delta} > 0$ such that
\begin{equation} \liminf_{n \to \infty} \Pr \big ( \|\xx_0- \xx^{\star} \|^2 \ge d \cdot \widetilde{R}^2 M_{\delta} \big ) \ge 1-\delta. \label{eq: growth_norm} \end{equation}
\end{lemma}

\begin{proof} We begin by defining the constant $M_{\delta} > 0$. The $n \times n$ matrix $\AA$ is invertible a.s. so without loss of generality the smallest eigenvalue of $\AA^T \AA$ is non-zero. Here the dimensions are equal, $d = n$. 
From \cite[Corollary 3.1]{edelman1988eigenvalues} and \cite[Theorem 1.3]{tao2010random}, we know that 
$n \lambda_{\min}(\AA^T \AA)$ converges in distribution where we denote the smallest eigenvalue of $\AA^T \AA$ as $\lambda_{\min}(\AA^T \AA)$. It is immediately clear that $\log(n \lambda_{\min}(\AA^T \AA))$ also converges in distribution. By Theorem~3.2.7 in \cite{durrett2010probability}, the sequence of distribution functions $\{F_n(x) = \Pr( \log(n \lambda_{\min}(\AA^T \AA)) \le x) \}$ is tight, that is, there exists an $C_{\delta} > 0$ such that 
\[ \limsup_{n \to \infty} \Pr \big (n \lambda_{\min}(\AA^T \AA) \not \in (e^{-C_{\delta}}, e^{C_{\delta}}] \big ) = \limsup_{n \to \infty} 1 - F_n(C_{\delta}) + F_n(-C_{\delta}) \le \tfrac{\delta}{2}. \]
In particular, we know that \begin{equation} \label{eq: avg_case_1} \limsup_{n \to \infty} \Pr \big ( n^{-1} ( \lambda_{\min}(\AA^T \AA))^{-1} < e^{-C_{\delta}}  \big ) \le \tfrac{\delta}{2}.  
\end{equation}
Another way to observe \eqref{eq: avg_case_1} is that $n\lambda_{\min}(\AA^T \AA)$ has a density supported on $[0, \infty)$ \citep{edelman1988eigenvalues}. For any $\chi^2_1$-squared with $1$-degree of freedom random variable of $X$, there exists a constant $\widehat{C}_{\delta} > 0$ such that \begin{equation} \label{eq: avg_case_3}
    \Pr(X \le \widehat{C}_{\delta}) \le \tfrac{\delta}{2}.
\end{equation}
Let $M_{\delta} \defas \tfrac{1}{4} \min \{ e^{-2C_{\delta}}, \widehat{C}_{\delta}^2 \}$. With $M_{\delta}$ defined, we are now ready to prove \eqref{eq: growth_norm}. The matrix $\AA$ is a.s. invertible so gradient descent converges to $\xx^\star = \AA^{-1} \bb$. Next we observe that \eqref{eq: growth_norm} is equivalent to proving
\begin{equation} \label{eq: avg_case_2} \limsup_{n \to \infty} \Pr \big ( \| \xx_0 - \AA^{-1} \bb \| <  \widetilde{R} \sqrt{n M_{\delta}} \big ) \le \delta.
\end{equation} 
Plugging in the value of $\bb$ and using the reverse triangle inequality, we obtain
\[ \Pr \big ( \|\xx_0- \AA^{-1} \bb\| < \widetilde{R}\sqrt{n M_{\delta}} \big ) \le \Pr \big ( \|\AA^{-1} \eeta\| < \widetilde{R}\sqrt{n  M_{\delta}} + \|\xx_0-\tilde{\xx}\| \big ).\]
Using Markov's inequality, we can obtain a bound on $\|\xx_0 - \widetilde{\xx}\|$ :
\[ \Pr \big ( \|\xx_0-\widetilde{\xx}\| \ge \widetilde{R} \sqrt{nM_{\delta}} \big ) \le \frac{R^2}{n M_{\delta} \widetilde{R}^2}.\]
Consider now the event given by $\mathcal{S} \defas \{ \|\xx_0-\widetilde{\xx}\| < \widetilde{R} \sqrt{nM_{\delta}}\, \} $. The total law of probability yields 
\begin{equation} \begin{aligned} \label{eq: avg_case_4} \Pr &\big ( \|\xx_0- \AA^{-1} \bb\| < \widetilde{R} \sqrt{n M_{\delta}} \big )\\
&\le   \Pr \big (\mathcal{S}^c \big ) + \Pr \big ( \mathcal{S} \cap \{ \|\AA^{-1} \eeta\| < \widetilde{R} \sqrt{nM_{\delta}} + \|\xx_0-\tilde{\xx}\|  \} \big )\\
&\le \frac{R^2}{n M_{\delta} \widetilde{R}^2} + \Pr \big ( \|\AA^{-1} \eeta\| < 2 \widetilde{R} \sqrt{n M_{\delta}} \big ) = \frac{R^2}{n M_{\delta} \widetilde{R}^2} + \Pr \big ( \|\AA^{-1} \eeta\|^2 < 4n\widetilde{R}^2 M_{\delta} \big ).
\end{aligned} \end{equation}
A simple calculation gives that $n^{-1} \widetilde{R}^{-2} \|\AA^{-1} \eeta\|^2 \ge n^{-1} \big ( \lambda_{\min}(\AA^T \AA) \big )^{-1} \widetilde{R}^{-2} ( \uu_{\min}^T \eeta)^2$ where the orthonormal vector $\uu_{\min}$ is the eigenvector associated with the eigenvalue $(\lambda_{\min}(\AA^T\AA))^{-1}$. From this, we deduce the following inequalities
\begin{align*}
    \Pr \big ( &\|\AA^{-1} \eeta\|^2 < 4n \widetilde{R}^2 M_{\delta} \big ) \le \Pr \big ( n^{-1} \lambda_{\min}(\AA^T\AA)^{-1} \cdot \widetilde{R}^{-2} (\uu_{\min}^T \eeta)^2 < \min \{ e^{-2C_{\delta}}, \widehat{C}^2_{\delta} \} \big )\\
    &\le \Pr \big ( n^{-1} \lambda_{\min}(\AA^T\AA)^{-1} < \min \{ e^{-C_{\delta}}, \widehat{C}_{\delta} \} \big ) + \Pr \big ( \widetilde{R}^{-2} (\uu_{\min}^T \eeta)^2  < \min \{ e^{-C_{\delta}}, \widehat{C}_{\delta} \}\big )\\
    &\le \Pr \big (  n^{-1} \lambda_{\min}(\AA^T\AA)^{-1} < e^{-C_{\delta}} \big ) + \Pr \big ( \widetilde{R}^{-2} (\uu_{\min}^T \eeta)^2 <  \widehat{C}_{\delta} \big ).
\end{align*}
 Since $\eeta$ is Gaussian and $\uu_{\min}$ is orthonormal, we know that $\widetilde{R}^{-2} (\uu^T_{\min} \eeta)^2 \sim \chi^2_1$, a chi-squared distribution, so \eqref{eq: avg_case_3} holds and we already showed that $n^{-1} (\lambda_{\min}(\AA^T\AA))^{-1}$ satisfies \eqref{eq: avg_case_1}. By taking $ \displaystyle \limsup$, we have
\[\limsup_{n \to \infty}  \Pr \big ( \|\AA^{-1} \eeta\|^2 < 4n \widetilde{R}^2 M_{\delta} \big ) \le \delta.\]
The inequality in \eqref{eq: avg_case_2} immediately follows after taking the limsup of \eqref{eq: avg_case_4}. 
\end{proof}

Combining this lemma with the equation \eqref{eq:worst_case_complexity}, we get with high probability that
\[ \|\nabla f(\xx_k)\|^2 \le \frac{(\lambda_{\HH}^+)^2 \|\xx_0-\xx^{\star}\|^2}{(k+1)^2} \approx \frac{16 \sigma^2 \widetilde{R}^2 d}{(k+1)^2}.\]
By setting the right-hand side equal to $\varepsilon$, we get the worst-case complexity result. 

\begin{figure*}[t!]
\begin{center}
    \includegraphics[scale = 0.4]{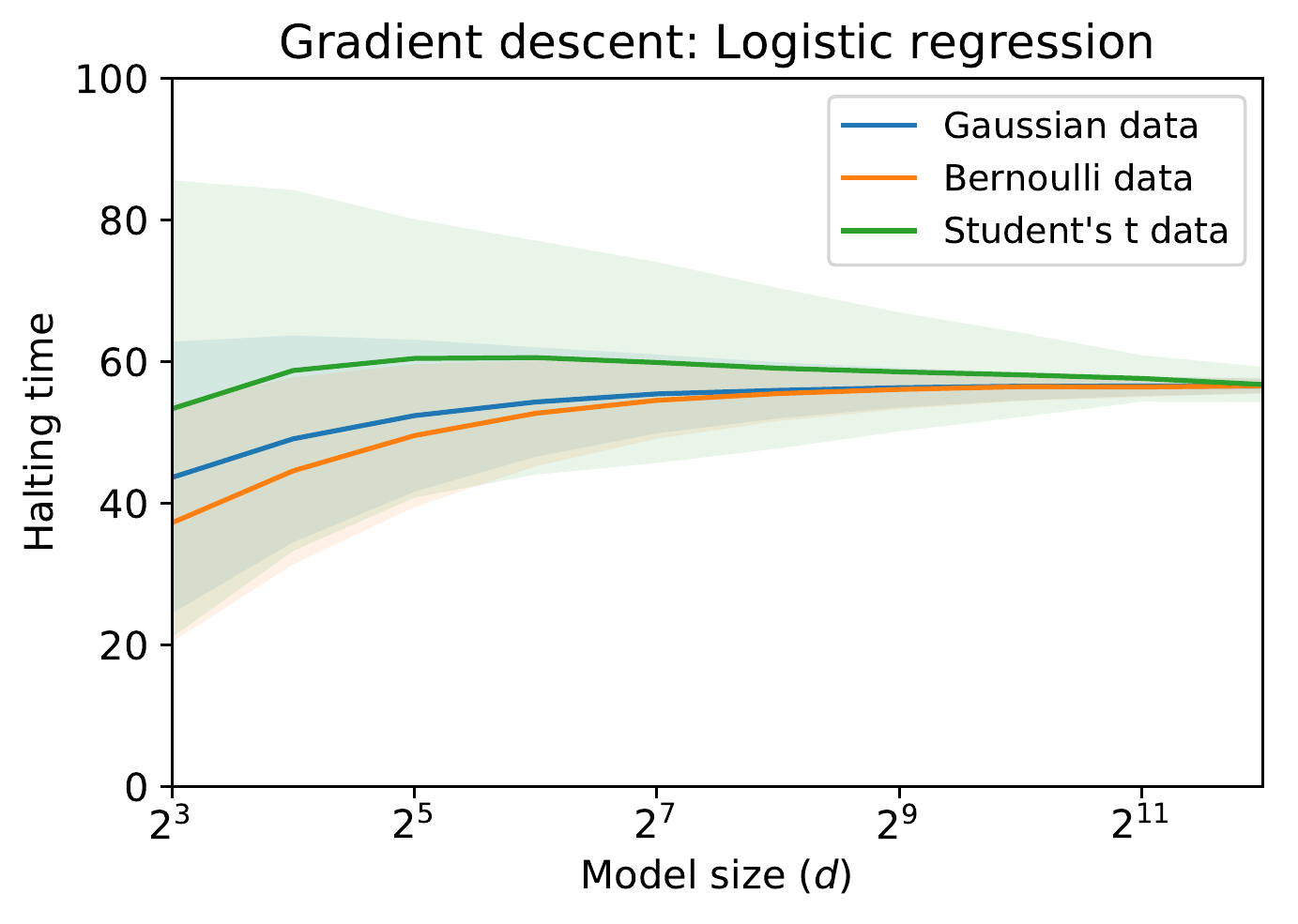}
    \includegraphics[scale = 0.4]{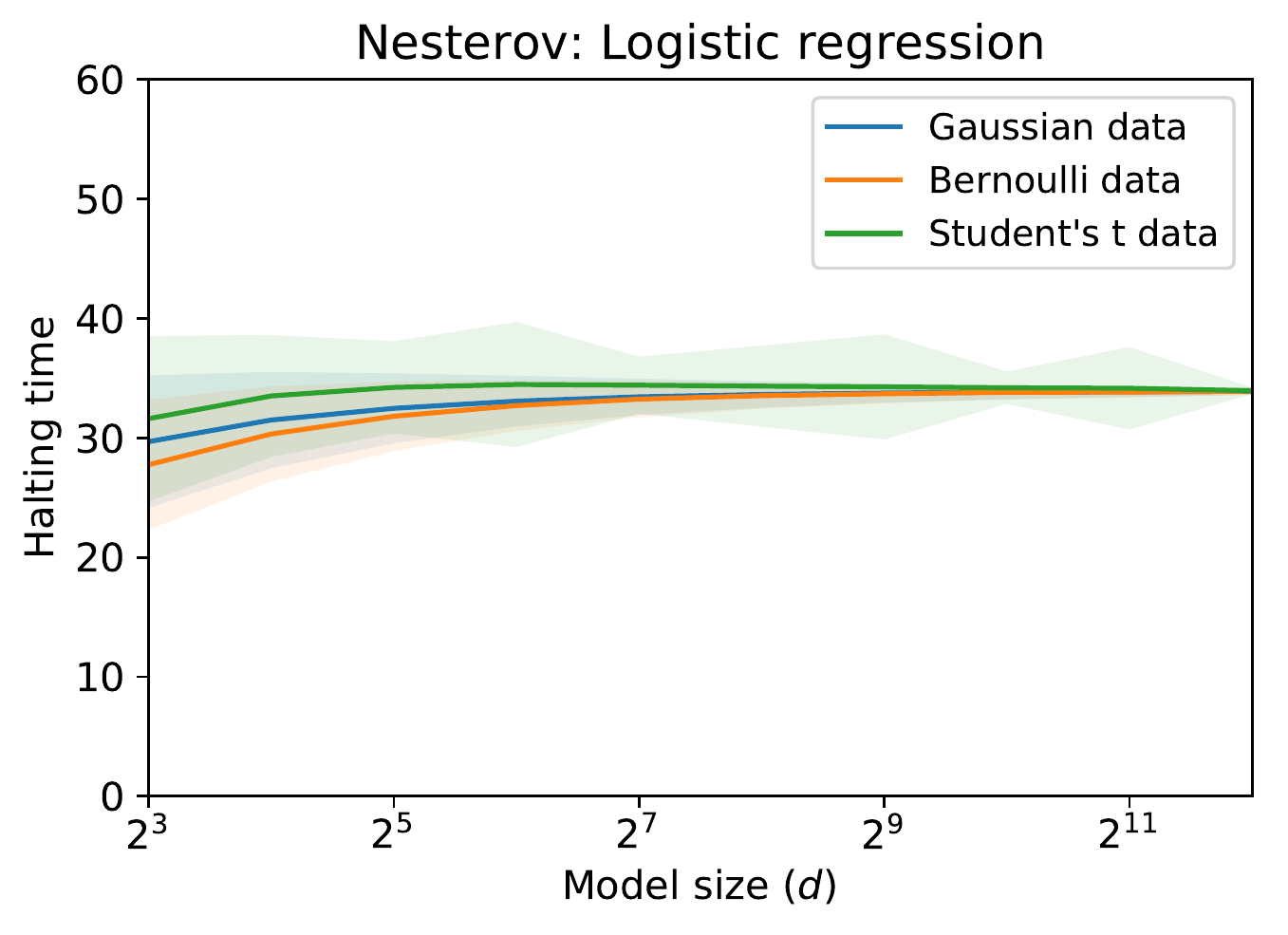}
\includegraphics[scale = 0.4]{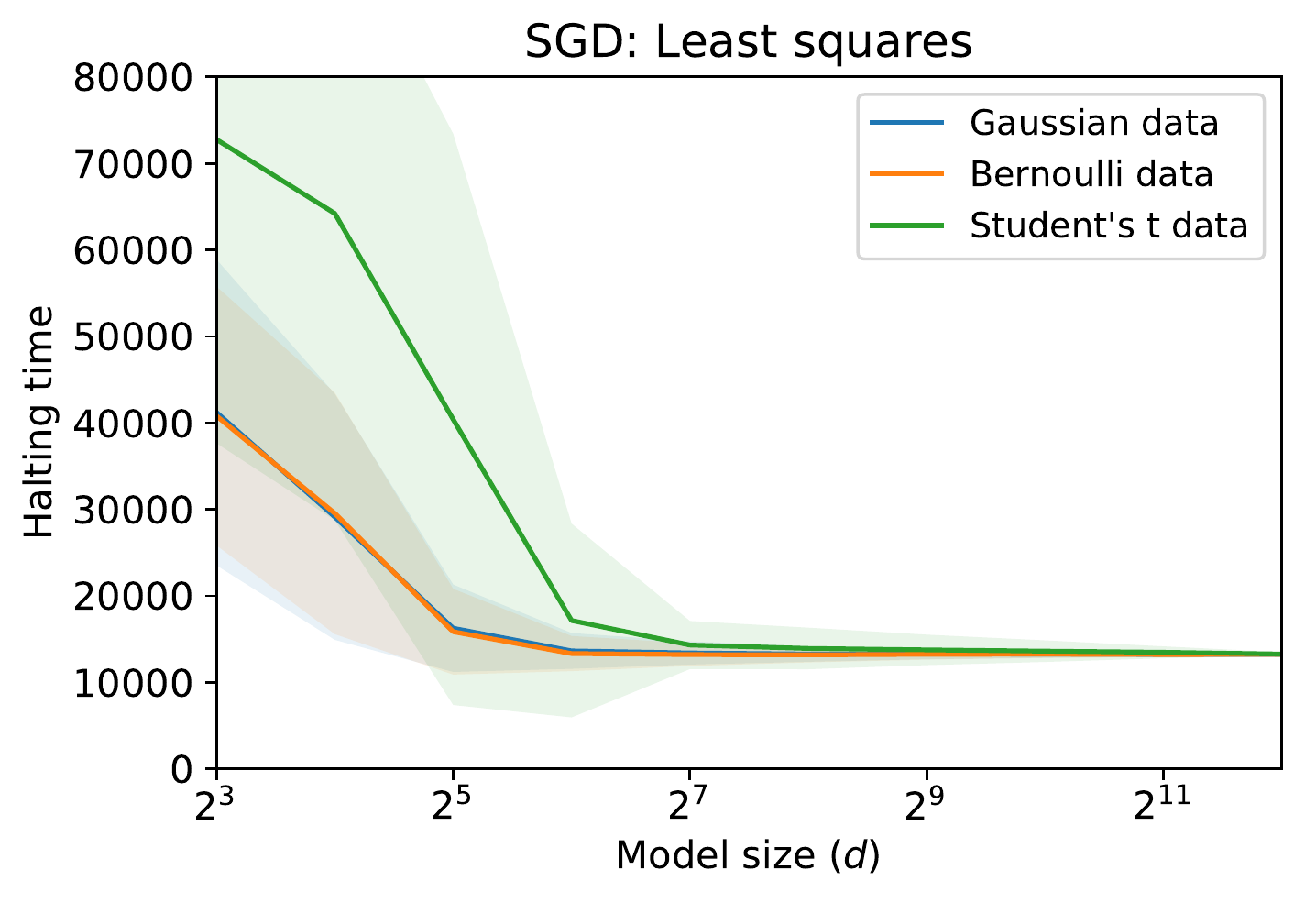}
    \includegraphics[scale = 0.4]{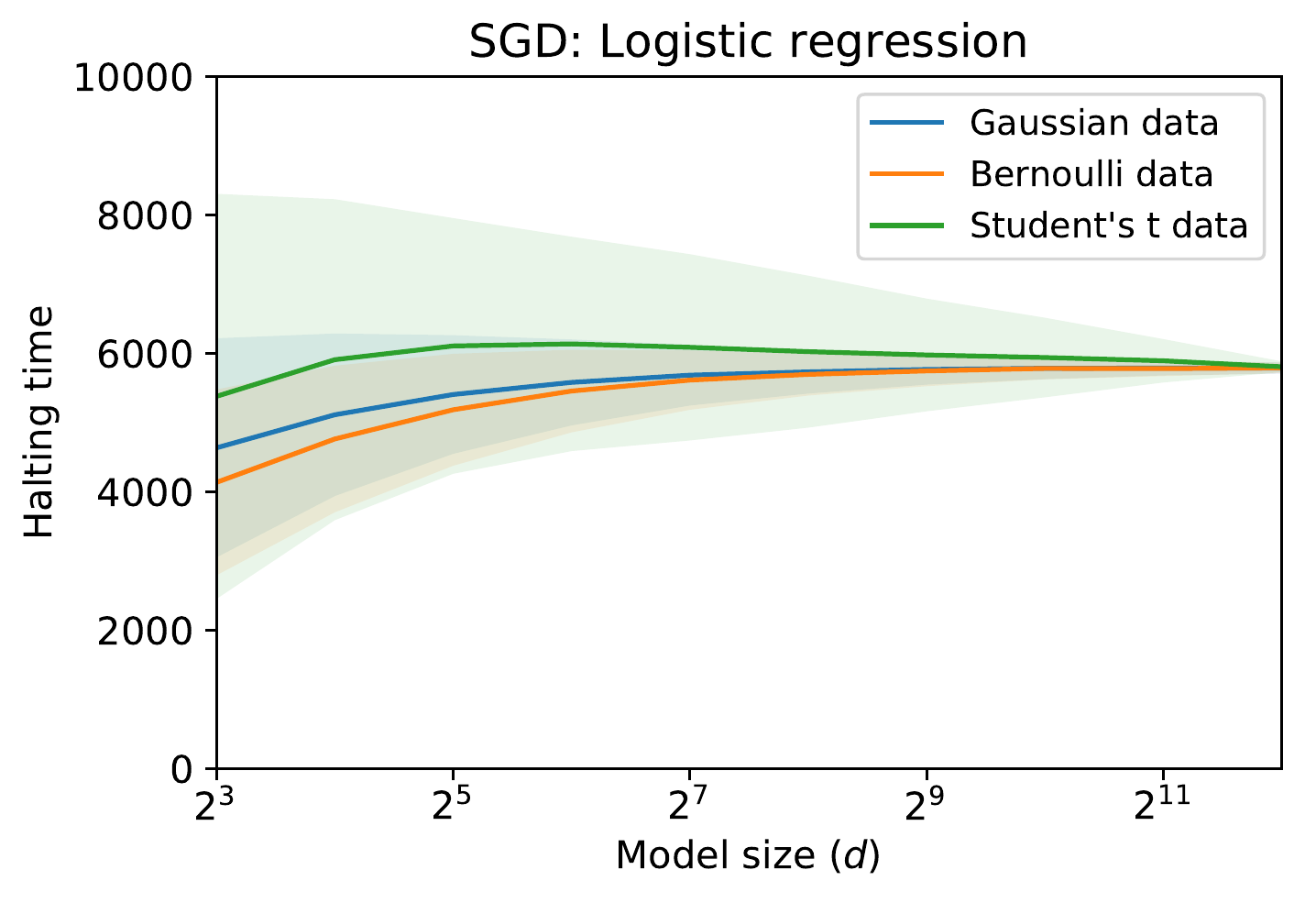}
\end{center}
\caption{{\bfseries Halting time universality beyond least squares.} We compute the halting time on algorithms and models not covered by our theory and note that the convergence to a deterministic and universal halting time is also empirically observed in these settings.
For different model size $d$ ($x$-axis) we sample the vectors $\widetilde{\xx}$, $\xx_0$ and the matrix $\AA$ ($\widetilde{R}^2 = 0.01$ and $r = 0.5$) and report the halting time ($y$-axis) and its standard deviation (shaded area) for GD and Nesterov (convex) on logistic regression and SGD on both least squares and logistic regression.
} \label{fig:halt_time_concentrates}\vspace{-1em}
\end{figure*}

\subsection{Adversarial Model}
Next we recall the adversarial model. Here we assume a noisy generative model for $\bb$ (Assumption~\ref{assumption: Vector}). Then our adversary chooses the matrix $\AA$ without knowledge of $\bb$ in such a way that \textit{maximizes the norm of the gradient} subject to the constraint that the convex hull of the eigenvalues of $\HH = \tfrac{1}{n}\AA^T \AA$ equals $[\lambda^{-},\lambda^+]$.  For comparison to the average-case analysis with isotropic features, we would choose $\lambda^{\pm}$ to be the endpoints of the Mar\v{c}enko-Pastur law.
In light of Proposition~\ref{proposition:conditional}, the adversarial model seeks to solve the constrained optimization problem
\begin{equation} \begin{gathered} \label{eq:adversary_H}
    \max_{\HH} \Big \{  \mathbb{E} \big [ \|\nabla f(\xx_k)\|^2 \big ] = \tfrac{R^2}{d} \text{tr}(\HH^2 P_k^2(\HH; \lambda_{\HH}^{\pm})) + \tfrac{\widetilde{R}^2}{n} \text{tr} ( \HH P_k^2(\HH; \lambda^{\pm}_{\HH})) \Big \} \\ \text{subject to} \quad \lambda_{\HH}^+ = \lambda^+ \, \text{and} \, \lambda_{\HH}^- = \lambda^-,
\end{gathered} \end{equation}
where the largest (smallest) eigenvalue of $\HH$ is restricted to the upper (lower) edge of Mar\v{c}enko-Pastur measure. The optimal $\HH$ of \eqref{eq:adversary_H}, $\HH_{\max}$, has all but two of its eigenvalues at 
\begin{equation} \lambda^*_k \defas \argmax_{\lambda \in [\lambda^-, \lambda^+]} \Big \{ R^2 \lambda^2P_k^2(\lambda; \lambda^{\pm}) + \widetilde{R}^2 \lambda P_k^2(\lambda; \lambda^{\pm}) \Big \} \, . \end{equation}
The other two eigenvalues must live at $\lambda^+$ and $\lambda^-$ in order to satisfy the constraints. The empirical spectral measure for this $\HH_{\max}$ is exactly
\[ \mu_{\HH_{\max}} = \frac{1}{d} \sum_{i=1}^d \delta_{\lambda_i} = \frac{1}{d} \cdot \delta_{\lambda^+} + \frac{1}{d} \cdot \delta_{\lambda^-} + \Big (1-\frac{2}{d} \Big ) \cdot \delta_{\lambda^*_k}. \]
Since this empirical spectral measure weakly converges to $\delta_{\lambda^*_k}$, we satisfy the conditions of Assumption~\ref{assumption: Vector} for these $\HH_{\max}$ and spectral measure $\mu_{\HH_{\max}}$. Hence, Theorem~\ref{thm: concentration_main} holds and the maximum expected squared norm of the gradient as the model size goes to infinity equals
\begin{equation} \begin{aligned} \label{eq: adversary_worst_case}
    \lim_{d \to \infty} \max_{\HH} \, \mathbb{E} \big [ \|\nabla f(\xx_k)\|^2 \big ] &=  \int \big [ R^2 \lambda^2 P_k^2(\lambda; \lambda^{\pm}) + \widetilde{R}^2 r \lambda P_k^2(\lambda; \lambda^{\pm})\big ] \, \delta_{\lambda^*_k}\\
    &= \max_{ \lambda \in [\lambda^-, \lambda^+] } R^2 \lambda^2 P_k^2(\lambda; \lambda^{\pm}) + \widetilde{R}^2 r \lambda P_k^2(\lambda; \lambda^{\pm})\,.
    \end{aligned}
\end{equation}
We called the above expression the \textit{adversarial average-case complexity}. Table~\ref{tab:comparison_worst_avg_cvx} shows these convergence guarantees. We defer the derivations to Appendix~\ref{apx: adversarial_model}. 

\begin{remark} In the strongly convex setting, we omitted the adversarial average-case guarantees out of brevity. For all the algorithms, the value of $\lambda_k^*$ occurs near or at the minimum eigenvalue $\lambda^+$. As such there is (almost) no distinction between the traditional worst-case guarantees and the adversarial guarantees. 
\end{remark}


\section{Numerical Simulations} \label{sec:numerical_simulations}
To illustrate our theoretical results we report simulations using gradient descent (GD) and Nesterov's accelerated method (convex) \citep{nesterov2004introductory,Beck2009Fast} on the least squares problem under the isotropic features model. We further investigate the halting time in logistic regression as well as least squares with mini-batch stochastic gradient descent (SGD). See Appendix~\ref{apx:exp_details} for details. 




\paragraph{Setup.} The vectors $\xx_0$ and $\widetilde{\xx}$ are sampled i.i.d. from the Gaussian $N({\boldsymbol{0}}, \tfrac{1}{d}\II)$ whereas the entries of $\AA$ are sampled either from a standardized Gaussian, a Bernoulli distribution, or a Student's \mbox{$t$-dis}tribution with 5 degrees of freedom, normalized so that they all have the same mean and variance. We train the following models:

\begin{itemize}[leftmargin=*]
    \item \textbf{Least squares.} The least squares problem minimizes the objective function $f(\xx) = \tfrac{1}{2n} \|\AA \xx -\bb \|^2$.  The targets, $\bb = \AA\widetilde{\xx} + \eeta$, are generated by adding a noise vector $\eeta$ to our signal, $\AA \widetilde{\xx}$. The entries of $\eeta$ are sampled from a normal, $N(0, \widetilde{R}^2)$, for different values of $\widetilde{R}^2$.
    \item \textbf{Logistic regression.} For the logistic regression model we generate targets in the domain $(0, 1)$ using $\bb = \sigma\left( \AA\widetilde{\xx} + \eeta\right)$ where $\sigma$ is the logistic function. The output of our model is $\yy = \sigma\left(\AA\xx\right)$, and the objective function is the standard cross-entropy loss: 
    \[f(\xx) = {-\frac{1}{n}\sum_{i=1}^n b_i \cdot \log(y_i) + (1 - b_i) \cdot \log(1 - y_i)}.\]
\end{itemize}

\begin{figure*}[t!]
\begin{center}
\includegraphics[scale =0.4]{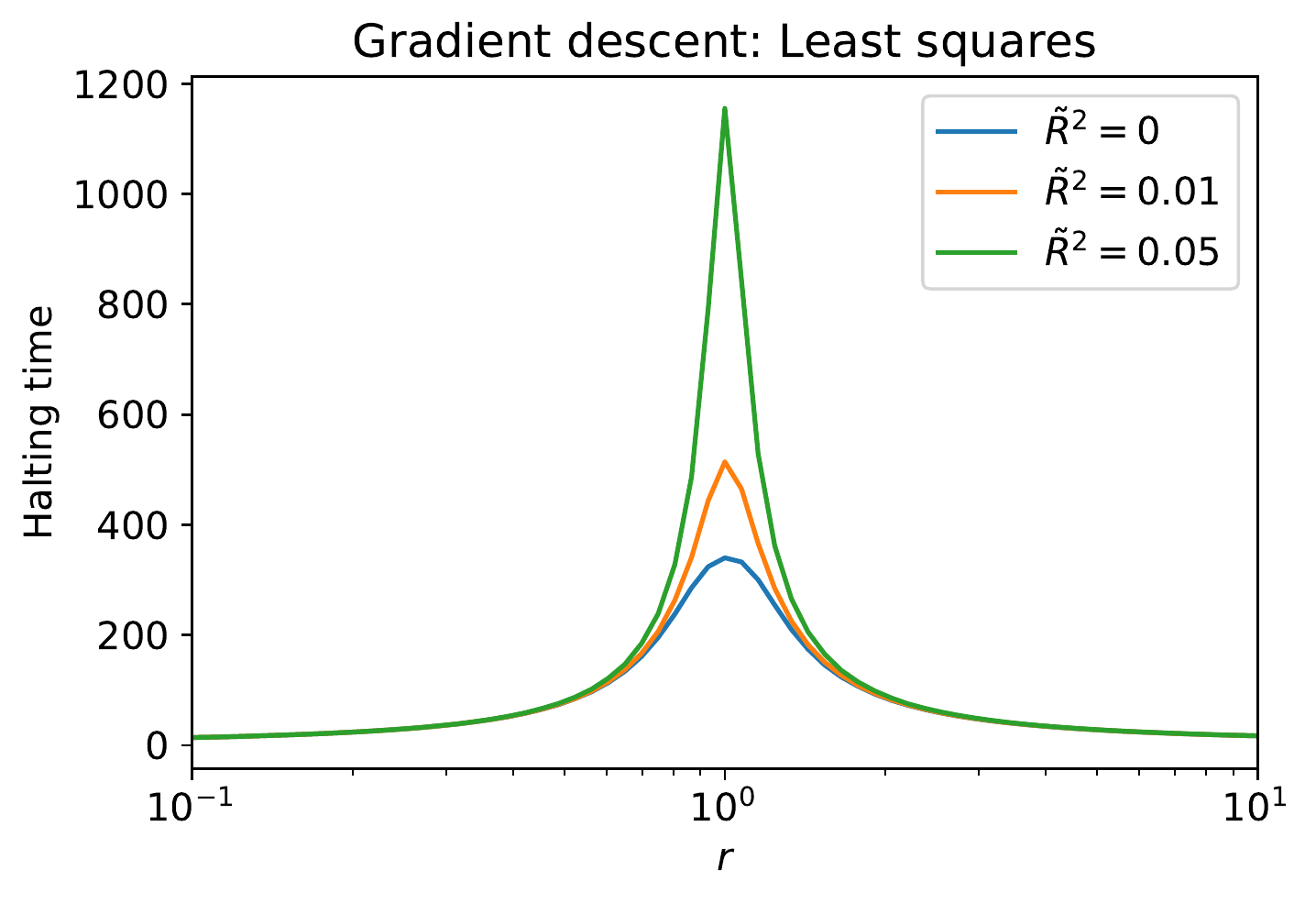}
\includegraphics[scale = 0.4]{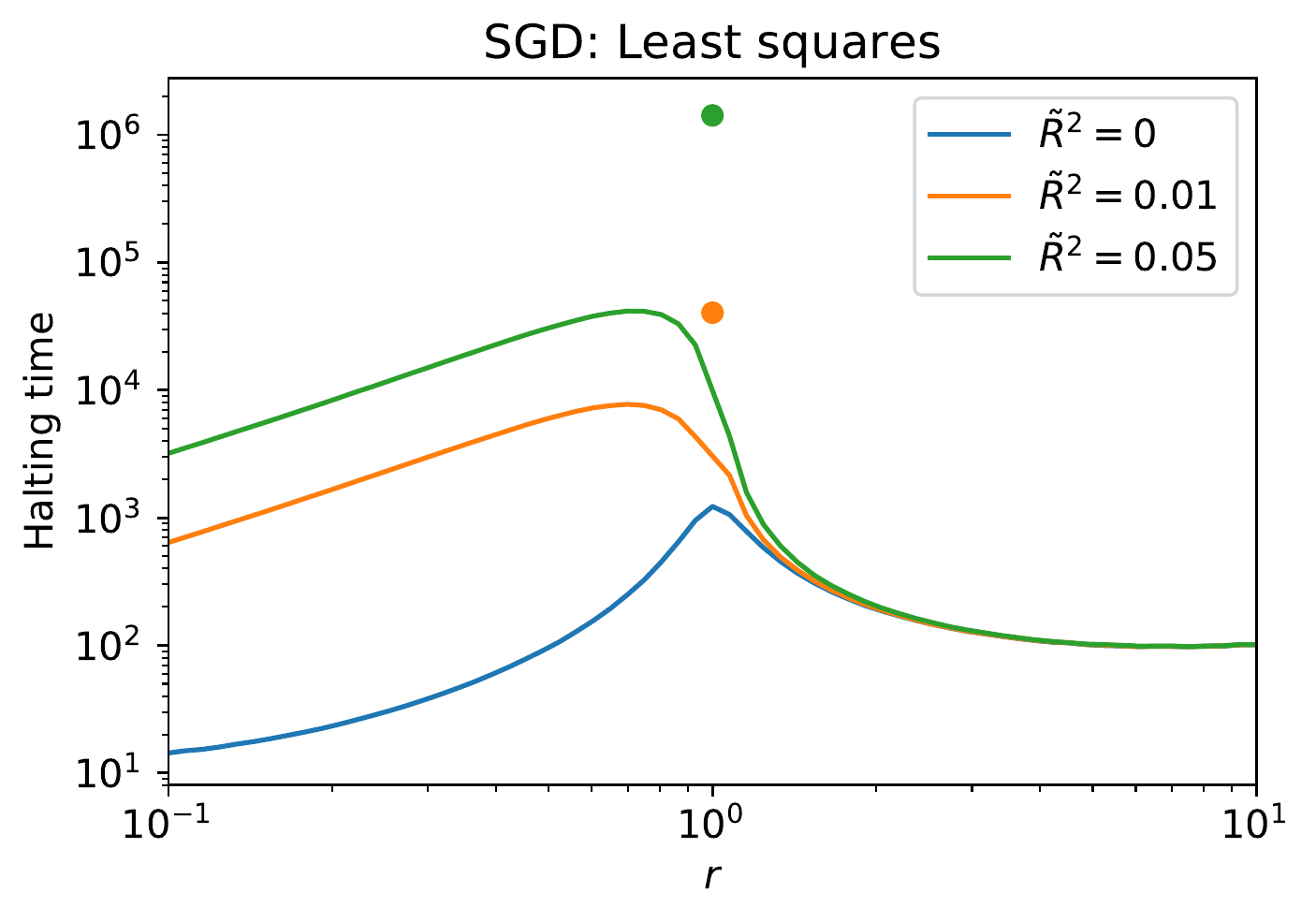}
\end{center}
    \caption{\textbf{Effect of the ratio $r = d/n$ on the halting time for various levels of noise $\widetilde{R}^2$.} The left plot shows the average halting time of gradient descent as a function of the ratio parameter $r$. As predicted by the theory, the halting time increases as $\AA$ approaches a square matrix ($r \to 1$), and the difference between the linear rates ($r \neq 1$) and the sublinear rates ($r = 1)$ grows as the noise level increases. A total of 104,960 models were trained, keeping fixed the number of entries in the matrix $dn = 2^{24}$. In the right plot we show the same curve but for SGD instead, with a batch-size of $\frac{n}{8}$. We plot the curves for all values $r \neq 1$ with the value for $r = 1$ as a single point due to its large value.} 
    \label{fig:r}
\end{figure*}

\paragraph{Parameter settings.} In all simulations, the halting criterion is the number of steps until the gradient falls below $\varepsilon$, $\|\nabla f(\xx)\|^2  < \varepsilon$, where $\varepsilon$ is chosen to be $10^{-6}$ for GD and Nesterov and $\varepsilon$ is $10^{-4}$ for SGD. The step size for GD and Nesterov's accelerated method is fixed to be $1/L$ where $L$ is the Lipschitz constant of the gradient. For least squares, $L=\lambda_{\HH}^+$. We approximate $\lambda_{\HH}^+$ by performing 64 steps of the power iteration method on the matrix $\HH$, initialized with a constant vector of norm 1. For logistic regression, we set the step size to be $4/\lambda_{\HH}^+$.

In SGD, we sample rows from the matrix $\AA$. The mini-batch size parameter is a fixed fraction of the data set size $\frac{n}{16}$, so that the comparison of halting times across model sizes is consistent. When the models are over-parametrized ($n < d$), a strong growth condition~\citep{schmidt2013fast} holds. This means a scaling of the GD step size can be used to ensure convergence. In the under-parametrized setting, SGD does not converge to the optimum. In this case we chose a step size such that the expected squared gradient norm at the stationary point equals the halting criterion. See Appendix~\ref{sec:step_sizes} for derivations. 

\paragraph{Results and conclusions.} Figure~\ref{fig:gd-ls} confirms our theoretical results: variability in the halting time decreases and the halting time converges to a deterministic quantity independent of the distribution of the data. Experimentally, the standard deviation decreased at a rate of $d^{-1/2}$, consistent with results in random matrix theory. For medium sized problems ($d = 2^5$), the heavy-tailed Student's t distribution occasionally produces ill-conditioned matrices resulting in large halting times. These ill-conditioned matrices disappear as the model size grows in large part because the maximum eigenvalue becomes stable. 

More interestingly, our results extend to non-quadratic functions, such as logistic regression, as well as SGD (see Figure~\ref{fig:halt_time_concentrates}). Surprisingly, we see different behaviors between logistic and least square models for smaller matrices when using SGD. Moreover, we note that the large halting times seen in the Student's t distribution for GD on medium sized problems disappear when we instead run SGD.

Secondly, Figure~\ref{fig:r} evaluates the halting times dependency on the ratio $r$. As predicted by the theory, the halting time takes its maximum value (\textit{i.e.}, algorithm is slowest) precisely when $r = 1$. For SGD different step sizes are used for the over-parametrized and under-parametrized regime resulting in an asymmetric curve and a clear discontinuity at $r = 1$.  We leave the study of these phenomena as future work.

\section*{Acknowledgements} The authors would like to thank our colleagues Nicolas Le Roux, Ross Goroshin, Zaid Harchaoui, Damien Scieur, and Dmitriy Drusvyatskiy for their feedback on this manuscript, and Henrik Ueberschaer for providing useful random matrix theory references.

\newpage
\appendix
\section{Derivation of polynomials} \label{apx: derivation_polynomial} In this section, we construct the residual polynomials for various popular first-order methods, including Nesterov's accelerated gradient and and Polyak momentum. 
\subsection{Nesterov's accelerated methods} \label{apx: Nesterov_accelerated_method}
Nesterov accelerated methods generate iterates using the relation 
\begin{align*}
    \xx_{k+1} = \yy_k - \alpha \nabla f(\yy_k),\quad  & \text{where} \quad \alpha = \frac{1}{\lambda_{\HH}^+}\\
    \yy_{k+1} = \xx_{k+1} + \beta_k( \xx_{k+1} - \xx_k), \quad & \text{where} \quad \beta_k = \begin{cases}
    \frac{\sqrt{\lambda_{\HH}^+} - \sqrt{\lambda_{\HH}^-}}{\sqrt{\lambda_{\HH}^+} + \sqrt{\lambda_{\HH}^-} }, & \text{if $\lambda_{\HH}^- \neq 0$}\\
    \frac{k}{k+3}, & \text{if $\lambda_{\HH}^- = 0$}.
    \end{cases}
\end{align*}
By developing the recurrence of the iterates on the least squares problem \eqref{eq:LS}, we get the following three-term recurrence 
\begin{align*}
\xx_{k+1}-\widetilde{\xx} &= (1 + \beta_{k-1}) (I- \alpha \HH) (\xx_k-\widetilde{\xx}) - \beta_{k-1} (I - \alpha \HH)(\xx_{k-1}-\widetilde{\xx}) + \alpha \cdot \tfrac{\AA^T \eeta}{n},
\end{align*}
with the initial vector $\xx_0 \in \RR^d$ and $\xx_1 = \xx_0-\alpha \nabla f(\xx_0)$. Using these standard initial conditions, we deduce from Proposition~\ref{prop: polynomials_methods} the following
\begin{align*}
P_{k+1}(\HH; \lambda_{\HH}^{\pm})&(\xx_0-\widetilde{\xx}) + Q_{k+1}(\HH; \lambda_{\HH}^{\pm}) \tfrac{\AA^T \eeta}{n}\\
&= \big [ (1+ \beta_{k-1}) (\II - \alpha \HH) P_k(\HH; \lambda_{\HH}^{\pm}) - \beta_{k-1} (\II - \alpha \HH) P_{k-1}(\HH; \lambda_{\HH}^{\pm}) \big ] (\xx_0-\widetilde{\xx})\\
& \quad + \big [ (1+\beta_{k-1})(\II-\alpha \HH) Q_k(\HH; \lambda_{\HH}^{\pm}) - \beta_{k-1}(\II-\alpha \HH)Q_{k-1}(\HH; \lambda_{\HH}^{\pm}) + \alpha \II \big ] \tfrac{\AA^T \eeta}{n}.
\end{align*}
It immediately follows that the residual polynomials satisfy the same three-term recurrence, namely,
\begin{equation} \begin{gathered} \label{eq:Nesterov_polynomial}
    P_{k+1}(\lambda; \lambda_{\HH}^{\pm}) = (1+\beta_{k-1}) (1-\alpha \lambda) P_k(\lambda; \lambda_{\HH}^{\pm}) - \beta_{k-1}(1-\alpha \lambda) P_{k-1}(\lambda; \lambda_{\HH}^{\pm})\\
    \text{with} \quad P_0(\lambda; \lambda_{\HH}^{\pm}) = 1, \quad P_1(\lambda; \lambda_{\HH}^{\pm}) = 1-\alpha \lambda\\
    Q_{k+1}(\lambda; \lambda_{\HH}^{\pm}) = (1+\beta_{k-1})(1-\alpha \lambda) Q_k(\lambda; \lambda_{\HH}^{\pm}) - \beta_{k-1} (1 - \alpha \lambda) Q_{k-1}(\lambda; \lambda_{\HH}^{\pm}) + \alpha \\
    \text{with} \quad Q_0(\lambda; \lambda_{\HH}^{\pm}) = 0, \quad Q_1(\lambda; \lambda_{\HH}^{\pm}) = \alpha.
    \end{gathered}
    \end{equation}
By Proposition~\ref{prop: polynomials_methods}, we only need to derive an explicit expression for the $P_k$ polynomials.
    
\subsubsection{Strongly-convex setting}

The polynomial recurrence relationship for Nesterov's accelerated method in the strongly-convex setting is given by
\begin{equation} \begin{gathered} \label{eq: sc_nesterov_1}
    P_{k+1}(\lambda; \lambda_{\HH}^{\pm}) = (1 + \beta)(1-\alpha \lambda) P_k(\lambda; \lambda_{\HH}^{\pm}) - \beta (1-\alpha \lambda) P_{k-1}(\lambda; \lambda_{\HH}^{\pm})\\
    \text{where} \qquad P_0(\lambda; \lambda_{\HH}^{\pm}) = 1, \quad P_1(\lambda; \lambda_{\HH}^{\pm}) = 1-\alpha \lambda, \quad \alpha = \tfrac{1}{\lambda_{\HH}^+} \quad \text{and} \quad \beta = \tfrac{\sqrt{\lambda_{\HH}^+}- \sqrt{\lambda_{\HH}^-}}{\sqrt{\lambda_{\HH}^+}+ \sqrt{\lambda_{\HH}^-}}\,.
\end{gathered}
\end{equation}
We generate an explicit representation for the polynomial by constructing the generating function for the polynomials $P_k$, namely
\begin{align*}
\mathfrak{G}(\lambda, t) &\defas \sum_{k=0}^\infty t^k P_{k}(\lambda; \lambda_{\HH}^{\pm}) \\
\text{(recurrence in \eqref{eq: sc_nesterov_1})} \qquad &= 1 + \frac{1}{t(1+\beta)(1-\alpha \lambda)} \sum_{k=2}^\infty t^k P_k(\lambda; \lambda_{\HH}^{\pm}) + \frac{t \beta}{1+ \beta} \sum_{k=0}^\infty t^k P_k(\lambda; \lambda_{\HH}^{\pm})\\ 
\text{(initial conditions)} \qquad &= 1 + \frac{\left ( \mathfrak{G}(\lambda, t) - (1 + t(1-\alpha \lambda)) \right )}{t(1+\beta)(1-\alpha \lambda)}  + \frac{t \beta}{1 + \beta} \mathfrak{G}(\lambda, t).
\end{align*}
We solve this expression for $\mathfrak{G}$, which gives
\begin{equation} \label{eq:sc_nesterov_2} \mathfrak{G}(\lambda, t) = \frac{1-t\beta(1-\alpha \lambda)}{1 + \beta(1-\alpha \lambda) t^2 - t(1+\beta) (1-\alpha \lambda)}\,.
\end{equation}
Ultimately, we want to relate the generating function for the polynomials $P_k$ to a generating function for known polynomials. Notably, in this case, the Chebyshev polynomials of the 1st and 2nd kind -- denoted $(T_k(x))$ and $(U_k(x))$ respectively--  resemble the generating function for the residual polynomials of Nesterov accelerated method. The generating function for Chebyshev polynomials is given as
\begin{equation} \label{eq:Chebyshev_gen_funct}
    \sum_{k=0}^\infty (T_k(x) + \delta U_k(x)) t^k = \frac{1-tx + \delta}{1-2tx + t^2}\,.
\end{equation}
To give the explicit relationship between \eqref{eq:sc_nesterov_2} and \eqref{eq:Chebyshev_gen_funct}, we make the substitution $t \mapsto \frac{t}{(\beta(1-\alpha \lambda)^{1/2}}$. A simple calculation yields the following 
\begin{equation} \begin{gathered} \label{eq:sc_nesterov_3}
    \sum_{k=1}^\infty \frac{t^k P_k(\lambda; \lambda_{\HH}^{\pm})}{(\beta (1-\alpha \lambda))^{k/2}} = \frac{1- \frac{\beta \sqrt{1-\alpha \lambda}}{\sqrt{\beta}}t}{ 1 - \frac{(1+\beta)\sqrt{1-\alpha \lambda}}{2 \sqrt{\beta}} \cdot 2 t + t^2 } = \frac{1- \tfrac{2\beta}{1+\beta} t x}{1-2tx + t^2} = \frac{\frac{2\beta}{1+\beta} \left (1-tx \right ) + \left ( 1 - \frac{2\beta}{1+\beta} \right )}{1-2tx + t^2}\\
    \text{where} \qquad x = \frac{(1+\beta) \sqrt{1 - \alpha \lambda}}{2 \sqrt{\beta}}\,.
\end{gathered} \end{equation}
We can compare \eqref{eq:Chebyshev_gen_funct} with \eqref{eq:sc_nesterov_3} to derive an expression for the polynomials $P_k$
\begin{equation} \begin{aligned} \label{eq:Nesterov_strongly_cvx_poly_Cheby}
    P_k(\lambda; \lambda_{\HH}^{\pm}) = \big (\beta (1-\alpha \lambda) \big )^{k/2} &\Big [ \frac{2\beta}{1+\beta} T_k \left ( \frac{(1+\beta) \sqrt{1 - \alpha \lambda}}{2 \sqrt{\beta}} \right )\\
    & \qquad + \left (1 - \frac{2\beta}{1 + \beta} \right )U_k \left (\frac{(1+\beta) \sqrt{1 - \alpha \lambda}}{2 \sqrt{\beta}} \right ) \Big ],
\end{aligned}
\end{equation}
where $T_k$ is the Chebyshev polynomial of the first kind and $U_k$ is the Chebyshev polynomial of the second kind.
    
\subsubsection{Convex setting: Legendre polynomials and Bessel asymptotics} \label{apx: Nesterov_accelerated_cvx}
When the objective function is convex (\textit{i.e.} $\lambda_{\HH}^- = 0$), the recurrence for the residual polynomial associated with Nesterov's accelerated method residual reduces to 
\begin{equation} \begin{gathered} \label{eq:Nesterov_polynomial_1}
    P_{k+1}(\lambda; \lambda_{\HH}^{\pm}) = (1+\beta_{k-1}) (1-\alpha \lambda) P_k(\lambda; \lambda_{\HH}^{\pm}) - \beta_{k-1}(1-\alpha \lambda) P_{k-1}(\lambda; \lambda_{\HH}^{\pm})\\
    \text{with} \quad P_0(\lambda; \lambda_{\HH}^{\pm}) = 1, \quad P_1(\lambda; \lambda_{\HH}^{\pm}) = 1-\alpha \lambda, \quad \alpha = \frac{1}{\lambda_{\HH}^+}, \quad \text{and} \quad \beta_k = \frac{k}{k+3}\,.
\end{gathered}
\end{equation}
We now seek to solve this recurrence. 
\paragraph{Nesterov's polynomials as Legendre polynomials.} First we observe that these polynomials are also polynomials in $u =\alpha \lambda$, so we can define new polynomials $\widetilde{P}_k(u)$ such that $\widetilde{P}_k(\alpha \lambda) = P_k(\lambda; \lambda_{\HH}^{\pm})$. Let us define new polynomials $\widetilde{R}_k(u) = \widetilde{P}_k(u)(1-u)^{-k/2}$. Then the recurrence in \eqref{eq:Nesterov_polynomial_1} can be reformulated as 
\begin{equation} \label{eq:Nesterov_polynomial_2}
    \widetilde{R}_{k+1}(u) = (1+\beta_{k-1})(1-u)^{1/2} \widetilde{R}_k(u) - \beta_{k-1} \widetilde{R}_{k-1}(u), \quad \widetilde{R}_0(u) = 1, \, \, \widetilde{R}_1(u) = (1-u)^{1/2}.
\end{equation}
A simple computation shows that the polynomials $\{\widetilde{R}_k\}$ are polynomials in $v = (1-u)^{1/2}$. Because of this observation, we define new polynomials $R_k(v)$ where $R_k( (1-u)^{1/2}) = \widetilde{R}_k(u)$. Now we will find a formula for the polynomials $R_k$ by constructing its generating function,
\[ \mathfrak{G}(v,t) \defas \sum_{k=0}^\infty R_k(v) t^k \]
The recurrence in \eqref{eq:Nesterov_polynomial_2} together with the definition of $R_k$ yields the following differential identity
\begin{align*}
2 \partial_t (v t^{1/2} \mathfrak{G}(v,t) ) &= \sum_{k=0}^\infty (2k + 1) R_k(v) v t^{k-1/2} = v t^{-1/2} + \sum_{k=1}^\infty (2k +1) R_k(v) v t^{k-1/2}\\
&= vt ^{-1/2} + \sum_{k=1}^\infty (k+2) \cdot \frac{2k+1}{k+2} \cdot v \cdot R_k(v) t^{k-1/2} \\
\text{(recurrence in \eqref{eq:Nesterov_polynomial_2})} \qquad &= v t^{-1/2} + \sum_{k=1}^\infty (k-1) R_{k-1}(v) t^{k-1/2} + \sum_{k=1}^\infty (k+2) R_{k+1}(v) t^{k-1/2} \\
&= vt^{-1/2} + t^{3/2} \sum_{k=0}^\infty k R_k(v) t^{k-1} + \sum_{k=2}^\infty (k+1) t^{k-3/2} R_k(v)\\
\text{$\big (\partial_t (t \mathfrak{G}) =  \sum_{k=0}^\infty (k+1) t^k R_k(v) \big)$} \quad &= v t^{-1/2} + t^{3/2} \partial_t( \mathfrak{G} ) + t^{-3/2} \partial_t (t (\mathfrak{G}-(1+vt))).
\end{align*}
One can see this is a first-order linear ODE with initial conditions given by 
\[ \partial_t(\mathfrak{G}) + \frac{1-tv}{t^3 - 2vt^2 +t} \mathfrak{G} = \frac{1+tv}{t^3-2vt^2 +t}, \quad \text{with} \quad \mathfrak{G}(v,0) = 1, \quad \partial_t \mathfrak{G}(v,0) = v.\]
Using an integrating factor of $\mu(t) = \tfrac{t}{\sqrt{t^2-2tv +1}}$ , the solution to this initial value problem is 
\[ \mathfrak{G}(v,t) = \frac{2v \sqrt{t^2 -2vt +1} + tv^2 + t -2v}{t(1-v^2)}.\]
At first glance, this does not seem related to any known generating function for a polynomial; however if we differentiate this function we get that
\[ \sum_{k=1}^\infty k R_k(v) t^k = t\partial_t (\mathfrak{G}) = \frac{2v( vt - 1 + \sqrt{t^2-2tv+1})}{t(1-v^2) \sqrt{t^2-2tv + 1} } = \frac{2v(vt-1)}{t(1-v^2) \sqrt{t^2 -2tv + 1}} + \frac{2v}{t(1-v^2)}, \]
and it is known that the generating function for the Legendre Polynomials $\{L_k\}$ is exactly
\[ \sum_{k=0}^\infty L_k(v) t^k = \frac{1}{\sqrt{t^2 - 2vt + 1}}. \]
Hence it follows that 
\begin{align*}
    \frac{2v(vt-1)}{t(1-v^2) \sqrt{t^2-2tv +1} } &= \left ( \frac{2v^2}{1-v^2} - \frac{2v}{t(1-v^2)} \right ) \frac{1}{\sqrt{t^2 - 2tv +1}}\\
    &= \frac{2v^2}{1-v^2} \sum_{k=0}^\infty L_k(v) t^k -  \frac{2v}{1-v^2} \sum_{k=0}^\infty L_k(v) t^{k-1}, 
\end{align*}
and so for $k \ge 1$, by comparing coefficients we deduce that 
\[
    k R_k(v) = \frac{2v^2}{1-v^2} L_k(v)- \frac{2v}{1-v^2} L_{k+1}(v). 
\]
Replacing all of our substitutions back in, we get the following representation for the Nesterov polynomials for $k \ge 1$
\begin{equation} \label{eq: Nesterov_poly_convex_1}
    P_k(\lambda; \lambda_{\HH}^{\pm}) = \frac{2(1-\alpha \lambda)^{(k+1)/2}}{k \alpha \lambda} \left ( \sqrt{1-\alpha \lambda} \cdot L_k(\sqrt{1-\alpha \lambda})  - L_{k+1}(\sqrt{1-\alpha \lambda}) \right ),
\end{equation}
where $\{L_k\}$ are the Legendre polynomials. 
\paragraph{Bessel asymptotics for Nesterov's residual polynomials.} In this section, we derive an asymptotic for the residual polynomials of Nesterov's accelerated method in the convex setting. We will show that the polynomials $P_k$ in \eqref{eq: Nesterov_poly_convex_1} satisfy in a sufficiently strong sense
\begin{equation} \label{eq:Bessel_asymptotic}
    P_k(\lambda; \lambda_{\HH}^{\pm}) \sim \frac{2J_1(k\sqrt{\alpha \lambda})}{k\sqrt{\alpha \lambda}} e^{-\alpha \lambda k / 2}\,,
\end{equation}
where $J_1$ is the Bessel function of the first kind. Another possible way to derive this asymptotic is to extract the second order asymptotics from \cite{kuijlaars2004Riemann}.
We will show that the Bessel asymptotic \eqref{eq:Bessel_asymptotic} follows directly from the Legendre polynomials in \eqref{eq: Nesterov_poly_convex_1}. To see this, recall the integral representation of a Legendre polynomial is given below by
\[L_k(\sqrt{1-u}) = \frac{1}{\pi} \int_0^\pi \left ( \sqrt{1-u} + i \sqrt{u} \cos(\phi) \right )^k d \phi, \]
and so we have
\begin{align*}
    \frac{1}{\sqrt{u}} \big (  \sqrt{1-u} \, L_k(\sqrt{1-u}) - L_{k+1}(\sqrt{1-u}) \big ) &= \frac{-i}{\pi} \int_0^{\pi} \big (\sqrt{1-u} + i \sqrt{u} \cos(\phi) \big )^k \cos(\phi) \, d\phi. \\
    &= \frac{1}{\pi} \int_0^{\pi} \textrm{Im} \{ \big ( \sqrt{1-u} + i \sqrt{u} \cos(\phi) \big )^k \} \cos(\phi) \, d\phi\\
    \text{(symmetry about $\tfrac{\pi}{2}$)} \qquad &= \frac{2}{\pi} \int_0^{\pi/2} \textrm{Im} \{ \big ( \sqrt{1-u} + i \sqrt{u} \cos(\phi) \big )^k \} \cos(\phi) \, d\phi.
    \end{align*}
Now define the polynomial $\widetilde{P}_k(u) = P_k( \lambda^+ u; \lambda_{\HH}^{\pm})$ where the polynomials $P_k$ satisfies Nesterov's recurrence \eqref{eq:Nesterov_polynomial}. Using the derivation of Nesterov's polynomial from the previous section \eqref{eq: Nesterov_poly_convex_1}, we obtain the following expression
\begin{equation} \begin{aligned} \label{eq:Nesterov_cvx_int}
    \widetilde{P}_k(u) &= \frac{2(1-u)^{(k+1)/2}}{ku} \left (\sqrt{1-u} \cdot L_k(\sqrt{1-u}) - L_{k+1}(\sqrt{1-u}) \right )\\
    &= \frac{4 (1-u)^{(k+1)/2}}{k \pi \sqrt{u}} \int_0^{\pi/2} \textrm{Im} \{ \big ( \sqrt{1-u} + i \sqrt{u} \cos(\phi) \big )^k \} \cos(\phi) \, d\phi.
\end{aligned} \end{equation}
We can get an explicit expression for the imaginary part of the $k$-th power in \eqref{eq:Nesterov_cvx_int} by expressing $\sqrt{1-u} + i \sqrt{u} \cos(\phi)$ in terms of its polar form. In particular, we have that \[\theta(u, \phi) \defas \tan^{-1} \left (\sqrt{\tfrac{u}{1-u}} \cos(\phi) \right ) \quad \text{and} \quad R(u,\phi) \defas \sqrt{1-u + u \cos^2(\phi)} = \sqrt{1- u \sin^2(\phi)}.\]
Hence, we have the following 
\begin{equation} \label{eq:tilde_P}
\begin{aligned} 
\widetilde{P}_k(u) &= \frac{4 (1-u)^{(k+1)/2}}{k \pi \sqrt{u}} \int_0^{\pi/2} \textrm{Im} \{ \big ( \sqrt{1-u} + i \sqrt{u} \cos(\phi) \big )^k \} \cos(\phi) \, d\phi\\
&= \frac{4 (1-u)^{(k+1)/2}}{k \pi \sqrt{u}} \int_0^{\pi/2} R(u, \phi)^k \sin(k \theta(u, \phi)) \cos(\phi) \, d\phi.
\end{aligned}
\end{equation}
Define the following integral
\begin{equation} \label{eq:I_k}
I_k(u) \defas \frac{2}{\pi} \int_0^{\pi/2} R(u, \phi)^k \sin(k \theta(u, \phi)) \cos(\phi) \, d\phi,
\end{equation}
and note the similarity of this integral with the Bessel function, namely
\[ J_1(k \sqrt{u}) = \frac{2}{\pi} \int_0^{\pi/2} \sin(k \sqrt{u} \cos(\phi)) \cos(\phi) \, d\phi. \]
Using this definition, the polynomial can be written as $\widetilde{P}_k(u) = \frac{2(1-u)^{(k+1)/2}}{k\sqrt{u}}I_k(u)$. Since $I_k$ is always bounded, then for $u \ge \log^2(k)/k$, the magnitude of $|\widetilde{P}_k(u)|$ is smaller than any power of $k$. This follows by using the bound that $(1-x)^k \le \text{exp}(-kx)$ and noting that $\text{exp}(-\log^2(k))$ decays faster than any polynomial in $k$. So the interesting asymptotic is for $u \le \log^2(k)/k$, and for this we show the following.

\begin{lemma} \label{lem:Bessel_bound} There is an absolute constant $C$ so that for all $k \ge 1$ and $0 \le u \le \log^2(k)/k$
\[|I_k(u) - J_1(k \sqrt{u}) | \le \begin{cases}
C k^{1/3} \sqrt{u}, & \text{if } u \le k^{-4/3}\\
C k^{-1/3}, & \text{if } u > k^{-4/3}.
\end{cases}\]
\end{lemma}

\begin{corollary}[Nesterov's polynomial asymptotic] \label{cor: Nesterov_poly_asymptotic} There is an absolute constant $C$ so that for all $k \ge 1$ and all $0 \le u \le \tfrac{\log^2(k)}{k}$, the following holds
\begin{equation} \big | \widetilde{P}_k(u) - \frac{2 e^{-uk/2}}{k \sqrt{u}} J_1(k \sqrt{u}) \big | \le \begin{cases}
C e^{-u k/2} k^{-2/3}, & \text{if } u \le k^{-4/3}\\
C e^{-uk/2} u^{-1/2} k^{-4/3}, & \text{if } u > k^{-4/3}.
\end{cases}
\end{equation}
In particular, the following result holds for all $0 \le u \le \frac{\log^2(k)}{k}$
\begin{equation} \label{eq: square_stuff}
\Big | \widetilde{P}_k^2(u) - \frac{4 e^{-uk} J_1^2(k \sqrt{u})}{k^2 u} \Big | \le \begin{cases}
C(k^{-4/3} + k^{-13/6} u^{-3/4}), & \text{if } u \le k^{-4/3}\\
Ce^{-uk} ( u^{-1} k^{-8/3} + u^{-5/4} k^{-17/6}), & \text{if } u \ge k^{-4/3}.
\end{cases}
\end{equation}
\end{corollary}

\begin{proof}[Proof of Corollary~\ref{cor: Nesterov_poly_asymptotic}] First, we have that $\widetilde{P}_k(u) = \frac{2(1-u)^{(k+1)/2}}{k \sqrt{u}} I_k(u)$. A simple triangle inequality shows that
\begin{align*}
    \big |\widetilde{P}_k(u) - \frac{2e^{-uk/2}}{k \sqrt{u}} J_k(k \sqrt{u}) \big | \le \big | \widetilde{P}_k(u) - \frac{2e^{-uk/2}}{k \sqrt{u}} I_k(u)  \big | + \big | \frac{2e^{-uk/2}}{k \sqrt{u}} I_k(u) - \frac{2e^{-uk/2}}{k \sqrt{u}} J_1(k \sqrt{u}) \big |.
\end{align*}
The first difference is small because $|(1-u)^{(k+1)/2} - e^{-uk/2}| \le C e^{-uk/2} (u + ku^2)$ for some absolute constant $C$ and $I_k$ is bounded. The second difference follows directly from Lemma~\ref{lem:Bessel_bound}. The second inequality \eqref{eq: square_stuff} follows from $|a^2-b^2| \le |a-b| (|a-b| + 2|b|)$ and $J_1(x) \le \frac{C}{\sqrt{x}}$. 
\end{proof}

\begin{proof}[Proof of Lemma~\ref{lem:Bessel_bound}] First, we observe that $z^k$ is $k$-Lipschitz on the interval $[0,1]$. Since $e^{-u}$ and $1-u$ lie in the interval $[0,1]$ for any $u \in [0,1]$, the Lipschitz property of $z^k$ and the 2nd-order Taylor approximation of $e^{-uk}$ imply that there exists an $\xi \in [0,1]$ such that
\begin{equation} \begin{aligned} \label{eq:exp} |e^{-ku\sin^2(\phi)/2} - R(u, \phi)^k| &\le \tfrac{k}{2} |e^{-u\sin^2(\phi)}-(1-u\sin^2(\phi))|\\
&= \tfrac{k}{2} |1-u\sin^2(\phi) + \tfrac{(u\sin^2(\phi))^2}{2} - \tfrac{e^{-\xi}(u\sin^2(\phi))^3}{3!} - (1-u\sin(\phi)) |\\
&= \tfrac{k}{2}| \tfrac{(u\sin^2(\phi))^2}{2} - \tfrac{e^{-\xi}}{6} (u\sin^2(\phi))^3| \le \frac{k u^2}{4} \le \frac{C \log^4(k)}{k}.
\end{aligned} \end{equation}
Here we used that $u \le \tfrac{\log^2(k)}{k}$. Similarly, we have that 
\begin{equation} \label{eq:exp_1}
    |R(u, \phi)^k-1| = |(1-u\sin^2(\phi))^{k/2}-1| \le \tfrac{k}{2}|1-u\sin^2(\phi)-1| = \tfrac{k}{2} u \sin^2(\phi).
\end{equation}
We also know that $\sin(k x)$ is $k$-Lipschitz. Therefore again by Taylor approximation on $\tan^{-1}(x)$, we deduce the following bound for some $\xi_{u, \phi} \in [0, \sqrt{\tfrac{u}{1-u}}]$
\begin{equation} \begin{aligned} \label{eq:sin}
    \big | \sin(k \theta(u, \phi)) - \sin \left (k \sqrt{\tfrac{u}{1-u}} \cos(\phi) \right ) \big | &\le k \big | \tan^{-1}\left (\sqrt{\tfrac{u}{1-u}} \cos(\phi) \right ) - \sqrt{\tfrac{u}{1-u}} \cos(\phi) \big |\\
    &\le k \left | \frac{6\xi_{u, \phi}^2-2}{(\xi_{u, \phi}^2 + 1)^3} \right | \left ( \sqrt{ \tfrac{u}{1-u} } \right )^3.
\end{aligned}
\end{equation}
Moreover, suppose $v = \sqrt{u}$ and consider the function $\sqrt{\frac{u}{1-u}} = \tfrac{v}{\sqrt{1-v^2}}$. Using a Taylor approximation at $v = 0$ and the $k$-Lip. of $\sin(kx)$, we obtain that
\begin{equation} \begin{aligned} \label{eq:sin_1}
\big |\sin \left (k \sqrt{\tfrac{u}{1-u}} \cos(\phi) \right ) - \sin(k \sqrt{u} \cos(\phi)) \big | \le k \big | \sqrt{\tfrac{u}{1-u}} - \sqrt{u} \big | \le k u^{3/2}.
\end{aligned}
\end{equation}
Since $0 \le u \le \tfrac{\log^2(k)}{k}$, we have that $\sqrt{\frac{u}{1-u}}$ is bounded by some constant $C$ independent of $k$ and hence, the constant $\xi_{u, \phi}$ is bounded. This implies that $\frac{6 \xi_{u, \phi}^2 -2}{(\xi_{u,\phi}+1)^3}$ is bounded by some absolute constant $\widetilde{C}$. 
By putting together \eqref{eq:sin} and \eqref{eq:sin_1}, we deduce that
\begin{equation} \begin{aligned} \label{eq:sin_2}
    \big | \sin(k \theta(u, \phi)) - \sin \left (k \sqrt{u} \cos(\phi) \right ) \big | &\le k \left | \frac{6\xi_{u, \phi}^2-2}{(\xi_{u, \phi}^2 + 1)^3} \right | \left ( \sqrt{ \tfrac{u}{1-u} } \right )^3 + k u^{3/2} \le C k u^{3/2} 
    \end{aligned}
\end{equation}
where $C$ is some absolute constant. Here we used that $u \le \tfrac{\log^2(k)}{k}$ and $\log^2(k) \le C k^{1/3}$. 
We now consider three cases. Suppose $u \le k^{-4/3}$. Using $u \le k^{-4/3}$ and \eqref{eq:exp_1} we deduce that 
\begin{equation} \label{eq:exp_2}
\big | R(u, \phi)^k-1 \big | \le \tfrac{k}{2} u \le \tfrac{k^{1/3} \sqrt{u}}{2}.
\end{equation}
By putting together \eqref{eq:sin_2} and \eqref{eq:exp_2}, we get the following bound for all $u \le k^{-4/3}$. 
\begin{align*}
    |I_k(u) - J_1(k \sqrt{u})| &\le \frac{2}{\pi} \big | \int_0^{\pi/2} R(u, \phi)^k \sin(k \theta(u, \phi)) \cos(\phi) \, d\phi - \int_0^{\pi/2} \sin(k \theta(u, \phi)) \cos(\phi) \, d\phi \big |\\
    & \quad + \frac{2}{\pi} \big | \int_0^{\pi/2} \sin(k \theta(u,\phi)) \cos(\phi) \, d\phi - \int_0^{\pi/2} \sin(k \sqrt{u} \cos(\phi)) \cos(\phi) \, d\phi \big |\\
    &\le \frac{k^{1/3} \sqrt{u}}{2} + C k^{1/3} \sqrt{u}.
    \end{align*}
    The result immediately follows. On the other hand for $k^{-4/3} \le u \le \log^2(k)/k$, we cut the range of $\phi$. Let $\phi_0$ be such that $\phi_0 = k^{-2/3} u^{-1/2}$. We know from $u \ge k^{-4/3}$ that
    \begin{equation}
       \phi_0 = k^{-2/3} u^{-1/2} \le k^{-2/3} k^{2/3} = 1.
    \end{equation}
Now for $\phi \le \phi_0$, we have in this range that
\begin{equation} \begin{aligned} \label{eq:small_phi}
  \big | &\int_0^{\phi_0} R(u, \phi)^k \sin(k\theta(u, \phi)) \cos(\phi) \, d\phi - \int_0^{\phi_0} \sin(k \sqrt{u} \cos(\phi)) \cos(\phi) \, d\phi \big |\\ &\le \int_0^{\phi_0} \big | R(u, \phi)^k-1|  \, d\phi + \int_0^{\phi_0} | \sin(k \theta(u, \phi)) - \sin(\sqrt{u} k \cos(\phi)) | d\phi\\
  &\le \tfrac{k}{2} u \phi_0^2 + C k u^{3/2} \le \frac{1}{2k^{1/3}} + \frac{C\log^3(k)}{k^{1/2}}.
\end{aligned} \end{equation}
In the last inequality we used that $\phi_0 = k^{-2/3} u^{-1/2}$ and $u \le \log^2(k)/k$. Since $\frac{\log^3(k)}{k^{1/2}} \le C k^{-1/3}$, the result immediately follows.

For larger $\phi$, we use integration by parts with $F(\phi) = \cos(k \sqrt{u} \cos(\phi))$ and $G(\phi) = e^{-k \sin^2(\phi)u/2} \cot(\phi)$ to express 
\begin{align*}
   I_1 &\defas \int_{\phi_0}^{\pi/2} \sin(k \sqrt{\tfrac{u}{1-u}} \cos(\phi)) e^{-k \sin^2 \phi^2 u /2} \cos(\phi) \, d\phi = \frac{1}{k \sqrt{\tfrac{u}{1-u}}} \int_{\phi_0}^{\pi/2} F'(\phi) G(\phi) \, d\phi\\
   &= \frac{\sqrt{1-u}}{k \sqrt{u}} F(\phi) G(\phi) \Big |_{\phi_0}^{\pi/2} - \frac{\sqrt{1-u}}{k \sqrt{u}} \int_{\phi_0}^{\pi/2} G'(\phi) F(\phi) \, d\phi\\
  \text{(def. of $G'(\phi)$)} \,  &= \frac{\sqrt{1-u}}{k \sqrt{u}} \Big [ F(\phi) G(\phi) \Big |_{\phi_0}^{\pi/2}  - \int_{\phi_0}^{\pi/2} F(\phi) e^{-k \sin^2(\phi) u/2} \left ( \cot(\phi) (-k \sin(2\phi)u) - \csc^2(\phi) \right ) \, d\phi \Big ] \\
  &= \frac{\sqrt{1-u}}{k \sqrt{u}} F(\phi) G(\phi) \Big |_{\phi_0}^{\pi/2} + \sqrt{u(1-u)} \int_{\phi_0}^{\pi/2} F(\phi) e^{-k \sin^2(\phi) u/2} \cot(\phi) \sin(2\phi) d\phi\\
  & \qquad \quad + \frac{\sqrt{1-u}}{k \sqrt{u}} \int_{\phi_0}^{/pi/2} F(\phi) e^{-k \sin^2(\phi) u/2} csc^2(\phi) \, d\phi. 
\end{align*}
Since $u \in [0,1]$, we get the following bound
\begin{equation} \begin{aligned} \label{eq:integral_1}
    |I_1| &\le \underbrace{\frac{C}{k \sqrt{u}} |F(\phi_0)G(\phi_0)|}_{\text{(a)}} + \underbrace{\Big | C \sqrt{u}  \int_{\phi_0}^{\pi/2} F(\phi) e^{-k \sin^2\phi u/2} \cot(\phi) \sin(2 \phi) \, d\phi \Big |}_{\text{(b)}}\\
    & + \underbrace{\Big | \frac{C}{k \sqrt{u}} \int_{\phi_0}^{\pi/2} F(\phi) e^{-k \sin^2 \phi u/2} \csc^2(\phi) \, d\phi \Big |}_{\text{(c)}}
\end{aligned} \end{equation}
for some $C > 0$. We will bound each of the terms in \eqref{eq:integral_1} independently. For (a), Taylor's approximation yields that $|\cot(\phi_0) - \tfrac{1}{\phi_0} | \le C$ which implies that $|\cot(\phi_0)| \le \tfrac{C}{\phi_0}$ for some positive constants. Therefore, we deduce that the quantity $\text{(a)}$ is bounded by $\frac{C}{k \sqrt{u} \phi_0}$. For (b) since $|F(\phi)| \le 1$, $|\cot(\phi)\sin(2\phi)| = |2\cos^2(\phi)| \le 2$ and, of course, $| e^{-k \sin^2 \phi u/2} | \le 1$, we have that the quantity (b) is bounded by $C \sqrt{u}$. As for the quantity (c), we use the following approximation $|\csc^2(\phi)- \tfrac{1}{\phi^2}| \le C$ so that $|\csc^2(\phi)| \le \frac{C}{\phi^2}$. Hence the integral (c) is bounded by $\tfrac{C}{k\sqrt{u} \phi_0}$. Therefore we conclude that 
\begin{equation}\label{eq: small_phi}
|I_1| \le \frac{C}{k \sqrt{u} \phi_0} + C \sqrt{u} + \frac{C}{k \sqrt{u} \phi_0} \le Ck^{-1/3}.
\end{equation}
Here we used that $k \sqrt{u} \phi_0 = k^{1/3}$ and $\sqrt{u} \le \log(k)/\sqrt{k} \le C k^{-1/3}$.

Now let's repeat this process replacing $G(\phi) = e^{-k \sin^2(\phi) u/2} \cot(\phi)$ with $G(\phi) = \cot(\phi)$. This time we have that
\begin{align*}
    I_2 \defas \int_{\phi_0}^{\pi/2} \sin(k \sqrt{\tfrac{u}{1-u}} \cos(\phi) ) \cos(\phi) = \frac{1}{k \sqrt{\frac{u}{1-u} }} \int_{\phi_0}^{\pi/2} F'(\phi) G(\phi) \, d\phi. 
\end{align*}
Using the same bounds as before, we deduce the following
\begin{equation} \label{eq:large_phi}
|I_2| \le \frac{C}{k \sqrt{u} \phi_0} + \frac{C}{k \sqrt{u} \phi_0} \le C k^{-1/3}.
\end{equation}

For $u \ge k^{-4/3}$, we have the following result
\begin{align*}
    |I_k(u)-J_1(k \sqrt{u})| &\le \frac{2}{\pi} \big | \int_0^{\phi_0} R(u, \phi)^k \sin(k \theta(u, \phi)) \cos(\phi) \, d \phi - \int_0^{\phi_0} \sin(k \sqrt{u}\cos(\phi) ) \cos(\phi) \, d\phi \big |\\
    & + \frac{2}{\pi} \big | \int_{\phi_0}^{\pi/2} R(u, \phi)^k \sin(k \theta(u, \phi)) \cos(\phi) \, d\phi - \int_{\phi_0}^{\pi/2} \sin(k \sqrt{u} \cos(\phi)) \cos(\phi)) \, d\phi \big |\\
    \text{(by \eqref{eq:small_phi})} \quad &\le C k^{-1/3} + \frac{2}{\pi} \int_{\phi_0}^{\pi/2} \big | (R(u, \phi)^k - e^{-k \sin^2\phi u /2} ) \sin(k \theta(u, \phi)) \cos(\phi) \big | \, d\phi + \frac{2}{\pi} |I_1|\\
    &+ \frac{2}{\pi} |I_2| + \frac{2}{\pi} \int_{\phi_0}^{\pi/2} \big | \big [ \sin(k \theta(u, \phi) ) - \sin(k \sqrt{u} \cos(\phi)) \big ] \cos(\phi) \big | \, d\phi\\
  \text{(by \eqref{eq: small_phi} and \eqref{eq:large_phi})} \quad  & \le Ck^{-1/3} + \frac{2}{\pi} \int_{\phi_0}^{\pi/2} \big | (R(u, \phi)^k - e^{-k \sin^2\phi u /2} ) \sin(k \theta(u, \phi)) \cos(\phi) \big | \, d\phi\\
    & \quad \quad + \frac{2}{\pi} \int_{\phi_0}^{\pi/2} \big | \big [ \sin(k \theta(u, \phi) ) - \sin(k \sqrt{u} \cos(\phi)) \big ] \cos(\phi) \big | \, d\phi\\
    \text{(by \eqref{eq:exp} and \eqref{eq:sin_2})} \qquad &\le C k^{-1/3}.
\end{align*}
This finishes the proof for the lemma. 
\end{proof}

\subsection{Polyak Momentum (Heavy-ball) Method}

The polynomials that generated Polyak's heavy ball method satisfies the following three-term recursion
\begin{equation} \begin{gathered} \label{eq:polyak_1}
P_{k+1}(\lambda; \lambda_{\HH}^{\pm}) = (1-m + \alpha \lambda) P_k(\lambda; \lambda_{\HH}^{\pm}) + m P_{k-1}(\lambda; \lambda_{\HH}^{\pm}), \quad P_0 = 1, \, \, \text{and} \, \, P_1(\lambda; \lambda_{\HH}^{\pm}) = 1- \beta \lambda\\
\text{where} \quad m = - \left ( \tfrac{\sqrt{\lambda_{\HH}^+} - \sqrt{\lambda_{\HH}^-}}{\sqrt{\lambda_{\HH}^+} + \sqrt{\lambda_{\HH}^-}} \right )^2, \quad \alpha = \tfrac{-4}{(\sqrt{\lambda_{\HH}^+} + \sqrt{\lambda_{\HH}^-})^2}, \quad \text{and} \quad \beta = \frac{2}{\lambda_{\HH}^+ + \lambda_{\HH}^-}.
\end{gathered}
\end{equation}
As in the previous examples, we will construct the generating function for the polynomials $P_k$ using the recurrence in \eqref{eq:polyak_1}
\begin{align*}
    \mathfrak{G}(\lambda,t) \defas \sum_{k=0}^\infty t^k P_k(\lambda; \lambda_{\HH}^{\pm}) &= 1 + \frac{1}{t(1-m + \alpha \lambda)} \sum_{k=2}^\infty t^k P_k(\lambda; \lambda_{\HH}^{\pm}) - \frac{m t}{1-m + \alpha \lambda} \sum_{k=0}^\infty t^k P_k(\lambda; \lambda_{\HH}^{\pm})\\
    &= 1 + \frac{1}{t(1 - m + \alpha \lambda)} \big [ \mathfrak{G}(\lambda, t) - 1 - t(1-\beta \lambda) \big ] - \frac{mt}{1-m+\alpha \lambda } \mathfrak{G}(\lambda, t).
\end{align*}
We solve for the generating function
\[ \mathfrak{G}(\lambda, t) = \frac{1 + t(m - (\alpha + \beta) \lambda)}{1 - t(1 - m + \alpha \lambda) - m t^2}. \]
This generating function for Polyak resembles the generating function for Chebyshev polynomials of the first and second kind \eqref{eq:Chebyshev_gen_funct}. First, we set $t \mapsto \frac{t}{\sqrt{-m}}$ (note that $m < 0$ by definition in \eqref{eq:polyak_1}). Under this transformation, we have the following
\begin{align} \label{eq:polyak_2}
    \sum_{k=0}^\infty \frac{t^k P_k(\lambda; \lambda_{\HH}^{\pm})}{(-m)^{k/2}} &= \frac{1 - \frac{t}{\sqrt{-m}} (-m + (\alpha + \beta) \lambda)}{1 - 2t \left ( \tfrac{1-m + \alpha \lambda}{2 \sqrt{-m}} \right ) + t^2 } = \frac{1-\sigma(\lambda)t  \cdot  \frac{-m + (\alpha + \beta) \lambda}{ \sqrt{-m} \cdot \sigma(\lambda)}}{1-2 \sigma(\lambda) t + t^2}\\
    &= \frac{ \frac{-m + (\alpha + \beta) \lambda}{ \sqrt{-m} \cdot \sigma(\lambda)} (1-\sigma(\lambda) t) + 1 - \frac{-m + (\alpha + \beta) \lambda}{ \sqrt{-m} \cdot \sigma(\lambda)} }{ 1- 2 \sigma(\lambda) t + t^2}, \nonumber
\end{align} 
where $\sigma(\lambda) = \frac{\lambda_{\HH}^+ + \lambda_{\HH}^- - 2 \lambda}{\lambda_{\HH}^+ - \lambda_{\HH}^-}$. A simple computation shows that \[\frac{-m + (\alpha + \beta) \lambda}{\sqrt{-m} \sigma(\lambda)} = \tfrac{(\sqrt{\lambda_{\HH}^+} - \sqrt{\lambda_{\HH}^-})^2}{\lambda_{\HH}^+ + \lambda_{\HH}^-} \quad \text{and} \quad 1- \frac{-m + (\alpha + \beta) \lambda}{\sqrt{-m} \sigma(\lambda)} = \tfrac{2\sqrt{\lambda_{\HH}^- \lambda_{\HH}^+}}{\lambda_{\HH}^+ + \lambda_{\HH}^-}.\] 
By matching terms in the generating function for Chebyshev polynomials \eqref{eq:Chebyshev_gen_funct} and Polyak's generating function \eqref{eq:polyak_2}, we derive an expression for the polynomials $P_k$ in Polyak's momentum
\begin{equation} \begin{gathered} \label{eq: Polyak_poly_apx}
    P_k(\lambda; \lambda_{\HH}^{\pm}) = \left ( \tfrac{\sqrt{\lambda_{\HH}^+}-\sqrt{\lambda_{\HH}^-}}{\sqrt{\lambda_{\HH}^+} + \sqrt{\lambda_{\HH}^-}} \right )^k \big [ \tfrac{ ( \sqrt{\lambda_{\HH}^+}-\sqrt{\lambda_{\HH}^-})^2}{\lambda_{\HH}^+ + \lambda_{\HH}^-} \cdot T_k(\sigma(\lambda)) + \tfrac{2 \sqrt{\lambda_{\HH}^- \lambda_{\HH}^+}}{\lambda_{\HH}^+ + \lambda_{\HH}^-} \cdot U_k(\sigma(\lambda)) \big ] \\
    \text{where $T_k(x) \,(U_k(x))$ is the Chebyshev polynomial of the 1st (2nd) kind respectively}\\
    \text{and \quad $\sigma(\lambda) = \tfrac{\lambda_{\HH}^+ + \lambda_{\HH}^- -2 \lambda}{\lambda_{\HH}^+ - \lambda_{\HH}^-}$.}
    \end{gathered} \end{equation}

\section{Average-case complexity} \label{apx:integral_computations}
In this section, we compute the average-case complexity for various first-order methods. To do so, we integrate the residual polynomials found in Table~\ref{table:polynomials} against the Mar\v{c}enko-Pastur density. 

\begin{lemma}[Average-case: Gradient descent] \label{lem: exact_formula_MP} Let $\dif\MP$ be the Mar\v{c}enko-Pastur law defined in \eqref{eq:MP} and $P_k, Q_k$ be the residual polynomials for gradient descent. 
\begin{enumerate}
    \item For $r = 1$ and $\ell = \{1,2\}$, the following holds
    \begin{align*} \int \lambda^\ell P_k^2(\lambda; \lambda^{\pm})  \, d\MP &= \frac{(\lambda^+)^{\ell +1}}{2 \pi \sigma^2} \cdot \frac{\Gamma(2k + \tfrac{3}{2}) \Gamma(\ell + \tfrac{1}{2})}{\Gamma(2k+\ell+2)} \sim \frac{(\lambda^+)^{\ell + 1}}{2 \pi \sigma^2} \cdot \frac{\Gamma(\ell + \tfrac{1}{2})}{(2k + 3/2)^{\ell + 1/2}}.
    \end{align*}
    \item For $r \neq 1$, the following holds
    \begin{align*}
        \int \lambda P_k^2(\lambda; \lambda^{\pm})  \, d\MP &= \frac{(\lambda^+-\lambda^-)^2}{2 \pi \sigma^2 r} \left (1 - \frac{\lambda^-}{\lambda^+} \right )^{2k} \frac{\Gamma(2k + \tfrac{3}{2}) \Gamma(\tfrac{3}{2})}{\Gamma(2k + 3)}\\
        & \sim \frac{(\lambda^+-\lambda^-)^2}{2 \pi \sigma^2 r} \left (1 - \frac{\lambda^-}{\lambda^+} \right )^{2k} \cdot \frac{\Gamma(\tfrac{3}{2})}{(2k + \tfrac{3}{2})^{3/2}}.
    \end{align*}
    \item For $r \neq 1$, the following holds
    \begin{align*}
        &\int \lambda^2 P_k^2(\lambda; \lambda^{\pm}) \, d\MP\\
        &= \frac{(\lambda^+-\lambda^-)^2}{2 \pi \sigma^2 r} \left (1 - \frac{\lambda^-}{\lambda^+} \right )^{2k} \left ( \frac{\lambda^- \cdot \Gamma(2k + \tfrac{3}{2}) \Gamma(\tfrac{3}{2})}{\Gamma(2k + 3)} + \frac{(\lambda^+-\lambda^-) \cdot \Gamma(2k + \tfrac{3}{2}) \Gamma(\tfrac{5}{2})}{\Gamma(2k + 4)} \right )\\
        & \sim \frac{(\lambda^+-\lambda^-)^2}{2 \pi \sigma^2 r} \left (1 - \frac{\lambda^-}{\lambda^+} \right )^{2k} \left ( \frac{\lambda^- \cdot \Gamma(\tfrac{3}{2})}{(2k + \tfrac{3}{2})^{3/2}} +   \frac{(\lambda^+-\lambda^-) \Gamma(\tfrac{5}{2})}{(2k + \tfrac{3}{2})^{5/2}} \right ).
    \end{align*}
\end{enumerate}
\end{lemma}

\begin{proof} The proof relies on writing the integrals in terms of $\beta$-functions. Let $\ell = \{1,2\}$. Using a change of variables $\lambda = \lambda^- + (\lambda^+-\lambda^-) w$, we deduce the following expression
\begin{equation} \begin{aligned} \label{eq:blah_2}
&\int \lambda^\ell P_k^2(\lambda; \lambda^{\pm})  \, d\MP = \frac{1}{2\pi \sigma^2 r} \int_{\lambda^-}^{\lambda^+} \lambda^{\ell-1} \big  (1-\tfrac{\lambda}{\lambda^+} \big )^{2k} \sqrt{(\lambda-\lambda^-)(\lambda^+-\lambda)} \, d\lambda\\
& \qquad \, \, = \frac{(\lambda^+-\lambda^-)^2}{2 \pi \sigma^2 r} \left (1 - \frac{\lambda^-}{\lambda^+} \right )^{2k} \int_0^1 (1-w)^{2k} (\lambda^- + (\lambda^+ - \lambda^-)w)^{\ell-1} \sqrt{w(1-w)} \, dw. 
\end{aligned}\end{equation}
We consider cases depending on whether $\lambda^- = 0$ or not (i.e. $r = 1$). First suppose $\lambda^- = 0$ so by equation \eqref{eq:blah_2} we have 
\begin{align*}
\frac{1}{2\pi \sigma^2 r} \int_{\lambda^-}^{\lambda^+} \lambda^{\ell-1} \big  (1-\tfrac{\lambda}{\lambda^+} \big )^{2k} \sqrt{(\lambda-\lambda^-)(\lambda^+-\lambda)} \, d\lambda &= \frac{(\lambda^+)^{\ell + 1}}{2 \pi \sigma^2 r} \int_0^1 (1-w)^{2k+1/2} w^{\ell - 1/2} \, dw.
\end{align*}
The result follows after noting that the integral is a $\beta$-function with parameters $ 2k + 3/2$ and $\ell + 1/2$ as well as the asymptotics of $\beta$-functions, $\beta(x,y) = \Gamma(y) x^{-y}$ for $x$ large and $y$ fixed. 

Next consider when $r \neq 1$ and $\ell =1$. Using \eqref{eq:blah_2}, we have that \begin{align*} \int \lambda P_k^2(\lambda; \lambda^{\pm}) \, d\MP = \frac{(\lambda^+-\lambda^-)^2}{2 \pi \sigma^2 r} &\left (1 - \frac{\lambda^-}{\lambda^+} \right )^{2k} \int_0^1 (1-w)^{2k + 1/2} w^{1/2} \, dw.
\end{align*}
The integral is a $\beta$-function with parameters $2k + 3/2$ and $3/2$. Applying the asymptotics of $\beta$-functions, finishes this case. 

Lastly consider when $r \neq 1$ and $\ell = 2$. Similar to the previous case, using \eqref{eq:blah_2}, the following holds
\begin{align*} \int \lambda^2 P_k^2(\lambda; \lambda^{\pm}) \, d\MP = \frac{(\lambda^+-\lambda^-)^2}{2 \pi \sigma^2 r} &\left (1 - \frac{\lambda^-}{\lambda^+} \right )^{2k} \Big (\lambda^- \int_0^1 (1-w)^{2k + 1/2} w^{1/2} \, dw\\
    & \qquad \quad +  (\lambda^+-\lambda^-) \int_0^1  (1-w)^{2k + 1/2} w^{3/2} \, dw \Big ).
\end{align*}
The first integral is a $\beta$-function with parameters $2k + 3/2$ and $3/2$ and the second term is a $\beta$-function with parameters $2k + 3/2$ and $5/2$. Again using the asymptotics for $\beta$-functions yields the result. 
\end{proof}

\begin{lemma}[Average-case: Nesterov accelerated method (strongly convex)] Let $d\MP$ be the Mar\v{c}enko-Pastur law defined in \eqref{eq:MP} and $P_k$ be the residual polynomial for Nesterov accelerated method on a strongly convex objective function \eqref{eq:Nesterov_strongly_cvx_poly_Cheby}.  Then the following holds
\begin{equation} \begin{gathered}
    \int \lambda P_k^2(\lambda; \lambda^{\pm}) \, d\MP = \tfrac{(\lambda^+-\lambda^-)^2}{2^5 4^k \sigma^2 r} \left (\beta \big ( 1-\tfrac{\lambda^-}{\lambda^+} \big ) \right )^k \Big [ \tfrac{4\beta^2}{(1+\beta)^2}  \left ( -k^2+\frac{k}{2}+1 - \binom{2k +2}{k} + \binom{2k +2}{k +1} \right )\\
    + 
    \tfrac{4\beta}{1+\beta} \big (1-\tfrac{2\beta}{1+\beta} \big ) \left (2k +1 -  \binom{2k +2}{k} + \binom{2k +2}{k +1} \right ) + 2\big (1-\tfrac{2\beta}{1+\beta} \big )^2 \left (\binom{2k +2}{k +1} -1 \right )
    \Big ]\\
    \sim \tfrac{(\lambda^+-\lambda^-)^2}{4 \sigma^2 r \sqrt{\pi}} \big (1-\tfrac{2\beta}{1+\beta} \big )^2 \left (\beta \big ( 1-\tfrac{\lambda^-}{\lambda^+} \big ) \right )^k \frac{1}{k^{1/2}},
\end{gathered} \end{equation}
and the integral equals
\begin{equation} \begin{gathered}
    \int \lambda^2 P_k^2(\lambda; \lambda^{\pm}) \, d\MP = \lambda^- \int \lambda P_k^2(\lambda; \lambda^{\pm})\, d\MP   + \tfrac{(\lambda^+-\lambda^-)^3}{2^7 4^k \sigma^2 r} \left (\beta \big ( 1-\tfrac{\lambda^-}{\lambda^+} \big ) \right )^k \Big [ \hspace{4cm} \\
   \tfrac{4\beta^2}{(1+\beta)^2}  \left ( \tfrac{1}{3}(2k^3-9k^2+k+6) -4 \binom{2k +2}{k} + \binom{2k +2}{k -1} + 3 \binom{2k +2}{k +1} \right )\\
    + 
    \tfrac{4\beta}{1+\beta} \big (1-\tfrac{2\beta}{1+\beta} \big ) \left (-2k^2+3k+2 - 4 \binom{2k +2}{k} + \binom{2k +2}{k-1} +3 \binom{2k +2}{k +1} \right )\\
    + 4 \big (1-\tfrac{2\beta}{1+\beta} \big )^2 \left ( k-  \binom{2k +2}{k} + \binom{2k +2}{k +1} \right )
    \Big ]\\
    \sim \tfrac{\lambda^- (\lambda^+-\lambda^-)^2}{4 \sigma^2 r \sqrt{\pi}} \big (1-\tfrac{2\beta}{1+\beta} \big )^2 \left (\beta \big ( 1-\tfrac{\lambda^-}{\lambda^+} \big ) \right )^k \frac{1}{k^{1/2}},
\end{gathered} \end{equation}
where $\beta = \frac{\sqrt{\lambda^+}-\sqrt{\lambda^-}}{\sqrt{\lambda^+} + \sqrt{\lambda^-}}$ and $\alpha = \frac{1}{\lambda^+}$.
\end{lemma}

\begin{proof} Throughout this proof, we define $P_k(\lambda)$ to be $P_k(\lambda; \lambda^{\pm})$ in order to simplify the notation. In order to integrate the Chebyshev polynomials, we reduce our integral to trig functions via a series of change of variables. Under the change of variables that sends $\lambda = \lambda^{-} + (\lambda^+-\lambda^-) w$, we have that for any $\ell \ge 1$
\begin{equation} \begin{aligned} \label{eq:Nesterov_blah_1} \int \lambda^{\ell} P_k^2(\lambda) \, d\MP
= 
    \tfrac{(\lambda^+-\lambda^-)^2}{2\pi \sigma^2 r} \! \! \int_0^1 \! \! \! P_k^2(\lambda^- + (\lambda^+-\lambda^-) w) (\lambda^- + (\lambda^+-\lambda^-) w)^{\ell-1} \sqrt{w (1-w)} \, dw.
\end{aligned}
\end{equation}
We note that under this transformation $1-\alpha \lambda = (1-\tfrac{\lambda^-}{\lambda^+}) (1-w)$ and $\tfrac{1+\beta}{2 \sqrt{\beta}} \sqrt{1-\alpha \lambda} = (1-w)^{1/2}$. Moreover by expanding out Nesterov's polynomial \eqref{eq:Nesterov_strongly_cvx_poly_Cheby}, we deduce the following
\begin{equation} \begin{gathered}
    P_k^2(\lambda) = (\beta x)^k \left (  \tfrac{4\beta^2}{(1+\beta)^2} T_k^2(y) + \tfrac{2\beta}{1+\beta} \left (1 - \tfrac{2\beta}{1+\beta} \right )  T_k (y) U_k(y) + \left (1 - \tfrac{4\beta}{1+\beta} \right )^2 U_k^2(y) \right ),\\
    \text{where} \qquad y = \tfrac{1+\beta}{2 \sqrt{\beta}}  \sqrt{x} \quad \text{and} \quad x = 1-\alpha \lambda.
\end{gathered} \end{equation}
First, we consider the setting where $\ell = 1$ in \eqref{eq:Nesterov_blah_1} and hence we deduce that
\begin{equation} \begin{aligned} \label{eq:Nesterov_blah_2}
    \int \lambda P_k^2(\lambda) \, d\MP& = \tfrac{(\lambda^+-\lambda^-)^2}{2\pi \sigma^2 r} \left (\beta \big ( 1-\tfrac{\lambda^-}{\lambda^+} \big ) \right )^k \int_0^1 (1-w)^k \sqrt{w(1-w)} \big [ \tfrac{4\beta^2}{(1+\beta)^2} T_k^2((1-w)^{1/2}) \\
     + \tfrac{4\beta}{1+\beta}& \big (1-\tfrac{2\beta}{1+\beta} \big ) T_k((1-w)^{1/2}) U_k( (1-w)^{1/2}) + \big (1-\tfrac{2\beta}{1+\beta} \big )^2 U_k^2((1-w)^{1/2}) \big ] \, dw\\
    (1-w = \cos^2(\theta)) & \quad \quad \, \, \, = \tfrac{2(\lambda^+-\lambda^-)^2}{2\pi \sigma^2 r} \left (\beta \big ( 1-\tfrac{\lambda^-}{\lambda^+} \big ) \right )^k \int_0^{\pi/2} \cos^{2k+2}(\theta) \sin^2(\theta) \big [ \tfrac{4\beta^2}{(1+\beta)^2} \cos^2(k\theta)\\
    &  \quad \qquad \, \, + \tfrac{4\beta}{1+\beta} \big (1-\tfrac{2\beta}{1+\beta} \big ) \cos(k\theta) \cdot  \tfrac{\sin((k+1)\theta)}{\sin(\theta)} + \big (1-\tfrac{2\beta}{1+\beta} \big )^2 \tfrac{\sin^2((k+1)\theta)}{\sin^2(\theta)} \big ] \, d\theta\\
    \text{(by symmetry)} & \quad \quad  \, = \tfrac{2(\lambda^+-\lambda^-)^2}{8 \pi \sigma^2 r} \left (\beta \big ( 1-\tfrac{\lambda^-}{\lambda^+} \big ) \right )^k \int_0^{2\pi} \cos^{2k+2}(\theta) \sin^2(\theta) \big [ \tfrac{4\beta^2}{(1+\beta)^2} \cos^2(k\theta)\\
    &  \quad \qquad \, \, + \tfrac{4\beta}{1+\beta} \big (1-\tfrac{2\beta}{1+\beta} \big ) \cos(k\theta) \cdot  \tfrac{\sin((k+1)\theta)}{\sin(\theta)} + \big (1-\tfrac{2\beta}{1+\beta} \big )^2 \tfrac{\sin^2((k+1)\theta)}{\sin^2(\theta)} \big ] \, d\theta.
\end{aligned} \end{equation}
We will treat each term in the summand separately. Because the integral of $\int_0^{2\pi} e^{i k\theta} =0 $ for any $k \in \mathbb{N}$, we only need to keep track of the constant terms. From this observation, we get the following
\begin{equation} \begin{gathered}
    \cos^{2k+2}(\theta) \sin^2(\theta)\cos^2(k \theta) = \frac{1}{2^4 4^k} \left ( -k^2+\frac{k}{2} + 1 - \binom{2k +2}{k} + \binom{2k +2}{k +1} \right ) + \text{terms}\\
    \cos^{2k+2}(\theta) \sin(\theta) \cos(k \theta) \sin((k+1)\theta) = \frac{1}{2^4 4^k} \left (2k +1 - \binom{ 2k +2}{k} + \binom{2k +2}{k +1} \right ) + \text{terms}\\
    \cos^{2k+2}(\theta) \sin^2((k+1) \theta) = \frac{1}{2^{4} 4^k} \left ( -2+ 2\binom{2k +2}{k +1} \right ) + \text{terms}
\end{gathered} \end{equation}
We note that $\tfrac{1}{4^k} \left ( \binom{2k+2}{k+1} - \binom{2k+2}{k} \right )  \sim \tfrac{4}{\sqrt{\pi} k^{3/2}}$ and $\tfrac{1}{4^k} \binom{2k+2}{k+1} \sim \frac{4}{\sqrt{\pi} k^{1/2}}$.

Next, we consider the setting where $\ell = 2$ and we observe from \eqref{eq:Nesterov_blah_1} that
\[ \int (\lambda^2-\lambda^- \lambda) P_k^2(\lambda) d\MP =   \tfrac{(\lambda^+-\lambda^-)^3}{2\pi \sigma^2 r}  \int_0^1 P_k^2(\lambda^- + (\lambda^+-\lambda^-) w) w \sqrt{w (1-w)} \, dw. \]
Since we know how to evaluate $\int \lambda P_k^2(\lambda) d\MP$, we only need to analyze the RHS of this integral. A similar analysis as in \eqref{eq:Nesterov_blah_2} applies
\begin{align*}
\int (\lambda^2-\lambda^- \lambda) &P_k^2(\lambda) d\MP  = \tfrac{2(\lambda^+-\lambda^-)^3}{8 \pi \sigma^2 r} \left (\beta \big ( 1-\tfrac{\lambda^-}{\lambda^+} \big ) \right )^k \int_0^{2\pi} \cos^{2k+2}(\theta) \sin^4(\theta) \big [ \tfrac{4\beta^2}{(1+\beta)^2} \cos^2(k\theta)\\
    &  \qquad \qquad \, \, + \tfrac{4\beta}{1+\beta} \big (1-\tfrac{2\beta}{1+\beta} \big ) \cos(k\theta) \cdot  \tfrac{\sin((k+1)\theta)}{\sin(\theta)} + \big (1-\tfrac{2\beta}{1+\beta} \big )^2 \tfrac{\sin^2((k+1)\theta)}{\sin^2(\theta)} \big ] \, d\theta.
\end{align*}
As before, we will treat each term in the summand separately and use that $\int_0^{2\pi} e^{ik\theta} = 0$ to only keep track of the constant terms:
\begin{equation} \begin{gathered}
    \cos^{2k+2}(\theta) \sin^4(\theta)\cos^2(k \theta) = \frac{1}{2^6 4^k} \Bigg ( \tfrac{1}{3}(2k^3-9k^2+k+6)\\
   \qquad \qquad \qquad \qquad \qquad \qquad \qquad  -4 \binom{2k +2}{k} + \binom{2k +2}{k -1} + 3 \binom{2k +2}{k +1} \Bigg ) + \text{terms}\\
    \cos^{2k+2}(\theta) \sin^3(\theta) \cos(k \theta) \sin((k+1)\theta) = \frac{1}{2^7 4^k} \Bigg (-4k^2+6k+4\\
   \qquad \qquad \qquad \qquad \qquad \qquad \qquad - 8 \binom{2k +2}{k} + 2 \binom{2k +2}{k-1} +6 \binom{2k +2}{k +1} \Bigg ) + \text{terms}\\
    \cos^{2k+2}(\theta) \sin^2(\theta) \sin^2((k+1) \theta) = \frac{1}{2^{4} 4^k} \left ( k- \binom{2k +2}{k} + \binom{2k +2}{k +1} \right ) + \text{terms}.\\
\end{gathered} \end{equation}
The result immediately follows. 
\end{proof}

\begin{lemma}[Average-case: Polyak Momentum] Let $d\MP$ be the Mar\v{c}enko-Pastur law defined in \eqref{eq:MP} and $P_k$ be the residual polynomial for Polyak's (heavy-ball) method \eqref{eq: Polyak_poly_apx}. Then the following holds
\begin{equation} \begin{aligned}
    \int \! \! \lambda P_k^2(\lambda; \lambda^{\pm}) \, d\MP &= \tfrac{(\lambda^+-\lambda^-)^2}{32 r \sigma^2} \left ( \tfrac{\sqrt{\lambda^+}-\sqrt{\lambda^-}}{\sqrt{\lambda^+}+\sqrt{\lambda^-}} \right )^{2k}\\
     & \qquad \qquad \times \big [ \left ( \tfrac{(\sqrt{\lambda^+}-\sqrt{\lambda^-})^2}{\lambda^+ + \lambda^-} \right )^2 \! \! \! \! + \! 2 \tfrac{(\sqrt{\lambda^+}-\sqrt{\lambda^-})^2}{\lambda^+ + \lambda^-} \tfrac{2 \sqrt{\lambda^-\lambda^+}}{\lambda^+ + \lambda^-} \! + \! 2 \left ( \tfrac{2 \sqrt{\lambda^-\lambda^+}}{\lambda^+ + \lambda^-} \right )^2  \big ]\\
\text{and} \quad \int \lambda^2 P_k^2(\lambda) \, d\MP &= \frac{\lambda^+-\lambda^-}{2} \int \lambda P_k^2(\lambda; \lambda^{\pm}) \, d\MP.
\end{aligned}
\end{equation}
\end{lemma}

\begin{proof} In order to simplify notation, we define the following
\begin{equation} \begin{gathered}
\beta =  \tfrac{\sqrt{\lambda^+}-\sqrt{\lambda^-}}{\sqrt{\lambda^+}+\sqrt{\lambda^-}}, \quad c = \tfrac{(\sqrt{\lambda^+}-\sqrt{\lambda^-})^2}{\lambda^+ + \lambda^-}, \quad d = \tfrac{2 \sqrt{\lambda^-\lambda^+}}{\lambda^+ + \lambda^-}, \quad \text{and} \quad \sigma(\lambda) = \tfrac{\lambda^++\lambda^- -2\lambda}{\lambda^+- \lambda^-}\\
\text{with} \quad \widetilde{P}^2_k(x) \defas \beta^{2k} \big [ c^2 T_k^2(x) +  2cd \cdot T_k(x)U_k(x) + d^2 U_k^2(x) \big ] \quad \text{and} \quad \widetilde{P}_k(\sigma(\lambda)) = P_k(\lambda; \lambda^{\pm}).
\end{gathered} 
\end{equation}
Under the change of variables, $u = \sigma(\lambda)$, we deduce the following for any $\ell \ge 1$
\begin{equation} \begin{aligned} \label{eq: Polyak_poly_blah}
    \int \lambda^{\ell} P_k^2(\lambda; \lambda^{\pm}) \, \dif\mu_{MP} &= \frac{1}{2 \pi \sigma^2 r} \int_{\lambda^-}^{\lambda^+} \lambda^{\ell-1} \widetilde{P}_k^2( \sigma(\lambda)) \sqrt{(\lambda^+-\lambda)(\lambda-\lambda^-)} \, d\lambda\\
    &= \frac{(\lambda^+-\lambda^-)^2}{8 \pi \sigma^2 r} \int_{-1}^1 (\tfrac{\lambda^+ + \lambda^-}{2} - \tfrac{\lambda^+ - \lambda^-}{2}u)^{\ell-1} \widetilde{P}_k^2(u) \sqrt{1-u^2} \, du. 
\end{aligned}
\end{equation}
First, we consider when $\ell = 1$. We convert this into a trig. integral using the substitution $u = \cos(\theta)$ and its nice relationship with the Chebyshev polynomials. In particular, we deduce the following
\begin{align*}
    \int \lambda P_k^2(\lambda; \lambda^{\pm}) d\MP = \frac{(\lambda^+-\lambda^-)^2}{8 \pi \sigma^2 r} \beta^{2k} \int_0^{\pi} \big [c^2 \cos^2(k\theta)+ 2 cd \tfrac{\cos(k\theta) \sin((k+1) \theta)}{\sin(\theta)} + d^2 \tfrac{\sin^2((k+1) \theta)}{\sin^2(\theta)} \big ] \sin^2(\theta) \, d\theta.
\end{align*}
Treating each term in the summand separately, we can evaluate each integral
\begin{equation} \begin{gathered}
\int_0^{\pi} \cos^2(k \theta) \sin^2(\theta) \, d\theta = \frac{\pi}{4}, \quad \int_0^{\pi} \sin^2((k+1) \theta) \, d\theta = \frac{\pi}{2}, \\
\text{and} \quad \int_0^{\pi} \cos(k\theta) \sin((k+1)\theta) \sin(\theta) d\theta = \frac{\pi}{4}.
\end{gathered}
\end{equation}
The result follows. Next we consider when $\ell = 2$. A quick calculation using \eqref{eq: Polyak_poly_blah} shows that
\begin{align*}
    \int \lambda^2 P_k^2(\lambda; \lambda^{\pm}) \, d\MP = \frac{\lambda^++ \lambda^-}{2} \int \lambda P_k^2(\lambda) \, d\MP - \frac{(\lambda^+-\lambda^-)^3}{16 \pi \sigma^2 r} \int_{-1}^1 u \widetilde{P}_k^2(u) \sqrt{1-u^2} \, du.
\end{align*}
Since we evaluated the first integral, it suffices to analyze the second one. Again, we use a trig. substitution $u = \cos(\theta)$ and deduce the following
\begin{align*}
    \int_{-1}^1 u \widetilde{P}_k^2(u) \sqrt{1-u^2} \, du &= \beta^{2k} \int_0^{\pi} \cos(\theta) \sin^2(\theta) \big [c^2 \cos^2(k\theta)+ 2 cd \tfrac{\cos(k\theta) \sin((k+1) \theta)}{\sin(\theta)} + d^2 \tfrac{\sin^2((k+1) \theta)}{\sin^2(\theta)} \big ] \, d\theta.
\end{align*}
Treating each term in the summand separately, we can evaluate each integral
\begin{equation} \begin{gathered}
\int_0^{\pi} \cos^2(k \theta) \sin^2(\theta) \cos(\theta) \, d\theta = \int_0^{\pi} \sin^2((k+1)  \theta) \cos(\theta) \, d\theta = 0, \\
\text{and} \quad \int_0^{\pi} \cos(k\theta) \sin((k+1)\theta) \sin(\theta) \cos(\theta) d\theta = 0.
\end{gathered}
\end{equation}
\end{proof}

\begin{lemma}[Average-case: Nesterov accelerated method (convex)] Let $d\MP$ be the Mar\v{c}enko-Pastur law defined in \eqref{eq:MP}. Suppose the polynomials $P_k$ are the residual polynomials for Nesterov's accelerated gradient \eqref{eq: Nesterov_poly_convex_1}. If the ratio $r = 1$, the following asymptotic holds
\begin{equation} \int \lambda P_k^2(\lambda; \lambda^{\pm}) \, d\MP \sim \frac{(\lambda^+)^2}{\pi^2 \sigma^2} \frac{\log(k)}{k^3} \quad \text{and} \quad \int \lambda^2 P_k^2(\lambda; \lambda^{\pm}) \, d\MP \sim \frac{2 (\lambda^+)^3}{\pi^2 \sigma^2} \frac{1}{k^4}.
\end{equation}
\end{lemma}

\begin{proof} Define the polynomial $\widetilde{P}_k(u) = P_k( \lambda^+ u; \lambda^{\pm})$ where the polynomial $P_k$ satisfies Nesterov's recurrence \eqref{eq: Nesterov_poly_convex_1}. Using the change of variables $u = \frac{\lambda}{\lambda^+}$, we get the following relationship
\begin{equation} \begin{aligned} \label{eq:Elliot_horrible_bound_1}
    \int_0^{\lambda^+} \lambda^{\ell} P_k^2(\lambda; \lambda^{\pm}) \, d\MP &= \frac{(\lambda^+)^{\ell+1}}{2 \pi \sigma^2} \int_0^1 u^{\ell-1} \widetilde{P}_k^2(u) \sqrt{u(1-u)} \, du\\
    &= \frac{(\lambda^+)^{\ell+1}}{2 \pi \sigma^2} \int_0^1 u^{\ell-1} \frac{4J_1^2(k\sqrt{u})}{k^2u} e^{-uk} \sqrt{u(1-u)} \, du\\
    & \qquad + \frac{(\lambda^+)^{\ell+1}}{2 \pi \sigma^2} \int_0^1 u^{\ell-1} \big [ \widetilde{P}_k^2(u) - \frac{4J_1^2(k\sqrt{u})}{k^2u} e^{-uk} \big ] \sqrt{u(1-u)} \, du.
\end{aligned} \end{equation}
In the equality above, the first integral will become the asymptotic and the second integral we bound using Corollary~\ref{cor: Nesterov_poly_asymptotic}. We start by bounding the second integral. We break this integral into three components based on the value of $u$
\begin{equation} \begin{aligned} \label{eq:Elliot_horrible_bound_2}
\Big ( \int_0^1 = \underbrace{\int_0^{k^{-4/3}}}_{\text{(i)}} + \underbrace{\int_{k^{-4/3}}^{\log^2(k)/k}}_{\text{(ii)}} + \underbrace{\int_{\log^2(k)/k}^1}_{\text{(iii)}} \Big ) u^{\ell-1} \big [\widetilde{P}_k^2(u) - \frac{4J_1^2(\sqrt{u}k)}{u k^2} e^{-uk} \big ] \sqrt{u(1-u)} \, du.
\end{aligned}
\end{equation}
For (i) in \eqref{eq:Elliot_horrible_bound_2}, we bound the integrand using Corollary~\ref{cor: Nesterov_poly_asymptotic} such that for all $u \le k^{-4/3}$
\begin{equation} \begin{aligned}
u^{\ell-1/2} \, \Big |\widetilde{P}_k^2(u) - \frac{4J_1^2(k \sqrt{u})}{u k^2} e^{-uk} \Big | \! \! &
&\le C ( u^{\ell-1/2} k^{-4/3} + k^{-13/6} u^{\ell-5/4} ).
\end{aligned}    
\end{equation}
Therefore, we get that the integral (i) is bounded by
\begin{equation}
    \begin{aligned} \label{eq:Elliot_horrible_5}
    \int_0^{k^{-4/3}} \! \! \! \!\! \!  u^{\ell-1} \big |\widetilde{P}_k^2(u)- \frac{4J_1^2(\sqrt{u}k)}{u k^2} e^{-uk} \big | \sqrt{u(1-u)} \, du &\le C \int_0^{k^{-4/3}} \! \! \! \!\! \!  \!u^{\ell-1/2} k^{-4/3} + k^{-13/6} u^{\ell-5/4} \, du\\
    &= C( k^{-2 - 4/3 \ell} + k^{-11/6 - 4/3 \ell})\\
    &\le C \begin{cases} k^{-19/6}, & \text{if } \ell =1\\
    k^{-9/2}, & \text{if } \ell = 2,
    \end{cases}
    \end{aligned}
\end{equation}
for sufficiently large $k$ and absolute constant $C$. 

For (ii) in \eqref{eq:Elliot_horrible_bound_2}, we bound the integrand using Corollary~\ref{cor: Nesterov_poly_asymptotic} to get for all $k^{-4/3} \le u \le \log^2(k)/k$ we have
\begin{equation} \begin{aligned}
u^{\ell-1/2} \, \Big |\widetilde{P}_k^2(u) - \frac{4J_1^2(k\sqrt{u})}{k^2 u} e^{-uk} \Big | 
\le C e^{-uk} (u^{\ell-3/2} k^{-8/3} + u^{\ell-7/4} k^{-17/6}).
\end{aligned}    
\end{equation}
Therefore, we get that the integral (ii) is bounded by
\begin{equation}
    \begin{aligned} \label{eq:Elliot_horrible_6}
    \int_{k^{-4/3}}^{\log^2(k)/k} u^{\ell-1} \big |\widetilde{P}_k^2(u)- &\frac{4J_1^2(k\sqrt{u})}{k^2 u} e^{-uk} \big | \sqrt{u(1-u)} \, du\\
    &\le C \int_{k^{-4/3}}^{\log^2(k)/k} e^{-uk} (u^{\ell-3/2} k^{-8/3} + u^{\ell-7/4} k^{-17/6}) \, du\\
  \text{($v = uk$)} \quad  &\le C \int_0^\infty e^{-v} v^{\ell-3/2} k^{-(\ell+13/6)} \, dv + C \int_0^\infty e^{-v} v^{\ell-7/4} k^{-(\ell + 25/12)} \, dv \\
  &= C(k^{-(\ell + 13/6)} + k^{-(\ell + 25/12)} )\\
  &\le C \begin{cases}
  k^{-37/12}, & \text{if } \ell = 1,\\
  k^{-49/12}, & \text{if } \ell =2.
  \end{cases}
    \end{aligned}
\end{equation}

For (iii) in \eqref{eq:Elliot_horrible_bound_2}, we use a simple bound on the functions $\widetilde{P}_k(u) = \frac{2e^{-ku/2}}{k\sqrt{u}} I_k(u)$ where $I_k(u)$ is defined in \eqref{eq:I_k} and $J_1(k\sqrt{u})$ for $u \ge \log^2(k)/k$
\begin{equation}
    \begin{aligned} \label{eq:Elliot_horrible_3}
    \big | \widetilde{P}_k^2(u) - \frac{4 J_1^2(k\sqrt{u})}{k^2u} e^{-uk} \big | \le \frac{4(1-u)^{k+1}}{k^2 u} I_k^2(u) + e^{-uk} \frac{4J_1^2(k\sqrt{u})}{k^2u} \le \frac{C e^{-\log^2(k)}}{k \log^2(k)}.
    \end{aligned}
\end{equation}
In the last inequality, we used that the functions $I_k(u)$ and $J_1(k\sqrt{u})$ are bounded and $(1-u)^{k+1} \le e^{-uk}$. Since $e^{-\log^2(k)}$ decays faster than any polynomial, we have that 
\begin{equation} \label{eq:Elliot_horrible_4}
    \int_{\log^2(k)/k}^1 \big |P_k^2(u) - \frac{4J_1^2(k\sqrt{u})}{k^2u} e^{-ku} \big | u^{\ell-1} \sqrt{u(1-u)} \, du \le C e^{-\log^2(k)}
\end{equation}
for sufficiently large $k$ and some absolute constant $C$.

It follows by combining \eqref{eq:Elliot_horrible_5}, \eqref{eq:Elliot_horrible_6}, and \eqref{eq:Elliot_horrible_4} into \eqref{eq:Elliot_horrible_bound_2} we have for $\ell \ge 1$ the following
\begin{equation} \label{eq:Elliot_horrible_7}
\Big | \int_0^1 \Big [\widetilde{P}_k^2(u)- \frac{4J_1(k\sqrt{u})}{k^2u} e^{-ku} \Big ] u^{\ell-1} \sqrt{u(1-u)} \, du \Big | \le Ck^{-(\ell + 25/12)}.
\end{equation}

All that remains is to integrate the Bessel part in \eqref{eq:Elliot_horrible_bound_1} to derive the asymptotic. Here we must consider cases when $\ell =1$ and $\ell =2$ separately. For $\ell = 1$ using the change of variables $v = k \sqrt{u}$ we have that
\begin{align*}
    \frac{(\lambda^+)^2}{2\pi \sigma^2} \int_0^1 \left ( \frac{2J_1(k\sqrt{u})}{k \sqrt{u}} \right )^2 e^{-uk} \sqrt{u(1-u)} \, du &= \frac{2 \cdot 4 (\lambda^+)^2}{2 \pi \sigma^2} \frac{1}{k^3} \int_0^k J_1^2(v) e^{-v^2/k} \sqrt{1-v^2/k^2} \, dv\\
    \text{($\sqrt{1-x} \approx 1-x$ for $x$ small)} \qquad &\sim \frac{2 \cdot 4 \cdot (\lambda^+)^2}{2 \pi \sigma^2} \frac{1}{k^3} \int_1^{\infty} J_1^2(v) e^{-v^2/k} \, dv\\
   \text{(Bessel asymptotic, $J_1^2(v) \sim \tfrac{1}{\pi v}$)} \qquad &\sim \frac{2 \cdot 4 \cdot (\lambda^+)^2}{2 \pi \sigma^2} \frac{1}{k^3} \int_1^\infty \frac{1}{\pi v} e^{-v^2/k} \, dv\\
   &= \frac{2 \cdot 4 \cdot (\lambda^+)^2}{2 \pi \sigma^2} \cdot \frac{1}{k^3} \cdot \frac{-1}{2\pi} \mathcal{E}_i(-\sqrt{k}),
\end{align*}
where $\mathcal{E}_i$ is the exponential integral. It is known that the exponential integral $\frac{-1}{2 \pi} \mathcal{E}_i(-\sqrt{k}) \sim \frac{\log(k)}{4\pi}$.

For $\ell =2$ using the change of variables $v = uk$ we have the following 
\begin{align*}
    \frac{(\lambda^+)^{3}}{2 \pi \sigma^2} \int_0^1 \frac{4 J_1^2(k\sqrt{u} )}{k^2u} e^{-uk} u \sqrt{u(1-u)} \, du &= \frac{4 \cdot (\lambda^+)^3}{2 \pi \sigma^2} \cdot \frac{1}{k^{7/2}} \int_0^k e^{-v} J_1^2( \sqrt{vk}) v^{1/2} \sqrt{1-\frac{v}{k}} \, dv\\
    &\sim \frac{2 \cdot 4 \cdot (\lambda^+)^3 }{2 \pi^2 \sigma^2} \cdot \frac{1}{k^{4}} \int_0^{\infty} \cos^2(\sqrt{vk} + C) e^{-v} \, dv\\
    &\sim \frac{4 \cdot (\lambda^+)^3}{2 \pi^2 \sigma^2} \cdot \frac{1}{k^4}  \int_0^\infty e^{-v} \, dv = \frac{4 \cdot (\lambda^+)^3}{2 \pi^2 \sigma^2} \cdot \frac{1}{k^4}.
\end{align*}
The results follow. 
\end{proof}

\section{Adversarial model computations} \label{apx: adversarial_model}
In this section, we derive the adversarial guarantees for gradient descent and Nesterov's accelerated method.

\begin{lemma}[Adversarial model: Gradient descent] Suppose Assumption~\ref{assumption: Vector} holds. Let $\lambda^+$ ($\lambda^-$) be the upper (lower) edge of the Mar\v{c}enko Pastur distribution \eqref{eq:MP} and $P_k$ the residual polynomial for gradient descent. Then the adversarial model for the maximal expected squared norm of the gradient is the following. 
\begin{enumerate}
\item If there is no noise $\widetilde{R} = 0$, then
\begin{align*}
    \lim_{d \to \infty} \max_{\HH} \mathbb{E} \big [ \|\nabla f(\xx_k)\|^2 \big ] \sim \begin{cases}
    \frac{R^2 (\lambda^+)^2}{(k+1)^2} e^{-2}, & \text{if $\lambda^- = 0$}\\
    R^2 (\lambda^-)^2 \left (1 - \frac{\lambda^-}{\lambda^+} \right )^{2k}, & \text{if $\lambda^- > 0$}.
    \end{cases}
\end{align*}
\item If $\widetilde{R} > 0$, then the following holds
\begin{align*}
    \lim_{d \to \infty} \max_{\HH} \mathbb{E} \big [ \| \nabla f(\xx_k) \|^2 \big ] \sim \begin{cases}
    \left [ \frac{R^2 (\lambda^+)^2}{4} \frac{1}{k^2} + \frac{ \widetilde{R}^2 \lambda^+}{2} \frac{1}{k} \right ] e^{-2}, & \text{if $\lambda^- = 0$}\\
    \big [ R^2 (\lambda^-)^2 + r \widetilde{R}^2 \lambda^- \big ] \big (1- \frac{\lambda^-}{\lambda^+} \big )^{2k}, & \text{if $\lambda^- > 0$}.
    \end{cases}
\end{align*}
\end{enumerate}
\end{lemma}

\begin{proof} Suppose we are in the noiseless setting. By a change of variables, setting $u = \lambda/\lambda^+$, the following holds
\begin{equation} \label{eq:adversarial_GD}
    \max_{\lambda \in [\lambda^-, \lambda^+]} \lambda^2 \big (1- \frac{\lambda}{\lambda^+} \big )^{2k} = \max_{u \in \big [\frac{\lambda^-}{\lambda^+}, 1 \big ]} (\lambda^+)^2 u^2 (1-u)^{2k}.
\end{equation}
Taking derivatives, we get that the maximum of the RHS occurs when $u = \frac{1}{k+1}$. For sufficiently large $k$ and $\lambda^- > 0$, the maximum lies outside the constraint of $\big [\tfrac{\lambda^-}{\lambda^+}, 1 \big ]$. Hence the maximum occurs on the boundary, or equivalently, where $u = \tfrac{\lambda^-}{\lambda^+}$. The result in the setting when $\lambda^- > 0$ immediately follows from this. When $\lambda^- = 0$, then the maximum does occur at $\frac{1}{k+1}$. Plugging this value into the RHS of \eqref{eq:adversarial_GD} and noting that for sufficiently large $k$, $(1-1/(k+1))^{2k} \to e^{-2}$, we get the other result for noiseless case. 

Now suppose that $\widetilde{R} > 0$. By a change variables, setting $u = \lambda / \lambda^+$, we have that
\begin{equation} \label{eq:adversarial_GD_1}
    \max_{\lambda \in [\lambda^-, \lambda^+]}~ \left ( R^2 \lambda^2 + r\widetilde{R}^2 \lambda \right ) \left(1- \frac{\lambda}{\lambda^+} \right )^{2k} = \max_{u \in \big [\tfrac{\lambda^-}{\lambda^+}, 1 \big ]} \Big \{ h(u) \defas \lambda^+ \big ( R^2 \lambda^+ u^2 + r \widetilde{R}^2 u \big )  (1-u)^{2k} \Big \}.
\end{equation}
The derivative $h'(u)$ equals $0$ at $u = 1$ (local minimum) and at solutions to the quadratic 
\begin{equation*}
    2 R^2 \lambda^+ (k+1) u^2 + [2r \widetilde{R}^2 k + r \widetilde{R}^2 - 2R^2 \lambda^+] u - r \widetilde{R}^2 = 0.
\end{equation*}
There is only one positive root of this quadratic so
\begin{equation} \label{eq:adversarial_GD_2}
    \lambda^* = \frac{\sqrt{(2r \widetilde{R}^2 k + r\widetilde{R}^2 - 2R^2 \lambda^+)^2 + 8 r \widetilde{R}^2 R^2 \lambda^+ (k+1)} - \big [2r \widetilde{R}^2 k + r \widetilde{R}^2 - 2R^2 \lambda^+ \big ] }{4 R^2 \lambda^+ (k+1)}.
\end{equation}
We can approximate the square root using Taylor approximation to get that 
\begin{align*}
    \frac{1}{k} \sqrt{(2r \widetilde{R}^2 k + r\widetilde{R}^2 - 2R^2 \lambda^+)^2 + 8 r \widetilde{R}^2 R^2 \lambda^+ (k+1)} &= 2r\widetilde{R}^2 \Big [ 1 + \frac{r\widetilde{R}^2-2R^2\lambda^+}{r \widetilde{R}^2k} + \frac{2 R^2 \lambda^+}{r \widetilde{R}^2 k} + \mathcal{O} ( k^{-2} ) \Big ]^{1/2} \\
    \text{(Taylor approximation)} \quad &= 2r\widetilde{R}^2 \Big [ 1 + \frac{1}{2k} + \mathcal{O}(k^{-2}) \Big ].
\end{align*}
Putting this together into \eqref{eq:adversarial_GD_2}, we get that
\begin{align*}
   &\frac{\sqrt{(2r \widetilde{R}^2 k + r\widetilde{R}^2 - 2R^2 \lambda^+)^2 + 8 r \widetilde{R}^2 R^2 \lambda^+ (k+1)} - \big [2r \widetilde{R}^2 k + r \widetilde{R}^2 - 2R^2 \lambda^+ \big ] }{4 R^2 \lambda^+ (k+1)}\\
& \qquad \qquad \qquad \qquad = \frac{\frac{r \widetilde{R}^2}{k} + \mathcal{O}(k^{-2}) - \frac{r\widetilde{R}^2}{k} + \frac{2 R^2 \lambda^+}{k}}{4 R^2 \lambda^+ + \frac{4 R^2 \lambda^+}{k}}
\sim \frac{1}{2k}.
\end{align*}
As before for sufficiently large $k$ and $\lambda^- > 0$, the root above lies outside the constraint of $\big [ \frac{\lambda^-}{\lambda^+}, 1 \big ]$ and so maximum occurs on the boundary, or equivalently, $u = \frac{\lambda^-}{\lambda^+}$. The result immediately follows by plugging this $u$ into \eqref{eq:adversarial_GD_2}. When $\lambda^- =0$, then the maximum is the root of the above quadratic which asymptotically equals $1/k$. Plugging this value into \eqref{eq:adversarial_GD_2} and noting for sufficiently large $k$ that $(1-1/k)^{2k} \approx e^{-2}$, we get the result for the noiseless setting. 
\end{proof}

\begin{lemma}[Adversarial model: Nesterov (convex)] Suppose Assumption~\ref{assumption: Vector} holds. Let $\lambda^+$ be the upper edge of the Mar\v{c}enko Pastur distribution \eqref{eq:MP} and $P_k$ the residual polynomial for gradient descent. Suppose $r = 1$. Then the adversarial model for the maximal expected squared norm of the gradient is the following. 
\begin{enumerate}
\item If there is no noise $\widetilde{R} = 0$, then
\begin{align*}
    \lim_{d \to \infty} \max_{\HH} \mathbb{E} \big [ \|\nabla f(\xx_k)\|^2 \big ] \sim 
    \frac{8e^{-1/2}}{\sqrt{2}\pi} (\lambda^+)^2 R^2 \frac{1}{k^{7/2}}.
\end{align*}
\item If $\widetilde{R} > 0$, then the following holds
\begin{align*}
    \lim_{d \to \infty} \max_{\HH} \mathbb{E} \big [ \| \nabla f(\xx_k) \|^2 \big ] \sim 
   \|J_1^2(x)\|_{\infty}  (\lambda^+)\widetilde{R}^2 \frac{1}{k^2}.
\end{align*}
\end{enumerate}
\end{lemma}

\begin{proof} First, we claim that \begin{equation} \label{eq: adversarial_Nesterov_1}
\|\lambda^2 P_k^2(\lambda; \lambda^{\pm})\|_{\infty} k^{7/2} 
\to \frac{8}{\pi} (\lambda^+)^2 \max_{x \geq 0} \{ x^{1/2}e^{-x} \} 
= \frac{8e^{-1/2}}{\sqrt{2}\pi} (\lambda^+)^2
\quad \text{as $k \to \infty$}.
\end{equation}
Using the definitions in \eqref{eq:tilde_P} and \eqref{eq:I_k} for $\widetilde{P}_k(u)$ and $I_k(u)$ respectively, we can write 
\[
P_k(\lambda^+ u; \lambda^{\pm}) = \widetilde{P}_k(u) = \frac{2(1-u)^{(k+1)/2}}{k \sqrt{u}} I_k(u). \]
Now by a change of variables we have the following
\begin{align}
    \max_{\lambda \in [0, \lambda^+]} \lambda^2 P_k^2(\lambda; \lambda^{\pm}) k^{7/2} 
    &= \max_{u \in [0,1]} (\lambda^+)^2 u^2 \widetilde{P}_k^2(u) k^{7/2} \nonumber \\
    &= \max \big \{ \! \! \! \! \! \max_{u \in [0, \frac{\log^2(k)}{k} ]} \! \! \! \!  (\lambda^+)^2 k^{7/2} u^2 \widetilde{P}_k^2(u), \! \! \max_{u \in [\frac{\log^2(k)}{k}, 1]} (\lambda^+)^2 k^{7/2} u^2 \widetilde{P}_k^2(u) \big \}. \label{eq:adversarial_Nesterov_2}
\end{align}
Let's first consider the second term in the maximum. Here we use that $|I_k(u)|$ is bounded so that 
\begin{align*}
  \max_{u \in [\frac{\log^2(k)}{k}, 1]} \! \! (\lambda^+)^2 k^{7/2} u^2 \widetilde{P}_k^2(u) 
  = \! \! \! \! \! \max_{u \in [\frac{\log^2(k)}{k}, 1]} \! \! \! 4 (\lambda^+)^2 k^{3/2} u (1-u)^{k+1} I_k^2(u) 
  &\le \! \! \! \! \! \max_{u \in [\frac{\log^2(k)}{k}, 1]} \! \! \! C (\lambda^+)^2 k^{3/2} u (1-u)^{k+1} \\
  &\defas \! \! \! \! \! \max_{u \in [\frac{\log^2(k)}{k}, 1]} \! \! \! h(u).
\end{align*}
The function $h(u)$ is maximized when $u = \tfrac{1}{k+2}$ and hence the maximum over the constrained set occurs at the endpoint $\tfrac{\log^2(k)}{k}$. With this value, it is immediately clear that the maximum over $u \in [ \tfrac{\log^2(k)}{k}, 1]$ of $(\lambda^+) k^{7/2} u^2 \widetilde{P}_k^2(u) \to 0$. Now we consider the first term in \eqref{eq:adversarial_Nesterov_2}. In this regime, the polynomial $\widetilde{P}_k^2(u)$ behaves like the Bessel function in \eqref{eq:Bessel_asymptotic_main}. We further break up the interval $[0, \log^2(k)/k]$ into larger or smaller than $k^{-4/3}$. When $u \in [0, k^{-4/3}]$, Corollary~\ref{cor: Nesterov_poly_asymptotic} says there exists a constant $C$ such that
\begin{align*}
    \max_{u \in [0, k^{-4/3}]} (\lambda^+)^2 u^2 k^{7/2} \big |\widetilde{P}_k^2(u)-\tfrac{4e^{-uk}J_1^2(k \sqrt{u})}{k^2 u} \big | \le C (\lambda^+)^2 [k^{-1/2} + k^{-1/3}] \to 0.
\end{align*}
Similarly when $u \in [k^{-4/3}, \log^2(k)/k]$, Corollary~\ref{cor: Nesterov_poly_asymptotic} yields that 
\begin{align*}
    \max_{u \in [k^{-4/3}, \log^2(k)/k]} \! \!\! \! \! \! \! \!\! \!\!(\lambda^+)^2 u^2 k^{7/2} \big |\widetilde{P}_k^2(u) - \tfrac{4 e^{-uk} J_1^2(k \sqrt{u})}{k^2 u} \big | 
    \le  \! \!\! \! \! \!\! \max_{u \in [k^{-4/3}, \log^2(k)/k]}  \! \!\! \! \! \! \! \!\! \!\! C (\lambda^+)^2 (u k^{5/6} + u^{3/4} k^{2/3} ) 
    \to 0.
\end{align*}

Using a change of variables, the relevant asymptotic to compute is
\[ 
\max_{u \in [0, \log^2(k)/k]} \! \! 4(\lambda^+)^2 u k^{3/2} J_1^2(k \sqrt{u})e^{-uk}
=
\max_{x \in [0, \log^2(k)]} \! \! 4(\lambda^+)^2 (\sqrt{(x k)}J_1^2(\sqrt{kx})) \sqrt{x} e^{-x}.
\]
From the uniform boundedness of the function $y \mapsto \sqrt{y}J_1^2(\sqrt{y}),$
there is a constant $\mathcal{C} > 0$ so that
\[
\max_{x \in [0, \delta]}  4(\lambda^+)^2 (\sqrt{(x k)}J_1^2(\sqrt{kx})) \sqrt{x} e^{-x} \leq 4(\lambda^+)^2 \mathcal{C}\sqrt{\delta}.
\]
Moreover, the Bessel function satisfies
\[
J_1(z) = \sqrt{\frac{2}{\pi z}}\cos( z - \tfrac{3\pi}{4}) + O(z^{-3/2}),
\]
and so for any fixed $\delta>0$
\[
\max_{x \in [\delta, \log^2(k)]} \! \! 4(\lambda^+)^2 (\sqrt{(x k)}J_1^2(\sqrt{kx})) \sqrt{x} e^{-x}
\to
\max_{x \in [\delta, \infty]} \biggl\{  \frac{8}{\pi}(\lambda^+)^2 \sqrt{x} e^{-x} \biggr\}.
\]
As $\delta > 0$ is arbitrary, picking it sufficiently small completes the claim. 

Next we claim that the following holds
\begin{equation}
    \max_{\lambda \in [0, \lambda^+]} k^2 \lambda P_k^2(\lambda; \lambda^{\pm}) = \max_{u \in [0,1]} k^2 \lambda^+ u \widetilde{P}_k^2(u)  \sim \lambda^+ \|J_1(u)\|^2_{\infty} \quad \text{as $k \to \infty$.}
\end{equation}
A similar argument as above using that in this regime the exponential dominates the polynomial $\widetilde{P}_k^2(u)$ we have 
\[ \max_{u \in [\log^2(k)/k, 1]} \! \! \! \! k^2 \lambda^+ u \widetilde{P}_k^2(u) \to 0 \quad \text{as $k \to \infty$}.  \]
Now we need to consider the regime where the Bessel function \eqref{eq:Bessel_asymptotic_main} becomes important. We use our asymptotic in Corollary~\ref{cor: Nesterov_poly_asymptotic} to show that the polynomial is close to the Bessel, namely, 
\begin{equation}
\begin{gathered}
    \max_{u \in [0, k^{-4/3}]} \! \! \! \! \lambda^+ k^2 u \big | \widetilde{P}_k^2(u) - \tfrac{4 e^{-uk} J_1^2(k \sqrt{u})}{k^2 u} \big | \le C \lambda^+ [k^{-2/3} + k^{-5/2}] \to 0 \quad \text{as $k \to \infty$} \\
    \text{and} \quad \max_{u \in [k^{-4/3}, \log^2(k)/k]} \! \! \! \! \! \! \! \lambda^+ k^2 u \big | \widetilde{P}_k^2(u) - \tfrac{4 e^{-uk} J_1^2(k \sqrt{u})}{k^2 u} \big | \le C \lambda^+ [k^{-2/3} + k^{-1/2}] \to 0 \quad \text{as $k \to \infty$.}
\end{gathered}
\end{equation}
It remains to compute the maximum of the Bessel equation in \eqref{eq:Bessel_asymptotic_main} for $u$ in $0$ to $\log^2(k)/k$. Now there exists an absolute constant $\mathcal{C}$ so that $|J_1(x)^2| \le \tfrac{\mathcal{C}}{|x|}$ and there is also an $\eta > 0$ so that $\displaystyle \max_{0 \le x \le \tfrac{1}{\eta}} |J_1^2(x)| > \eta$. Moreover, the maximizer of $J_1^2(x)$ exists. By picking $R$ sufficiently large, we see that 
\[ \max_{u > \tfrac{R}{k^2}} 4 \lambda^+ e^{-ku} J_1^2(k \sqrt{u}) \le \frac{4 \mathcal{C} \lambda^+}{k^2 u} \Big |_{u = R/k^2} = \frac{4\mathcal{C} \lambda^+}{R} < \eta.\]
This means that the maximum must occur for $u$ between $0$ and $R/k^2$. Hence, by picking $R$ sufficiently large, we have the following
\begin{align*}
    \max_{u \in [0, R/k^2]} \! \! \! \! \! \!  4\lambda^+ e^{-ku} J_1^2(k \sqrt{u}) 
    &\to \max_{x \in [0,R]} 4 \lambda^+ J_1^2(x) = 4 \lambda^+ \|J_1^2(x)\|_{\infty}.
\end{align*}
Consequently, for sufficiently large $k$, the maximizer of
\begin{equation} \begin{aligned} \label{eq:adversarial_Nesterov_4}  \max_{u \in [0,1]} (\lambda^+)^2 R^2 u^2  \widetilde{P}_k^2(u) + \lambda^+ \widetilde{R}^2 r u\widetilde{P}_k^2(u) &= \! \! \! \! \! \! \max_{u \in [0, \log^2(k)/k]} \! \! \! \! \! \!  4\lambda^+ e^{-ku} J_1^2(k \sqrt{u})\\
&\to \max_{x \in [0,R]} 4 \lambda^+ J_1^2(x) = 4 \lambda^+ \|J_1^2(x)\|_{\infty}.
\end{aligned}
\end{equation}
\end{proof}

\section{Simulation details} \label{apx:exp_details}

For Figures~\ref{fig:gd-ls} and~\ref{fig:halt_time_concentrates}, which show that the halting time concentrates, we perform $\frac{2^{12}}{\sqrt{d}}$ training runs for each value of $d$. In our initial simulations we observed that the empirical standard deviation was decreasing as $d^{-1/2}$ as the model grows. Because the larger models have a significant runtime, but very little variance in the halting time, we decided to scale the number of experiments based on this estimate of the variance.

As discussed in the text, the Student's $t$-distribution can produce ill-conditioned matrices with large halting times. To make the numerical experiments feasible we limit the number of iterations to 1000 steps for the GD and Nesterov experiments and discard the very few runs that have not converged by this time (less than 0.1\%). 


For Figure~\ref{fig:avg_rates}, which shows the convergence rates, we trained 8192 models for $d = n = 4096$ steps both with ($\widetilde{R}^2 = 0.05$) and without noise. The convergence rates were estimated by fitting a line to the second half of the log-log curve.

For each run we calculate the worst-case upper bound on $\|\nabla f(\xx_k)\|^2$ at $k = n$ using~\citet[Conjecture 3]{taylor2017smooth}.\[
\|\nabla f(\xx_k)\|^2 \leq \frac{L^2 \|\xx_0 - \xx^{\star}\|^2}{(k + 1)^2} \defas \mathrm{UB}_{\text{cvx}}(\|\nabla f(\xx_k)\|^2)
\]
where $\xx^{\star}$ is the argmin of $f$ calculated using the linear solver in JAX~\citep{jax2018github}. To visualize the difference between the worst-case and average-case rates, we draw a log-log histogram of the ratio, \[\frac{\mathrm{UB}_{\text{cvx}}(\|\nabla f(\xx_k)\|^2)}{\|\nabla f(\xx_k)\|^2}.\]

\subsection{Step sizes}\label{sec:step_sizes}
In this appendix section, we discuss our choices of step sizes for logistic regression and stochastic gradient descent (SGD). 

\subsubsection{Logistic regression}

For both gradient descent and Nesterov's accelerated method (convex) on the least squares problem we use the step size $\frac{1}{L}$. The Lipschitz constant, $L$, is equal to the largest eigenvalue of $\HH$ which can be quickly approximated using a power iteration method.

For logistic regression the Hessian is equal to $\AA^T\DD\AA$, where $\DD$ is the Jacobian matrix of the sigmoid activation function. Hence, the Hessian's eigenvalues are equal to those of $\HH$ scaled by the diagonal entries $\DD_{ii} = \sigma((\AA\xx)_i)(1 - \sigma(\AA\xx)_i))$. Since the maximum value of these entries is $\frac{1}{4}$ we use a step size of $\frac{4}{L}$ for our logistic regression experiments.

\subsubsection{Stochastic gradient descent (SGD)}
The least squares problem \eqref{eq:LS} can be reformulated as
\begin{equation}
\min_{\xx \in \RR^d}~ \frac{1}{2n} \|\AA \xx-\bb\|^2 = \frac{1}{2n} \sum_{i=1}^n (\aa_i \xx-b_i)^2\,,
\end{equation}
where $\aa_i$ is the $i$th row of the matrix $\AA$. We perform a mini-batch SGD, \textit{i.e.}, at each iteration we select uniformly at random a subset of the samples $b_k \subset \{1, \hdots, n\}$ and perform the update
\begin{equation}
     \xx_{k+1} = \xx_k - \frac{\alpha}{|b_k|} \sum_{i \in b_k} \nabla f_i(\xx_k)\,.
\end{equation}
With a slight abuse of notation, we denote by $\nabla f_i(\xx_k) = \frac{1}{|b_k|} \sum_{i \in b_k} \nabla f_i(\xx_k)$ the update direction and use the shorthand $b = |b_k|$ for the mini-batch size since it is fixed across iterations. The rest of this section is devoted to choosing the step size $\alpha$ so that the halting time is consistent across dimensions $n$ and $d$. Contrary to (full) gradient descent, the step size in SGD is dimension-dependent because a typical step size in SGD uses the variance in the gradients which grows as the dimension $d$ increases.

\paragraph{Over-parametrized.}

If $n \leq d$ we call the model over-parametrized. In this case, the strong growth condition from~\citet{schmidt2013fast} holds. This implies that training will converge when we use a fixed step size $\frac{2}{LB^2}$ where $B$ is defined as a constant verifying for all $\xx$
\begin{equation}
\max_i \left\{\|\nabla f_i(\xx)\|\right\} \leq B\|\nabla f(\xx)\|\,.
\end{equation}
To estimate $B$ we will compute the expected values of $\|\nabla f_i(\xx)\|^2$ and $\|\nabla f(\xx)\|^2$. To simplify the derivation we will assume that $\widetilde{\xx}$ and $\eeta$ are normally distributed. At iterate $\xx$ we then have \begin{align}
    \nabla f(\xx) &= \frac{1}{n} \AA^T(\AA(\xx - \widetilde{\xx}) - \eeta) \\
    \nabla f(\xx) &\sim N\left(\HH\xx,\frac{1}{d}\HH^2+\frac{\widetilde{R}^2}{n}\HH\right).
\end{align}
Hence the expected value of $\|\nabla f(\xx)\|^2$ is $\|\HH\xx\|^2 + \text{\rm tr }\left(\frac{1}{d}\HH^2+\frac{\widetilde{R}^2}{n}\HH\right)$. Following~\citet[Equation 3.1.6 and Lemma 3.1]{bai2010spectral} we know that for large values of $n$ and $d$, the expected trace $\frac{1}{d}\text{\rm tr }\HH \approx 1$ and $\frac{1}{d}\text{\rm tr }\HH^2 \approx 1 + r$. Further,  $\EE \left[ \|\HH\xx\|^2\right] = (1 + r)\|\xx\|^2$ and hence
\begin{equation}
\begin{split}
\EE\left[\|\nabla f(\xx)\|^2\right] &\approx (1 + r)\|\xx\|^2 + (1 + r) + r\widetilde{R}^2 \\
 &= (1 + r)(1 + \|\xx\|^2) + r\widetilde{R}^2\,.
\end{split}
\end{equation}
We can approximate the same value for a mini-batch gradient, where
\begin{equation}
\EE\left[\|\nabla f_i(\xx)\|^2\right] \approx (1 + r')(1 + \|\xx\|^2) + r'\widetilde{R}^2\,,
\end{equation}
for batch size $b$ and $r' = \frac{d}{b}$. Note that $\|\xx\|\approx 1$ because of the normalization of both the initial point and the solution, so for our experiments we set $B^2 = \frac{2 + r'(2 + \widetilde{R}^2)}{2+r(2 + \widetilde{R}^2)}$.

\paragraph{Under-parametrized.}
In the under-parametrized case SGD will not converge but reach a stationary distribution around the optimum. Given a step size $\overline{\alpha} \leq \frac{1}{LM_G}$ the expected square norm of the mini-batch gradients will converge to $\overline{\alpha}LM$ where $M$ and $M_G$ are constants such that $\EE\left[\|\nabla f_i(\xx)\|^2\right] \leq M + M_G\|\nabla f(\xx)\|^2$~\cite[Theorem 4.8, Equation 4.28]{bottou2018optimization}. We will use rough approximations of both $M$ and $M_G$. In fact, we will set $M_G = B^2 = \frac{2 + 3r'}{2+3r}$.

To approximate $M$ we will estimate the norm of the mini-batch gradients at the optimum for our least squares model. Set $\xx^* = \AA^+\bb = \AA^+\eeta + \widetilde{\xx}$ where $\AA^+$ is the Moore-Penrose pseudoinverse. We will write the row-sampled matrix $\widetilde{\AA}$ in mini-batch SGD $\widetilde{\AA} = \PP\AA$ where $\PP$ consists of exactly $b$ rows of the identity matrix. Note that $\PP^T\PP$ is idempotent. \begin{align*}
\widetilde{\nabla}f(\xx^*) &= \frac{1}{b}\widetilde{\AA}^T(\widetilde{\AA}(\AA^+\eeta +\widetilde{\xx} -\widetilde{\xx}) - \widetilde{\eeta}) \\
&= \frac{1}{b}\AA^T\PP^T(\PP\AA\AA^+\eeta - \PP\eeta) \\
&= \frac{1}{b}\AA^T\PP^T\PP(\AA\AA^+  - \boldsymbol{I})\eeta\,.
\end{align*}
To simplify the derivation we will again assume that $\eeta$ is
normally distributed and that $\widetilde{R} = 1$. Thus we have
\begin{align}
\widetilde{\nabla}f(\xx^*) &\sim N\left(0, \frac{1}{b^2}
                             \AA^T\PP^T\PP(\AA\AA^+  - \boldsymbol{I})(\AA\AA^+  -
                             \boldsymbol{I})^T\PP^T\PP\AA\right)\,.
\end{align}
By taking the expectation of the squared norm, we derive the following
\begin{align*}
\EE\left[\|\widetilde{\nabla}f(\xx^*)\|^2\right] &= \frac{1}{b^2} \mathrm{tr}\left(\AA^T\PP^T\PP(\AA\AA^+  - \boldsymbol{I})(\AA\AA^+  - \boldsymbol{I})^T\PP^T\PP\AA\right) \\
&= \frac{1}{b^2} \mathrm{tr}\left(\AA^T\PP^T\PP\AA\AA^+\AA^{+T}\AA^T\PP^T\PP\AA\right) + \\
& \quad\quad \frac{1}{b^2}\mathrm{tr}\left(\AA^T\PP^T\PP\PP^T\PP\AA\right) -\\
& \quad\quad \frac{2}{b^2}\mathrm{tr}\left(\AA^T\PP^T\PP\AA\AA^+\PP^T\PP\AA\right)\\
&= \mathrm{tr}\left(\widetilde{\HH}^2\AA^+\AA^{+T}\right) + \frac{1}{b}\mathrm{tr}\left(\widetilde{\HH}\right) - \frac{2}{b}\mathrm{tr}\left(\widetilde{\HH}\AA^+\PP^T\PP\AA\right)\\
&= \frac{1}{n} \mathrm{tr}\left(\widetilde{\HH}^2\HH^+\right) + \frac{1}{b}\mathrm{tr}\left(\widetilde{\HH}\right) - \frac{2}{b}\mathrm{tr}\left(\widetilde{\HH}{\AA^+\AA^{+T}\AA^T}\PP^T\PP\AA\right)\\
&= \frac{1}{n} \mathrm{tr}\left(\widetilde{\HH}^2\HH^+\right) + \frac{1}{b}\mathrm{tr}\left(\widetilde{\HH}\right) - 2\mathrm{tr}\left(\widetilde{\HH}^2\AA^+\AA^{+T}\right)\\
&= \frac{1}{b}\mathrm{tr}\left(\widetilde{\HH}\right) - \frac{1}{n}\mathrm{tr}\left(\widetilde{\HH}^2\HH^+\right).
\end{align*}
Now, we must find an approximation of
$\frac{1}{n}\mathrm{tr}\left(\widetilde{\HH}^2\HH^+\right)$. For $b
\approx n$ we have $\widetilde{\HH}^2\HH^+ \approx \HH$ whereas for $b
\approx 1$ we argue that $\widetilde{\HH}$ and $\HH^+$ can be seen as
independent matrices with $\HH^+ \approx \II$. We can linearly
interpolate between these two extremes, 
\begin{align}
    \EE\left[\|\widetilde{\nabla}f(\xx^*)\|^2\right] &\approx \frac{1}{b}\mathrm{tr}\left(\widetilde{\HH}\right) - \frac{b}{n}\frac{1}{n}\mathrm{tr}\left(\HH\right) - \left(1 - \frac{b}{n}\right)\frac{1}{n}\mathrm{tr}\left(\widetilde{\HH}^2\right) \\
    &= r' - \frac{b}{n}r - \left(1-\frac{b}{n}\right)r(1 + r') = (1 - r)(r' - r)\,.
\end{align}
Experimentally these
approximations work well. Hence in our simulations we set $M = \widetilde{R}^2(1 - r)(r' - r)$.

\bibliographystyle{plainnat}
\bibliography{biblio}

\appendix

\end{document}